\documentclass[11pt]{amsart}

\usepackage{amsmath}
\usepackage{amsfonts}
\usepackage{amssymb}
\usepackage{ytableau}
\usepackage{amsthm}
\usepackage{tikz}
\usetikzlibrary{tikzmark,decorations.pathreplacing}
\usepackage{hyperref}
\usepackage{color}
\usepackage{graphicx}
\usepackage{tikz-cd}
\usepackage[numbers]{natbib}
\usepackage{appendix}

\usepackage{mathtools}
\usepackage{enumerate}
\usepackage{multicol}
\usepackage{seqsplit}

\newtheorem{theorem}{Theorem}[section]

\newtheorem{lem}[theorem]{Lemma}
\newtheorem{cor}[theorem]{Corollary}

\newtheorem{prop}[theorem]{Proposition}
\newtheorem{prob}[theorem]{Problem}
\newtheorem{rem}[theorem]{Remark}
\newtheorem{exm}[theorem]{Example}
\newtheorem{example}[theorem]{Example}

\newtheorem{corollary}[theorem]{Corollary}
\newtheorem{definition}[theorem]{Definition}

\newcommand{\bb}{\mathbb}
\newcommand{\ca}{\mathcal}
\newcommand{\fr}{\mathfrak}

\newcommand{\Z}{\ensuremath{\mathbb{Z}}}

\newcommand{\C}{\ensuremath{\mathbb{C}}}

\renewcommand{\P}{\ensuremath{\mathrm{P}}}
\newcommand{\bP}{\ensuremath{\mathbb{P}}}

\renewcommand{\O}{\ensuremath{\mathcal{O}}}
\newcommand{\rO}{\ensuremath{\mathrm{O}}}
\newcommand{\cS}{\ensuremath{\mathcal{S}}}

\newcommand{\bS}{\ensuremath{\mathbb{S}}}
\newcommand{\bG}{\ensuremath{\mathbb{G}}}
\newcommand{\fX}{\ensuremath{\fr{X}}}

\newcommand{\rV}{\mathrm{V}}
\newcommand{\cN}{\ca{N}}

\newcommand{\vd}{\vec{d}}
\newcommand{\ra}{\rightarrow}
\newcommand{\rS}{\mathrm{S}}

\newcommand{\rH}{\mathrm{H}}

\newcommand{\tikznode}[3][inner sep=0pt]{\tikz[remember
	picture,baseline=(#2.base)]{\node(#2)[#1]{$#3$};}}

\DeclareMathOperator{\Sym}{\mathrm{Sym}}
\DeclareMathOperator{\nbun}{\mathcal{N}}
\newcommand{\ch}{\mathrm{ch}^T}
\newcommand{\td}{\mathrm{td}^T}
\newcommand{\ev}{\mathrm{ev}}
\newcommand{\wev}{\widetilde{\ev}}

\newcommand{\vir}{\mathrm{vir}}

\newcommand{\ovir}{\ca{O}^{\mathrm{vir}}}
\newcommand{\Hom}{\operatorname{\ca{H}\! \mathit{om}}}
\newcommand{\GLr}{\mathrm{GL}_r(\bb{C})}

\newcommand{\kg}{K_{T}^0}
\newcommand{\kG}{K_0^{T}}
\newcommand{\spec}{\operatorname{Spec}}
\newcommand{\qk}{\mathrm{QK}}

\newcommand{\prk}{\P_{r,N-r}}
\newcommand{\prkk}{\P_{r,k}}

\makeatletter
\@namedef{subjclassname@2020}{%
	\textup{2020} Mathematics Subject Classification}
\makeatother

\DeclareMathOperator{\cQ}{\mathcal{Q}}
\DeclareMathOperator{\cO}{\mathcal{O}}
\DeclareMathOperator{\quot}{\mathsf{Quot}}
\newcommand{\Quot}{\quot_d(\mathbb{P}^1,N,r)}
\newcommand{\QuotC}{\quot_{d}(C,N,r)}

\DeclareMathOperator{\fix}{\mathrm{F}}

\newcommand{\grass}{\mathrm{Gr}(r,N)}
\newcommand{\grtwo}{\mathrm{Gr}(2,N)}
\newcommand{\ms}{\overline{M}_{0,n}(X,d)}

\begin{document}
	\title[Quantum $K$-invariants via Quot schemes II]{
		Quantum $K$-invariants via Quot schemes II
	}
	
	\author[S.~Sinha]{SHUBHAM SINHA}
	\address{The Abdus Salam International Centre for Theoretical Physics, Strada Costiera 11, 34151 Trieste}
	\email{ssinha1@ictp.it}
	\author[M.~Zhang]{Ming Zhang}
	\address{Fair Isaac Corporation, San Diego}
	\email{mingzhang@fico.com}
	
	
	\maketitle
	
	\begin{abstract}		
		We derive a $K$-theoretic analogue of the Vafa--Intriligator formula, computing the (virtual) Euler characteristics of vector bundles over the Quot scheme
		that compactifies the space of degree $d$ morphisms from a fixed projective curve to the Grassmannian $\grass$.
		As an application, we deduce interesting vanishing results, used in our previous work to study the quantum $K$-ring of $\grass$.
		In the genus-zero case, we prove a simplified formula involving Schur functions, consistent with the Borel--Weil--Bott theorem in the degree-zero setting. These new formulas offer a novel approach for computing the structure constants of quantum $K$-products.

	\end{abstract}

	\section{Introduction}
	Let $\grass$ denote the Grassmannian of $r$-dimensional subspaces of $\mathbb{C}^N$, and let $C$ be a smooth curve of genus $g$. Grothendieck's Quot scheme on $C$ compactifies the space of morphisms from $C$ to $\grass$. Str\o mme~\cite{Stromme} initiated the study of the Quot scheme for $C = \bP^1$, later extended by Bertram~\cite{Bertram2} to explore the quantum cohomology of the Grassmannian. The Quot scheme on a genus $g$ curve has also been used to compute fixed-domain Gromov--Witten counts~\cite{Bertram3, Bertram, Marian-Oprea}, with extensions to type A flag varieties and isotropic Grassmannians~\cite{CF1, CF2, CF3, Chen1, KreschTamvakis, Sinha}.
	
	In our previous work~\cite{SinhaZhang}, we showed that the quantum $K$-ring and related $K$-theoretic Gromov--Witten invariants of Grassmannians can be studied through the Quot scheme. We also demonstrated that the topological quantum field theory (TQFT) associated with the quantum $K$-ring is naturally expressed via certain (virtual) Euler characteristics of sheaves over Quot schemes.
	
	In this paper, we present formulas for these Euler characteristics over the Quot scheme,
	giving a $K$-theoretic analogue of the Vafa--Intriligator formulas.
	In the genus-zero case, we derive a bialternant formula for the relevant invariants.
	As an application, we prove vanishing results used in Part I~\cite{SinhaZhang} to study the quantum $K$-ring of Grassmannians,
	including proving the finiteness of the quantum product and deriving ring presentations.
	In the final section, we introduce an algorithm to compute the quantum $K$-product,
	yielding a quantum $K$-Littlewood--Richardson rule for $\grtwo$, with a multiplication table for $\qk(\mathrm{Gr}(3,6))$ provided in the appendix.
	\subsection{Vafa--Intriligator formula}
	Let $C$ be a smooth projective curve of genus $g$. The Quot scheme $\QuotC$
	parametrizes short exact sequences of coherent sheaves over $C$:
	$$
	0\to S\to \O_C^{\oplus N}\to Q\to 0,
	$$
	where $Q$ has rank $N-r$ and degree $d$. The intersection theory of the Quot scheme was studied by Bertram--Daskalopoulos--Wentworth~\cite{Bertram},
	Bertram~\cite{Bertram3}, and Marian--Oprea~\cite{Marian-Oprea} to
	provide a framework for studying fixed-domain Gromov--Witten counts of maps to Grassmannians.
	We will state the results as formulated in~\cite{Marian-Oprea}.
	The Quot scheme carries a universal short exact sequence on $\QuotC\times C$:
	\begin{equation*}
		0\rightarrow
		\ca{S}
		\rightarrow
		\O_{\QuotC\times C}^{\oplus N}
		\rightarrow
		\ca{Q}
		\rightarrow 
		0.	
	\end{equation*}
	We fix a point $p\in C$ and denote by $\cS_p$ the restriction of $\ca{S}$ to $\QuotC\times\{p\}$.

	When $d$ is large, the Quot scheme $\QuotC$ is irreducible, see \cite{Bertram,PopaRoth}, and we can evaluate the top
	intersection numbers of the Chern classes of $\cS_p^\vee$ against the fundamental class. For arbitrary degrees, the Quot scheme may fail to be smooth or irreducible. Nonetheless, it was shown in~\cite{Marian-Oprea,C-F-Kapranov} that $\QuotC$ admits a two-term perfect obstruction theory and consequently a virtual fundamental class, see~\cite{Behrend-Fantechi,Li-Tian} for the general construction,
		\[
		[\QuotC]^{\vir}\in H_{2e}(\QuotC),
		\]
		where $e=Nd-r(N-r)(g-1)$ is the virtual dimension of the Quot scheme. This demonstrates that intersection theory can still be pursued virtually by pairing Chern classes with the virtual fundamental class.
	These intersection numbers are given by the remarkable \emph{Vafa--Intriligator formula}~\cite{Vafa,Intriligator}.
	Let $x_1, \dots, x_r$ denote the Chern roots of $\cS^\vee_p$.
	Note that polynomials in the Chern classes of $\cS^\vee_p$ are symmetric in the $x_i$'s.
	\begin{theorem}[\cite{Bertram3,Siebert-Tian,Marian-Oprea}]\label{thm:VI-formula_MO}
		We have 
		\[
		\int_{
			[\QuotC]^{\vir}} f(x_1,\dots,x_r)
		=
		(-1)^{d(r-1)}
		\sum_{\xi_1,\dots,\xi_r}
		f(\xi_1,\dots,\xi_r)J^{g-1}(\xi_1,\dots,\xi_r),	
		\]
		where
		\begin{itemize}
			\item $f$ is a symmetric polynomial of degree $e$,
			\item 
			the sum is over all $\binom{N}{r}$ tuples $(\xi_1,\dots,\xi_r)$
			of distinct $N$th roots of unity, and
			\item $J(z_1,\dots,z_r)$ is the symmetric function
			\[
			J(z_1,\dots,z_r)=N^r\cdot z_1^{-1}\cdots z_r^{-1}
			\cdot
			\prod_{ i\ne j}	(z_i-z_j)^{-1}.
			\]
		\end{itemize}
	\end{theorem}
	\begin{rem}\label{rem:Grass_GIT_Coh}
		The Grassmannian admits the GIT presentation
		$\grass=[M_{r\times N}^{ss}/\GLr],$
		where $M_{r\times N}$ is the space of $r\times N$ complex matrices and $M_{r\times N}^{ss}\subset M_{r\times N}$ is the open subset of matrices of full rank.  Let
		$$\fr{X}=[M_{r\times N}/\GLr]$$
		be the quotient stack containing $\grass$. The Quot scheme admits a `stacky' evaluation morphism 
		\begin{equation}\label{eq:evaluation_stacky}
			\wev_{p}:\QuotC\to \fr{X}
		\end{equation}
		by restricting a subsheaf $[S\to \ca{O}_{C}^{\oplus N}]$ to its stalk at the point $p$,  $S_p\to \ca{O}_{p}^{\oplus N}$.
		The cohomology classes considered in Theorem~\ref{thm:VI-formula_MO} are the pullbacks via $\wev_{p}$ of classes from $H^*(\fr{X},\bb{Z}) $, which is the ring of symmetric polynomials in $r$ variables.
	\end{rem}
	In the physics literature, Witten~\cite{Witten} observed a connection
	between Verlinde numbers and the quantum cohomology of Grassmannians.
	Verlinde numbers represent the dimensions of the space of global sections of
	theta bundles over the moduli space of vector bundles on curves.
	In~\cite{marian2,MarianOprea2},
	the authors showed that Verlinde numbers can be expressed via intersections on Quot schemes, providing another form of the Vafa–-Intriligator formula.

	\subsection{Euler characteristics over Quot schemes}
	Previous work on explicit formulas for Euler characteristics of \( K \)-theory classes (tautological bundles) on Quot schemes over curves was initiated in \cite{OpreaPandharipande} and \cite{OpreaShubham}
	for zero-rank quotients. The latter also derived some formulas for higher rank quotients over \(\mathbb{P}^1\).
	These formulas, for Quot schemes of zero-rank quotients on curves, utilize universality results in the sense of~\cite{EllingsrudGottscheLehn} and torus localization when $C=\bb{P}^1$ (see~\cite{OpreaShubham} for details). 
	
	For $\QuotC$ parameterizing higher rank quotients, i.e., $r<N$,
	a universality result is not expected in the same sense.
	In Part I~\cite{SinhaZhang}, we proved that the virtual Euler characteristics
	of certain $K$-theory classes over Quot schemes
	satisfy the topological quantum field theory (TQFT) axioms, allowing induction on the genus of $C$.
	In this paper, we use torus localization for $C=\bb{P}^1$ to derive a novel formula,
	serving as a $K$-theoretic analogue of the Vafa--Intriligator formula. 
	
	Recall the definition of $\fr{X}$ from Remark~\ref{rem:Grass_GIT_Coh}.
	The $K$-theory $K(\fr{X})$ is isomorphic to the representation ring $R(\GLr)$ of $\GLr$.
	For a $\GLr$-representation $\mathrm{V}$, this isomorphism maps it to the associated $K$-theory
	\begin{equation}
		\label
		{eq:associated-bundle-intro}
		\begin{aligned}
			\mathrm{V}\to M_{r\times N}\times_{\GLr} \mathrm{V}
			\in K(\fr{X}).
		\end{aligned}
	\end{equation}
	By abuse of notation, we also denote this $K$-theory class as $\mathrm{V}$.
	
	Let $\ovir_{\QuotC}$ denote the virtual structure sheaf induced by the two-term perfect obstruction theory for $\QuotC$ (see~\cite{Lee}).
	The virtual Euler characteristic of a coherent sheaf $F$ on $\QuotC$ is defined as
	\[
	\chi^{\vir}\big(\QuotC,F\big)
	:=
	\chi\big(\QuotC,\ovir_{\QuotC}\otimes F\big).
	\]
	Our first main result is the following Vafa--Intriligator type formula.
	\begin{theorem}\label{thm:Vafa_Intr_Ktheory_intro}
		For any $\GLr$-representation $\mathrm{V}$ and a point $p\in C$, we have 
		\[
		\chi^{\vir}\big(\QuotC,\wev_{p}^* (\mathrm{V})\big)
		=
		[t^d]
		\sum_{z_1,\dots,z_r}
		(v\cdot h^g)(z_1,\dots,z_r)\frac{\prod_{i=1}^{r}z_i^{N-r+d}\cdot\prod_{i\ne j}(z_i-z_j)}{\prod_{i=1}^{r}P'(z_i)},
		\]
		where 
		\begin{itemize}
			\item 
			the sum is over all $\binom{N}{r}$ sets of $r$ distinct roots $z_1,z_2,\dots,z_r$ of
			\begin{equation}
				\label{eq:intro-non-equivariant-bethe-equation}
				P(z)=(z-1)^N +(-1)^r z^{N-r}t=0,
			\end{equation}
			\item $v$ is the character of $\mathrm{V}$,
			\item $h$ is the character of a $\GLr$-representation that only depends on $N$ (see Section~\ref{subsec:VI-in-all-genera}),
			\item the operation $[t^{d}]$ on a symmetric Laurent polynomial in the $N$ roots of~\eqref{eq:intro-non-equivariant-bethe-equation} involves three steps:
			express the Laurent polynomial in terms of $t$ using~\eqref{eq:intro-non-equivariant-bethe-equation}, expand it as a power series around $t=0$,
			and extract the coefficient of $t^{d}$.
		\end{itemize} 
	\end{theorem}

	Recall that $\QuotC$ is irreducible for sufficiently large $d$, and its virtual structure sheaf becomes the usual structure sheaf.
	A degree estimate from Theorem~\ref{thm:Vafa_Intr_Ktheory_intro} leads to an interesting vanishing result
	(see Corollary~\ref{cor:vanishin_higher_genus} for a stronger version):
	\[
	\chi(\QuotC, \left(\det \cS_{p}^{\vee}\right)^{m}) = 0
	\quad \text{and} \quad
	\chi(\QuotC, \Sym^m \left(\cS_{p}^{\vee}\right))=0
	\]
	for all $m < g$ and sufficiently large $d$ (relative to $g$). 
	\begin{rem}	
	When the genus of the curve $g\ge 1$, setting $\lambda=\emptyset$ in Corollary~\ref{cor:vanishin_higher_genus} implies 
			\[
			\chi(\QuotC, \cO_{\quot_{ d}}) = 0\qquad \text{when $d\gg 0$ with respect to $g$, $r$ and $n$}.
			\]
			Note that this contrasts with the genus zero case, when $\quot_{d}(\bb{P}^1, N,r)$ is a smooth projective rational variety with $\chi(\quot_{d}(\bb{P}^1, N,r), \cO_{\quot_{ d}}) = 1$. The above phenomenon is clear in the case $r=1$, when
			the Quot scheme $\QuotC$ is a
			projective bundle
			\[
			\QuotC\cong \bb{P}(\pi_*\ca{L}^{\oplus N})\to \mathrm{Pic}^d(C),
			\]
			where $d\ge2g-1$, $\ca{L}$ is a universal bundle on $\mathrm{Pic}^d(C)\times C$ and $\pi$ is the first projection.
		
		In general, there is a map $\phi :\QuotC\to \mathrm{Pic}^d(C)$ given by taking the determinant of the quotient. The vanishing result will follow if the following question has an affirmative answer: Does the derived pushforward $$R\phi_*(\mathcal{O}_{\quot})=\mathcal{O}_{\mathrm{Pic}^d(C)} \quad \text{ on $\mathrm{Pic}^d(C)$ when $d\gg0$?}$$ In~\cite{GangopadhyaySabastian}, the authors study certain properties of the fibers of this determinant map.
	\end{rem}

	\subsection{Schur bundles}
	A partition $\lambda=(\lambda_1,\dots,\lambda_r)$ is a sequence of integers such that $\lambda_1\geq \lambda_2\geq\cdots\geq
	\lambda_r\geq0$.
	Let $\prk$ be the set of partitions $\lambda$ with $r$ parts
	where $\lambda_1\leq N-r$.
	Let $\bS^\lambda$ be the Schur functor associated with the partition $\lambda$,
	and let $\mathrm{S}:=\left(\bb{C}^r\right)^{\vee}$ denote the dual of the standard $\GLr$ representation.
	Then, $\bS^\lambda\left( \mathrm{S}^\vee \right)$ is the irreducible representation with highest weight $\lambda$.
	
	The representation $\mathrm{S}$ corresponds to the universal subbundle $S$ on $\grass$,
	as seen from~\eqref{eq:associated-bundle-intro} and restricting to the semistable locus $\grass=[M_{r\times N}^{ss}/\text{GL}_r(\bb{C})]$.
	Schur bundles $\bS^\lambda(S^\vee)$ and $\bS^\lambda(S)$ on $\grass$ are defined similarly for any partition $\lambda$.
	Their cohomology groups can be described by the Borel--Weil--Bott theorem. In particular, we recall the following result of Kapranov which was used to study the derived category of Grassmannian:
	
	\begin{prop}[\cite{Kapranov2}]
		For any nonempty partition $\lambda\in \P_{r,N-r}$,
		we have
		\begin{enumerate}[\normalfont(i)]
			\item $
			H^i(\grass,\bS^{\lambda}(S))=0
			$ for all $i\geq 0$.
			\item \[
			H^i(\grass,\bS^{\lambda}(S^\vee))
			=\begin{cases}
				\bS^\lambda((\C^N)^\vee)
				& \quad\text{if}\ i=0,\\
				0
				& \quad\text{if}\ i>0.
			\end{cases}
			\]
		\end{enumerate}
		The above equalities hold as $\mathrm{GL}_N(\bb{C})$-representations.
	\end{prop}
	This implies that for any non-empty partition  $\lambda\in\P_{r,N-r}$:
	\begin{equation}
		\label{eq:classic-S}
		\chi(\grass,\bS^{\lambda}(S))=0\quad \text{and}\quad
		\chi(\grass,\bS^{\lambda}(S^\vee))=s_\lambda(1,\dots,1).
	\end{equation}
	In this paper, we consider ``quantum" versions of~\eqref{eq:classic-S}. For the rest of the introduction, we set $C=\mathbb{P}^1$ and work with $\quot_{ d}(\bP^1,N,r)$.
	Note that the representation $\mathrm{S}$, as a $K$-theory class on $\fX$,
	pulls back to the vector bundle $\cS_p$ over $\quot_{d}$ under the evaluation map $\wev_p$ in \eqref{eq:evaluation_stacky}. We will consider
	the Euler characteristics of the following Schur bundles over $\quot_{ d}(\bP^1,N,r)$:
	\[\bS^{\lambda}
	\left(
	\cS_{p}
	\right)=\wev_{p}^*\left(\bS^{\lambda}
	\left(
	\mathrm{S}
	\right)\right)
	\quad\text{and} \quad \bS^{\lambda}
	\left(
	\cS_{p}^\vee
	\right)=\wev_{p}^*\left(\bS^{\lambda}
	\left(
	\mathrm{S}^\vee
	\right)\right).
	\]
	\begin{theorem}
		\label
		{prop:intro-vanishing}
		Let $\lambda=(\lambda_{1},\dots,\lambda_{r})$ be a non-empty partition.
		We have 
		\begin{align*}
			\chi(\quot_{ d}(\bP^1,N,r), \bS^\lambda(\cS_p))=0,
		\end{align*} 
		if one of the following three conditions holds
		\begin{enumerate}[\normalfont(i)]
			\item $\lambda_{1}\le d+(N-r)$;
			\item $\lambda_{1}\le d+2(N-r)$ with additional conditions $r<N$, $d>0$ and $\lambda_r>0$;
			\item $\lambda_{1}\le d+2(N-r)$ with additional conditions $r<N$ and $d\ge r$.
		\end{enumerate}
	\end{theorem}

	\begin{rem}
		In Section 3 and 4 of Part I~\cite{SinhaZhang},
		we showed that the vanishing results in Theorem~\ref{prop:intro-vanishing} recover
		important properties such as finiteness and ring presentation for the quantum $K$-ring of $\grass$,
		as obtained in~\cite{Buch-Mihalcea} and~\cite{GMSXZZ}.
		We expect similar vanishing results to apply more broadly
		to genus-zero quasimap spaces for flag varieties and symplectic/orthogonal Grassmannians,
		offering a new approach to studying the quantum $K$-theory of these varieties.
	\end{rem}
	\begin{rem}
		Consider the action of $T=(\C^*)^N$ on $\quot_d(\bP^1,N,r)$ defined via the diagonal action on $\cO_{\bP^1}^{\oplus N}$. In the main body, we prove the formulas for Euler characteristics in Theorem~\ref{prop:intro-vanishing} and \ref{thm:intro-2} in the torus-equivariant setting
		(see Theorems~\ref{prop:Vanshing_rparts} and \ref{thm:Localization_calculation}).
	\end{rem}
	\begin{theorem}
		\label
		{thm:intro-2}
		Let
		$\lambda=(\lambda_{1},\lambda_{2},\dots,\lambda_r)$
		be a partition, with no extra condition on $\lambda_1$.
		The following formula holds:
		\begin{align*}
			\chi(\quot_d(\bP^1,N,r), 
			\bS^{\lambda}(
			\cS_{p}^{\vee}
			)
			)
			&=
			[t^d]s_{\lambda+(d)^r}(z_1,\dots,z_N).
		\end{align*}
		Here $z_1,z_2,\dots,z_N$ are the distinct roots of the equation~\eqref{eq:intro-non-equivariant-bethe-equation},
		and $\lambda+(d)^r$ is the partition obtained by adding $d$ to each part, i.e.,
		$\lambda+(d)^r=(\lambda_1+d,\dots,\lambda_r+d)$.
	\end{theorem}
	Schur polynomials can be expressed as determinants of elementary symmetric polynomials using the classical Jacobi--Trudi formula. The elementary symmetric polynomial evaluated at the distinct roots $z_1,\dots,z_N$ of \eqref{eq:intro-non-equivariant-bethe-equation} has a very simple description:
	\begin{equation*}
		e_{j}(z_1,\dots,z_N)=\begin{cases}
			\binom{N}{j}
			& \quad \text{if}\ j\ne r,
			\\[5pt]	
			\binom{N}{r}
			+t& \quad \text{if}\ j= r.
		\end{cases}
	\end{equation*}
	As a direct application of Theorem~\ref{thm:intro-2},
	we obtain closed-form expressions
	for the Euler characteristics of
	the exterior and symmetric powers of $\cS_p^\vee$ over the Quot scheme.
	For torus-equivariant versions, see
	Corollary~\ref{cor:wedge-powers} and Corollary~\ref{cor:symmetric-power}.
	\begin{cor}
		We have
		\[
		\sum_{d=0}^{\infty}
		q^d
		\chi(\quot_d(\bP^1,N,r), \wedge^m(\cS_p^\vee))
		=\begin{cases}
			\frac{1}{1-q}
			\binom{N}{m}
			& \text{if } m\ne r\\[5pt]
			\frac{1}{(1-q)^2}
			\binom{N}{r}
			& \text{if }\ m=r
		\end{cases}
		\]
		and
		\begin{align*}
			\sum_{d=0}^{\infty}q^d
			\chi
			(\quot_d(\bP^1,N,r), \Sym^m(\cS_p^\vee))
			=\begin{cases}
				\frac{1}{1-q}
				\binom{N+m-1}{m}
				& \text{if } m< r\\[5pt]
				\frac{1}{1-q}
				\binom{N+r-1}{r}
				+
				\frac{(-1)^{r+1}q}{(1-q)^2}
				\binom{N}{r}
				& \text{if } m=r
			\end{cases}.
		\end{align*}
	\end{cor}
	
	\begin{prob}\label{prob:vanishing}
		The formulas in Theorems~\ref{prop:intro-vanishing} and~\ref{thm:intro-2}
		match~\eqref{eq:classic-S} when $d=0$.
		We aim to characterize conditions under which the higher cohomology groups in
		Theorem~\ref{thm:intro-2} and \ref{prop:intro-vanishing} vanish.
		Inspired by the Borel--Weil--Bott Theorem for degree zero and
		examples in higher degree cases, we conjecture
		that $$H^i(\quot_{ d}(\bP^1,N,r), \bS^\lambda(\cS_p))=0$$ for all $i\ge 0$ and partitions $\lambda$ with $\lambda_1\le d+N-r$.
	\end{prob}
	The formulas for computing individual sheaf cohomology groups of the tautological bundles on Quot schemes on $\bb{P}^1$ were obtained recently in \cite{Gautam_Lin_Sinha}, and \cite{Marian-Oprea-Sam}. In particular, \cite[Theorem~1.11]{Gautam_Lin_Sinha} partially solves Problem~\ref{prob:vanishing}.
	
	Let $\pi$ and $\pi_{\bP^1}$ be the first and second projection from $\quot_{ d}(\bP^1,N,r)\times\bP^1$ to $\quot_{ d}(\bP^1,N,r)$ and $\bP^1$,
	respectively, and let $\pi_{\star}$ denote the proper pushforward in
	$K$-theory.
	For any integer $\ell$, the determinant line bundle $\det(\pi_{\star}\cS^\vee)^{-\ell}$ is referred to as
	a level structure in~\cite{RZ1,RZ2}.
	We also derive bialternant-type formulas for
	$$\chi^{T}(\quot_d(\bP^1,N,r), 
	\det(\pi_{\star}\cS^\vee)^{-\ell}
	\cdot
	\bS^{\lambda}(
	\cS_{p}^{\vee}
	)
	)$$
	in Theorem~\ref{thm:Localization_calculation} and Corollary~\ref{cor:EulerCharShurBundle}.
	This generalization is useful for computing the Euler characteristics of the determinants
	of tautological bundles
	as considered in~\cite{OpreaShubham}. Let $M$ be a line bundle on $\bP^{1}$ and
	define $M^{[d]}:=
	\pi_{\star} \left(\pi_{\bP^1}^{*} M\otimes \mathcal Q\right)$.
	\begin{cor}
		If the level $\ell$ satisfies $-r<\ell\le (N-r)$, then
		\[
		\chi
		(\quot_{ d}(\bP^1,N,r),(\det M^{[d]})^\ell )= [t^d]s_{
			\left((m+1)\ell+d\right)^r
		}(z_1,\dots,z_N),
		\]
		where
		$z_1,z_2,\dots,z_N$ are the roots of~\eqref{eq:intro-non-equivariant-bethe-equation}.
	\end{cor}

	\subsection{Application to quantum $K$-theory}
	The $K$-group of $X=\grass$ has a basis $\{\O_{\lambda}\}_{\lambda\in \prk}$,
	where $\O_{\lambda}:=[\O_{X_{\lambda}}]\in K(X)$
	represents the $K$-theory class of the structure sheaf of the Schubert variety $X_{\lambda}\subset X$.
	These $\O_\lambda$ are referred to as \emph{Schubert structure sheaves}.
	As shown in~\cite[\textsection{8}]{Lenart}, for any partition $\lambda$, there is a polynomial functor $\bG^{\lambda}$, defined via \emph{stable Grothendieck polynomials}, such that for $\lambda \in \prk$, we have
	\begin{equation*}
		\O_{\lambda}
		=
		\bG^{\lambda}(S^{\vee})
	\end{equation*}
	in $K(X)$.
	We lift this to the level of representation and define
	$$\rO_{\lambda}:=\bG^{\lambda}(\C^r)\in R(\GLr),$$
	for any partition $\lambda$ with $r$ parts.
	Here $\C^r$ denotes the standard representation of $\GLr$.
	Note that $\rO_\lambda$ is nonzero even for $\lambda \notin \prk$,
	and when $\lambda \in \prk$,
	the Schubert structure sheaf $\O_\lambda$ corresponds to the $K$-theory class associated with $\rO_\lambda$ on $X$.
	
	Let $\wev_p$ be the stacky evaluation morphism in~\eqref{eq:evaluation_stacky}.
	We define the 1-pointed Quot
	scheme invariant of $\rO_{\lambda}$ by
	\[
	\langle
	\rO_\lambda
	\rangle^{\quot}_{0,d}
	:=
	\chi
	\left(
	\Quot,
	\wev_{p}^*(\rO_\lambda)
	\right).
	\]
	Note that 
	\[
	\wev_{p}^*(\rO_\lambda )
	= 	\bG^{\lambda}
	(
	\cS_{p}^{\vee}
	).
	\]
	As a result of Theorem~\ref{thm:Vafa_Intr_Ktheory_intro},
	we obtain the following bialternant-type formula of the 1-pointed Quot
	scheme invariant of $\rO_{\lambda}$ (see Theorem~\ref{cor:G-lambda}):

	\begin{cor}
		\label
		{cor:intro-stacky-schubert-structure-sheaf}
		For any partition $\lambda=(\lambda_{1},\dots,\lambda_{r})$ with $r$ (nonnegative) parts,
		we have
		\[
		\chi
		\big(
		\Quot,
		\bG^{\lambda}
		(
		\cS_{p}^{\vee}
		)
		\big)=[t^d]s_{(f_1,\dots, f_r)} (z_1,\dots,z_N)
		\]
		where
		$
		f_i(z)=
		z^d(1-1/z)^{\lambda_i+r-i},$
		$s_{(f_1,\dots,f_r)}$ is the generalized Schur function~\eqref{eq:generalized-schur-function},
		and $z_1,z_2,\dots,z_N$ are the distinct roots of \eqref{eq:intro-non-equivariant-bethe-equation}.
		
	\end{cor}
	
	Let $q$ be a formal parameter representing the Novikov variable. Givental \cite{WDVV} and Lee~\cite{Lee}
	introduced the \emph{quantum product} on the $\Z[[q]]$-module $\qk(X):=K(X)\otimes_{\Z}\Z[[q]]$, a deformation of the tensor product, defined using generating functions of 2-pointed and 3-pointed genus-0 quantum $K$-invariants of $X$
	(see Section~\ref{subsec:quantum-k-ring}).
	The \emph{quantum $K$-ring} $\qk(X)$ is the $\Z[[q]]$-algebra equipped with this new quantum product.
	Buch and Mihalcea~\cite{Buch-Mihalcea} were the first to study the quantum $K$-ring of the Grassmannian.
	
	In Section~\ref{subsec:quantum-reduction-map}, we propose a strategy to compute the quantum product of Schubert structure sheaves
	using the Littlewood--Richardson rule for Grothendieck polynomials~\cite{Buch}
	and finitely many 1-pointed Quot scheme invariants obtained via Corollary~\ref{cor:intro-stacky-schubert-structure-sheaf}.
	To illustrate this strategy, we provide a complete description of the quantum product in the rank 2 case.	
	\begin{theorem}\label{thm:intro_G(2,N)}
		Let $\lambda, \mu\in \P_{2,N-2}$ such that $\mu_{1}-\mu_{2}\le \lambda_{1}-\lambda_2$.
		Then the product in the quantum $K$-ring $\qk(\mathrm{Gr}(2,N))$ is given by
		\begin{equation*}
			\cO_{\lambda}\bullet \cO_{\mu}=\ca{G}_{\lambda,\mu}+ \ca{G}_{\tilde{\lambda},\mu}q+\ca{G}_{\lambda-(N,N),\mu}q^2.
		\end{equation*}
		Here $\tilde{\lambda} = (\lambda_2-1,\lambda_1-N+1)$
		and $\ca{G}_{\alpha,\beta}$ is an element in $K(\mathrm{Gr}(2,N))$ defined by
		\[ 
		\ca{G}_{\alpha,\beta}
		:= \cO_{\alpha+\beta}+
		\sum_{i=1}^{\min\{\beta_1-\beta_2,\alpha_{1}-\alpha_{2}\}}
		\Big[\cO_{\alpha+\beta+(-i,i)}-
		\cO_{\alpha+\beta+(-i+1,i)}\Big],
		\]
		where we set $\cO_\nu =0$ if $\nu\notin \P_{2,N-2}$.
		
	\end{theorem}
	\begin{rem}
		We refer to the above formula as the \emph{quantum $K$-Littlewood--Richardson rule} for $\mathrm{Gr}(2,N)$. In particular, Theorem~\ref{thm:intro_G(2,N)} immediately implies the positivity conjectured in \cite{Buch-Mihalcea} for $\mathrm{Gr}(2,N)$: Let $F_{\lambda,\mu}^{\nu,d}$ be the structure constants for the quantum $K$-product, defined by
			\[
			\mathcal{O}_{\lambda}\bullet \mathcal{O}_{\mu}
			=
			\sum_{\nu\in \mathcal{P}_{2,N-2}}\sum_{d=0}^{2}
			F_{\lambda,\mu}^{\nu,d}\ q^d\mathcal{O}_{\nu}.
			\]
			Then
			$
			(-1)^{|\nu|-|\lambda|-|\mu|+dN}F_{\lambda,\mu}^{\nu,d}\geq 0
			$
			for all $\lambda,\mu,\nu\in \mathcal{P}_{2,N-2}$ and all $0\leq d\leq 2$.
	\end{rem}

	As noted in~\cite[Remark 5.2]{Buch-Mihalcea},
	the quantum cohomology ring $\mathrm{QH}(X)$ of
	the Grassmannian is an associated graded ring of $\qk(X)$.
	Consequently, the quantum Littlewood--Richardson rule in $K$-theory,
	as stated in the above theorem,
	yields a formula for the product of two Schubert cycles in quantum cohomology.
	We refer the reader to Corollary~\ref{cor:cohomological-quantum-LR-rule} for the precise formula.
	
	Our approach enables numerical computation of
	the quantum $K$-products for any Grassmannian. In the appendix, we present the complete multiplication tables for $\qk(\mathrm{Gr}(3,6))$.

	\subsection{Future directions}
	Motivated by the quantum $K$-invariants of the Grassmannian,
	we examine the Euler characteristics of certain $K$-theory
	classes on the Quot scheme,
	using localization formulas to derive Vafa--Intriligator and
	bialternant-type formulas for these invariants. 
	The Quot scheme serves as the moduli space of stable quasimaps from
	a fixed curve to the quotient stack $\fX$ containing the
	Grassmannian and extends to more general GIT quotients,
	such as isotropic Grassmannians and partial flag manifolds.
	We aim to generalize our Vafa--Intriligator formula (Theorem~\ref{thm:Vafa_Intr_Ktheory_intro}) and
	bialternant formula (Theorem~\ref{thm:intro-2}) to more general quasimap spaces,
	and to explore the vanishing of higher cohomology in the definition of our $K$-theoretic invariants (Theorem~\ref{prop:intro-vanishing}).
	Additionally, we hope to extend the strategy in Section~\ref{subsec:quantum-reduction-map} to compute
	the quantum $K$-rings of all Grassmannians and other GIT quotients.

	\subsection{Plan of the paper}
	In Section~\ref{sec:K-theory}, we review basic $K$-theory notation, symmetric functions, and the $K$-group of the Grassmannian. 
	In Section~\ref{sec:localization}, we define $K$-theoretic Quot scheme invariants and use localization and genus-induction formulas to prove the Vafa--Intriligator formula for these invariants.
	In Section~\ref{sec:schur-bundle}, we focus on the genus zero case,
	deriving a bialternant formula for 1-pointed $K$-theoretic Quot scheme invariants of Schur bundles and proving vanishing results relevant to applications to quantum $K$-theory.
	In Section~\ref{sec:applications}, we propose a strategy to compute the quantum $K$-theory of the Grassmannian using these invariants, implementing it in rank 2 to derive a quantum $K$-Littlewood--Richardson rule for $\mathrm{Gr}(2,N)$.
	In Appendix~\ref{subsec:examples}, we provide the complete multiplication tables for the
	quantum $K$-rings of $\mathrm{Gr}(3,6)$.

	\subsection{Acknowledgment}
	
	Both authors are grateful to Leonardo C. Mihalcea for sharing insights on quantum $K$-theory of the Grassmannian. This paper, together with the adjoining~\cite{SinhaZhang},
	was inspired by our efforts to verify the $K$-theoretic divisor equation~\cite[Conjecture 4.3]{GMSXZZ} via Quot schemes.
	The first author thanks Alina Marian and Dragos Oprea for valuable conversations about the Quot schemes of curves.
	The second author would like to thank Wei Gu and Du Pei for explaining the physics connected to this work.


	\section{Preliminaries}
	\label
	{sec:K-theory}

	\subsection{Basic notation in $K$-theory}
	
	Let $X$ be a scheme with an action of an algebraic torus $T$. 
	The group $\kG(X)$ (resp. $\kg(X)$) denotes the Grothendieck group of $T$-equivariant coherent sheaves (resp. $T$-equivariant locally free sheaves) on $X$. For a coherent sheaf or vector bundle $E$, we denote its associated $K$-theory class simply by $E$. The non-equivariant $K$-groups of coherent sheaves and locally free sheaves are denoted by $K_0(X)$ and $K^0(X)$, respectively.
	
	If $X$ is nonsingular, the map $K^0(X) \rightarrow K_0(X)$, which sends a vector bundle to its sheaf of sections, is an isomorphism. In this case, we denote the non-equivariant $K$-group of $X$ by $K(X)$.
	
	Let $\mathrm{pt}=\spec\bb{C}$.
	For $T= (\C^{*})^{N}$, we have $$\Gamma:=\kg(\mathrm{pt}) =K^0(BT)\cong \bb{Z}[\alpha_{1}^{\pm},\dots, \alpha_{N}^{\pm}],$$
	where $\alpha_i$ is the $T$-character corresponding to the projection onto the $i$-th factor.

	For a class $F\in \kG(X)$ on a proper scheme $X$,
	the $T$-equivariant Euler characteristic is defined as
	\[
	\chi^{T}(X,F): = \sum_{i \geq0} \text{char}_{T} \, H^i(X, F) \in \Gamma.
	\]
	where $\mathrm{char}_{T}$ denotes the character of a $T$-module.

	\subsection{Symmetric functions and representations of $\GLr$}
	\label
	{sec:symmetric-functions}
	The $K$-theory of the Grassmannian is closely connected to $\GLr$ representations and symmetric functions.
	Let $r$ be a positive integer.
	We denote the representation ring of $\GLr$ by $R(\GLr)$.
	A \emph{partition} (with $r$ parts) is a weakly decreasing sequence of nonnegative integers $\lambda = (\lambda_1, \dots, \lambda_r)$, where $\lambda_1 \geq \lambda_2 \geq \dots \geq \lambda_r \geq 0$. Each partition corresponds to a \emph{Young diagram} with $\lambda_i$ boxes in row $i$. We often omit trailing zeros, writing $\lambda = (\lambda_1, \dots, \lambda_m)$ if $\lambda_i = 0$ for $i > m$. When $\lambda_1 = \lambda_2 = \dots = \lambda_r = a$, we denote this as $(a)^r := (a, \dots, a)$.
	
	Let $s_{\lambda}$ be the \emph{Schur polynomial} for a partition $\lambda$, defined as:
	\begin{equation*}
		s_{\lambda}(x_{1},\dots, x_{r})=
		\frac
		{\det
			(x_{j}^{\lambda_{i}+r-i}
			)_{1\leq i,j\leq r}
		}
		{
			\det
			( x_{j}^{r-i}
			)_{1\leq i,j\leq r}
		}
		,
	\end{equation*}
	where $\det
	( x_{j}^{r-i}
	)
	=\prod_{1\leq i<j\leq r}(x_{i}-x_{j})
	$ is the~\emph{Vandermonde determinant}.
	This is known as Jacobi's \emph{bialternant formula}. Note that $s_\lambda$ is a symmetric polynomial in the $x_i$'s.

	Let $\bS^{\lambda}$ be the \emph{Schur functor} associated with the partition $\lambda$. For any $\GLr$-module $\rV\in R(\GLr)$ and $g \in \GLr$, the trace of $g$ on $\bS^{\lambda}(\rV)$ is given by
	\[
	\chi_{\bS^{\lambda}(\rV)}(g)
	=
	s_{\lambda}(x_{1},\dots,x_{r}),
	\]
	where $x_{1},\dots,x_{r}$ are the eigenvalues of $g$ on $\rV$ (see~\cite[\textsection{6}]{Fulton-Harris}).
	Let $\C^r$ denote the standard representation of $\GLr$.
	Then, $\bS^{\lambda}(\C^{r})$ is the unique irreducible polynomial representation of $\GLr$ with highest weight $\lambda$. For example, if $\lambda = (a)$, then $\bS^{(a)}(\C^{r}) = \Sym^{a}(\C^{r})$; if $\lambda = (1)^{b}$ with 1 repeated $b$ times, then $\bS^{(1)^{b}}(\C^{r}) = \wedge^{b}(\C^{r})$.

	In this paper, we also consider the \emph{stable Grothendieck polynomials}, denoted by $G_{\lambda}$ for the partition $\lambda$ (see~\cite{Buch}), which have the following Weyl-type bialternant formula:
	\begin{equation}
		\label
		{eq:G-lambda-bialternant-formula}
		G_{\lambda}(x_{1},\dots, x_{r})=
		\frac
		{\det
			(x_{j}^{\lambda_{i}+r-i}(1-x_{j})^{i-1}
			)_{1\leq i,j\leq r}
		}
		{
			\det
			( x_{j}^{r-i}
			)_{1\leq i,j\leq r}
		}
		.
	\end{equation}

	By~\cite[Theorem 2.2]{Lenart}, we have the expansion
	\begin{equation}
		\label
		{eq:schur-expansion}
		G_{\lambda}
		(x_{1},\dots,x_{r})
		=
		\sum_{\lambda\subseteq \mu \subseteq \hat{\lambda}}
		(-1)^{|\mu|-|\lambda|}
		g_{\lambda\mu}
		s_{\mu}
		(x_{1},\dots,x_{r})
		.
	\end{equation}
	Here $\hat{\lambda}$ represents the maximal partition with $r$ rows,
	obtained from $\lambda$ by adding up to $i-1$
	boxes in the $i$th row for $2\leq i\leq r$.
	The coefficients $g_{\lambda,\mu}$ are nonnegative and
	count the number of increasing tableaux
	of skew shape $\mu/\lambda$,
	with the constraint that row $i$ contains entries
	no greater than $i-1$.
	According to~\cite[Theorem 2.7]{Lenart},
	Schur functions can also be expanded into Grothendieck polynomials as:
	\begin{equation*}
		s_\lambda
		(x_1,\dots,x_r)
		=
		\sum_{\substack{\mu\supseteq \lambda\\ \mu_1=\lambda_1}}
		f_{\lambda\mu}
		G_\mu
		(x_1,\dots,x_r),
	\end{equation*}
	where $f_{\lambda\mu}$ is the number of semistandard Young tableaux
	of shape $\mu/\lambda$ with row $i$ entries limited to
	$1,2,\dots,i-1$.

	Let $\rV\in R(\GLr)$. According to~\cite[\textsection{8}]{Lenart},
	there is a well-defined virtual $\GLr$-representation $\bG^{\lambda}(\rV)$ whose character satisfies
	\begin{equation}
		\label{eq:character-G-functor}
		\chi_{\bG^{\lambda}(\rV)}(g)=
		G_{\lambda}(1-x^{-1}_{1},\dots,1- x^{-1}_{r})
		,
	\end{equation}
	where $x_{1},\dots,x_{r}$ are the eigenvalues of $g$ on $\rV$. The following lemma provides details on the transition matrices between the representations $\bG^\lambda(\rV)$ and $\bS^\nu(\rV^\vee)$.

	\begin{lem}
		\label
		{lem:leading-term-one}
		Let $\lambda$ be a partition with at most $r$ parts.
		There exist unique integers $D_\lambda^\nu$ such that
		\[
		G_{\lambda}
		(1-y_{1},\dots,1- y_{r})
		=
		\sum_{\nu\subseteq (\lambda_1)^r } D^{\nu}_\lambda
		s_{\nu}(y_{1},\dots,y_{r})
		\]
		with $D^{\emptyset}_\lambda=1$.
	\end{lem}
	\begin{proof}
		The coefficients $D_\lambda^\nu$ can be determined by first expressing the character in~\eqref{eq:character-G-functor} as a linear combination of $s_\lambda(1 - x_1^{-1}, \dots, 1 - x_r^{-1})$ using \eqref{eq:schur-expansion},
		and then expanding it in terms of $s_{\lambda}(x_1^{-1}, \dots, x_r^{-1})$.
		We then prove that the constant term in this expansion is 1, i.e., $D^{\emptyset}_\lambda=1$.
		Specifically, observe that
		\begin{align*}
			G_{\lambda}{(1-y_1,\dots,1-y_r)}=&\frac{\det((1-y_j)^{\lambda_i+r-i} y_j^{i-1})_{1\le i,j\le r}}{\det((1-y_j)^{r-i})_{1\le i,j\le r}}\\&
			=\frac{\det((1-y_j)^{\lambda_i+r-i} y_j^{i-1})_{1\le i,j\le r}}{\det(y_j^{i-1})_{1\le i,j\le r}}.
		\end{align*}
		Expanding $(1-y_j)^{\lambda_i+r-i} $ using the binomial theorem shows that the lower-degree terms in the numerator match the denominator, making the constant term of $G_{\lambda}
		(1-y_{1},\dots,1- y_{r})$ equal to 1.
	\end{proof}

	\subsection{$K$-theory of the Grassmannian}
	\label
	{sec:K-theory-Grassmannian}

	We fix two positive integers $r$ and $N$ such that $N\geq r$.
	Throughout the paper, we set
	$$k:=N-r.
	$$

	Consider the Grassmannian
	$$X=\grass=\{V\subset \C^{N}\mid \dim_{\C} V=r\},$$
	which parametrizes $r$-planes in $\C^N$.
	The maximal torus $T \cong (\C^*)^N$ in $\mathrm{GL}_N(\C)$, consisting of diagonal matrices, acts on $X$. On $X$, there is a short exact sequence:
	\[
	0\rightarrow S \rightarrow \O^{\oplus N}_{X}\rightarrow Q\rightarrow 0,	
	\]
	where $S$ is the rank $r$ tautological subbundle with a natural $T$-linearization.
	
	The standard GIT presentation of $X$ is:
	\[
	X=[M_{r\times N}^{ss}/\text{GL}_r(\bb{C})],
	\]
	where $M_{r\times N}$ is the affine space of $r\times N$ complex matrices,
	and $M_{r\times N}^{ss}$ is its open subset of full-rank matrices.
	Define the quotient stack
	\begin{equation*}
		\fr{X}=[M_{r\times N}/\text{GL}_r(\bb{C})],
	\end{equation*}
	which contains $X$.

	Define $\prkk$ as the set of partitions
	contained within the rectangle $(k)^{r}=(k,k,\dots,k)$ with $r$ rows and $k$ columns:
	\[
	\prkk
	=\{
	\lambda=(\lambda_{1},\dots,\lambda_{r})
	\mid
	\lambda_{1}\leq k
	\}.
	\]
	The $K$-group $K(X)$ has a $\Z$-basis indexed by $\prkk$:
	\begin{equation*}
		\{
		\O_{\lambda}
		\}_{\lambda\in \prkk},
	\end{equation*}
	where $\O_\lambda$
	represents the $K$-theory class $[\O_{X_\lambda}]$ of the
	structure sheaf of the \emph{Schubert variety} $X_\lambda$ in $X$ (see~\cite[\textsection{8}]{Buch}).
	These classes $\O_\lambda$ are referred to as
	\emph{Schubert structure sheaves}.

	The $K$-theories of $X$ and $\fX$ are closely related to
	representations of $\GLr$.
	For a $\GLr$-representation $\mathrm{V}$, we define its associated $K$-theory class
	on $\fX$ as
	\begin{equation*}
		M_{r\times N} \times_{\GLr} \mathrm{V} \in \kg(\fr{X}),
	\end{equation*}
	where the $T$-linearization is given by
	$\lambda\cdot[m,v]=[m \lambda,v]$ with $\lambda\in T$, $v\in \mathrm{V}$,
	and $m\in M_{r\times N}$.
	We denote this $K$-theory class also by $\rV$ for simplicity.
	For example, let $\C^{r}$ denote the standard representation of $\GLr$, and let
	$$\mathrm{S}=(\C^{r})^{\vee}$$ be
	its dual.
	Then the restriction of the associated $K$-theory class $\mathrm{S}$ to $X$ is the tautological subbundle $S$.

	For later applications, we introduce (virtual) representations
	corresponding to Schubert structure sheaves.
	\begin{definition}
		For any partition $\lambda$ with $r$ (nonnegative) parts, we define 
		\[
		\rO_\lambda:=\bG^\lambda(\C^r) \in R(\GLr).
		\]
	\end{definition}
	\begin{rem}
		The representation $\rO_\lambda$ is nonzero,
			even for $\lambda\notin\prkk$.
			According to~\cite[Theorem 8.1]{Buch},
			the restriction of $\rO_\lambda$ to $X$ equals $\O_\lambda$
			for $\lambda\in\prkk$ and is zero otherwise. Consider the example $X=\mathrm{Gr}(1,N)\cong \mathbb{P}^{N-1}$, with universal subsheaf $S=\mathcal{O}_{\mathbb{P}^{N-1}}(-1)$. In our convention, for any non-negative integer $\lambda$, the representation $\mathrm{O}_{\lambda}=(1-\mathbb{C}^\vee )^{\lambda}\neq 0$, where $\mathbb{C}$ is the standard $\mathbb{C}^*$-representation. On the other hand, its restriction $(1-S)^{\lambda}=\mathcal{O}_H^{\otimes \lambda}$ is the $\lambda$-th power of the structure sheaf of a hyperplane $H\subset \mathbb{P}^{N-1}$, which vanishes when $\lambda\geq N$.
	\end{rem}

	\section{Vafa--Intriligator type formula}
	\label{sec:localization}

	In this section, we first apply the Atiyah–Bott localization to derive
	a Vafa--Intriligator type formula for the Euler characteristics of certain $K$-theory classes over Quot schemes on $\mathbb{P}^1$. While the torus localization used here is inspired by ~\cite{Marian-Oprea,OpreaShubham}, the combinatorics in our approach differ slightly.
	In the last subsection, we generalize the formula to all genera.

	\subsection{1-pointed insertions over Quot schemes}
	
	Let $\Quot$ be the Quot scheme parametrizing isomorphism classes of quotients of $\O_{U\times \bb{P}^{1}}^{\oplus N}$ that are flat over a scheme $U$.
	There is a universal short exact sequence of coherent sheaves over $\Quot\times \bP^{1}$:
	\begin{equation}
		\label
		{eq:universal-seq-quot-scheme}
		0\rightarrow
		\ca{S}
		\rightarrow
		\O_{\Quot\times \bb{P}^{1}}^{\oplus N}
		\rightarrow
		\ca{Q}
		\rightarrow 
		0.
	\end{equation}
	For any closed point $z \in \Quot$, the restriction of $\ca{S}$ to $\{z\}\times \bb{P}^{1}$ is a vector bundle of rank $r$ and degree $-d$.
	
	The Quot scheme $\quot_d(\bP^1,N,r)$ is a smooth irreducible scheme of dimension $dN+r(N-r)$; see~\cite{Stromme}. Moreover, the tangent bundle is $\Hom_\pi (\cS,\cQ)=\pi_{\star}(\mathcal{S}^\vee \otimes \mathcal{Q})$, where $\pi: \Quot \times \bP^{1} \rightarrow \Quot$ is the projection to the first factor.
	
	Consider the standard $T = (\C^*)^N$ action on $\mathcal{O}_{\bP^1}^{\oplus N}$, which induces a $T$-action on $\Quot$ and a natural linearization of the tautological subbundle $\cS$. For $1 \leq i \leq N$, let $\alpha_i$ denote the $T$-character corresponding to the projection onto the $i$-th factor.
	
	Recall the following quotient stack containing the Grassmannian $X=\grass$:
	$$\fr{X}:=[M_{r\times N}/\GLr].$$
	Fix a point $p\in \bb{P}^1$.
	We denote $\cS_p$ as the restriction of $\cS$ to $\Quot \times \{p\}$. 
	There is a $T$-equivariant `stacky' evaluation morphism
	\[
	\wev_p:\Quot \rightarrow 
	\fX
	\]
	induced
	by the restriction $\ca{S}_{p}\rightarrow
	\O_{p}^{\oplus N}$.
	Note that $\wev_p$ may not always factor through $X\subset \fr{X}$.
	Recall that for a $\GLr$-representation $\mathrm{V}$, the associated $K$-theory class on $\fr{X}$ is $$M_{r\times N} \times_{\GLr}\mathrm{V} \in \kg(\fr{X}),$$ giving the isomorphism $\kg(\fr{X}) \cong R(\GLr) \otimes_{\bb{Z}} \Gamma$. By abuse of notation, we denote this $K$-theory class by $\rV$.
	
	\begin{definition}
		Let $\ell\in \Z$. For any $\mathrm{V}\in \kg(\fX) $, the level-$\ell$ $T$-equivariant 1-pointed
		$K$-theoretic Quot scheme invariant is defined as
		\[
		\langle
		\mathrm{V}
		\rangle^{\quot,T,\ell}_{0,d}
		:=
		\chi^{T}
		\left(
		\Quot,
		\det(\pi_{\star}\cS^\vee)^{-\ell}
		\cdot
		\wev_p^{*}
		\left(
		\mathrm{V}
		\right)
		\right)
		.
		\]
		For $\ell = 0$, we denote the $T$-equivariant invariant as $\langle V \rangle^{\quot,T}_{0,d}$. The superscript $T$ is omitted for non-equivariant invariants. Note that we denote the derived pushforward by $\pi_{\star}$.
	\end{definition}
	
	\begin{rem}
		For an irreducible representation $V=\bS^{\lambda}(\mathbb{C}^r)$, the stacky pullback $\widetilde{\mathrm{ev}}_p^{*}\left(\mathrm{V}\right) \cong \bS^{\lambda}(\mathcal{S}_{p}^\vee)$ is the Schur functor applied to the vector bundle $\mathcal{S}_{p}^\vee$ (see \cite{weyman} for construction of Schur functors). Note that the $K$-theory class of $\mathcal{S}_{p}$ is independent of the point $p$; for instance, it sits in the short exact sequence
			\[
			0 \to \pi_{\star}\mathcal{S}(-1)\to \pi_{\star}\mathcal{S}\to \mathcal{S}_{p}\to 0.
			\]
			In particular, this implies that the $K$-theory class of $\widetilde{\mathrm{ev}}_p^{*}\left(\mathrm{V}\right)$ is independent of the point $p$.
	\end{rem}

	\begin{rem}
		\label{rem:quantum/gauge-corr}
		$K$-theoretic Quot scheme invariants have close connections to 3d gauge theories~\cite{Jockers-Mayr1, Jockers-Mayr2,  Ueda-Yoshida, GMSZ1,GMSZ2,GMSZ3,GPZ}.
		The level $\ell$ corresponds to a specific choice of Chern--Simons levels, and the generating function of $n$-pointed $K$-theoretic Quot scheme invariants
		correspond to that of Wilson loops in the 3d $\ca{N}=2$ Chern--Simons-matter theory on $S^{1}\times S^{2}$;
		see~\cite[\textsection{4}]{Ueda-Yoshida}. 
	\end{rem}

	The remainder of this section is dedicated to establishing the localization framework and proving the
	following Vafa--Intriligator type formula. For clarity, the setup and proof are presented in separate sections.
	
	\begin{theorem}\label{thm:Vafa_Int_formula_equivariant}
		For any $\GLr$-representation $\mathrm{V}$ and the level $\ell$ satisfying $-r<\ell\le (N-r)$,
		we have
		\[
		\langle
		\rV
		\rangle^{\quot,T,\ell}_{0,d}
		=
		[t^d]
		\sum_{z_1,\dots,z_r}
		v(z_1,\dots,z_r)\frac{\prod_{i=1}^{r}z_i^{N-r+d-\ell}\cdot\prod_{i\neq j}(z_i-z_j)}{\prod_{i=1}^{r}P'(z_i)},
		\]
		where 
		\begin{itemize}
			\item 
			the sum is over all $\binom{N}{r}$ sets of $r$ distinct roots $z_1,z_2,\dots,z_r$ of
			\begin{equation}
				\label
				{eq:equivariant-bethe-ansatz}
				P(z)
				=
				\prod_{i=1}^{N}(z-\alpha_i^{-1})
				+(-1)^rz^{N-r-\ell}t=0,
			\end{equation}
			and $\alpha_i$ is the character of $T=\left(\bb{C}^*\right)^N$ given by the $i$-th projection;
			\item $v(z_1,\dots,z_r)$ is the character of $\rV$; and 
			\item the operation $[t^{d}]$ on a symmetric Laurent polynomial in the $N$ roots of~\eqref{eq:equivariant-bethe-ansatz} involves three steps:
			express the Laurent polynomial in terms of $t$ using~\eqref{eq:equivariant-bethe-ansatz}, expand it as a power series around $t=0$,
			and extract the coefficient of $t^{d}$.
		\end{itemize} 
	\end{theorem}

	\subsection{Fixed loci}
	
	Consider the $T$-action on $\quot_{d}:=\quot_d(\bP^1,N,r)$
	induced by the standard $T=(\C^*)^{N}$-action on $\mathcal{O}_{\bP^1}^{\oplus N}$.
	Set $[N]=\{1,2,\dots,N\}$.
	The fixed loci of this action is parameterized by pairs $(\vec{d}, I)$, where $\vec{d}=(d_1,\dots, d_r)$ with $|\vec{d}|=d_1+\cdots+d_r=d$, and $I\subseteq [N]$ is a subset of size $r$. Moreover, the fixed loci are isomorphic to products of projective spaces:
	\[\fix_{\vec{d},I}=\mathbb{P}^{d_1}\times \cdots \times \mathbb{P}^{d_r}. \]
	The factor $\mathbb{P}^{d_i}$ corresponds to the Hilbert scheme of $d_i$ points parameterizing short exact sequences \[0\to K_i\to \mathcal{O}_{\mathbb{P}^1}\to T_i\to 0 \] such that $T_i$ is a torsion sheaf of length $d_i$. 
	The corresponding point in the fixed locus $\fix_{\vec{d},I}$ is represented by \[0\to S\to \bigoplus_{i\in I}\mathcal{O}_{\mathbb{P}^1}\to \bigoplus_{i\in [N]}\mathcal{O}_{\mathbb{P}^1} \to Q\to 0, \]
	where
	\[
	S=K_1\oplus\cdots \oplus K_r\hspace{1em} \text{and}\hspace{1em} 
	Q\cong T_1\oplus \cdots \oplus T_{r}\oplus\cO_{\mathbb{P}^1}^{\oplus N-r}.
	\]
	Let $\mathcal{K}_i$ and $\mathcal{T}_i$ denote the tautological subbundle and the quotient bundle on $\mathbb{P}^{d_i}\times \mathbb{P}^1$. We shall use the same notation for their pullback to $\fix_{\vec{d},I}\times\mathbb{P}^1$. 	Note that
	\[
	\mathcal{K}_i =
	\cO_{\mathbb{P}^{d_i}}(-1)
	\boxtimes 
	\cO_{\mathbb{P}^1}(-d_i).
	\]

	\subsection{Todd calculations and equivariant normal bundle}
	Recall that for $1\leq i\leq N$, we denote by $\alpha_{i}$ the character of $T$ given by the projection to the $i$-th factor.
	For any character $\alpha$ of $T$, we denote by $\C_{\alpha}$ the corresponding 1-dimensional representation of $T$. Given any $K$-theory class $E$, we set $E_{(\alpha)}:=E\otimes \C_{\alpha}$.

	Let $\pi:\quot_{d}\times\bP^{1}\rightarrow \quot_{d}$ be the projection to the first factor.
	Let $\cS$ and $\ca{Q}$ be the tautological subbundle and quotient sheaf over $\pi$, respectively (see~\eqref{eq:universal-seq-quot-scheme}).
	It is a standard result that the tangent bundle $T_{\quot_d}=\Hom_\pi (\cS,\cQ)$ restricts to
	\[
	\bigoplus_{i,j\in I}\pi_{\star}(\mathcal{K}^{\vee}_i \otimes\mathcal{T}_j)_{(\alpha_j\alpha_i^{-1})}
	\oplus
	\bigoplus_{i\in I,j\in [N]\backslash I }\pi_{\star}(\mathcal{K}_i^\vee
	)_{(\alpha_j\alpha_i^{-1})}
	\]
	over the fixed loci $\fix_{\vec{d},I}$.
	In $K$-theory, it equals
	\[
	\bigoplus_{i\in I,j\in [N]}\pi_{\star}(\mathcal{K}^{\vee}_i)_{(\alpha_j\alpha_i^{-1})}
	\ominus
	\bigoplus_{i,j\in I}\pi_{\star}(
	\mathcal{K}^{\vee}_i \otimes\mathcal{K}_j)_{(\alpha_j\alpha_i^{-1})}
	.
	\]
	Therefore the equivariant Todd class of $T_{\quot_d}$ restricted to the fixed loci is
	\[
	\prod_{i\in I,j\in [N]}^{}\td(
	\pi_{\star}(\mathcal{K}^{\vee}_i)_{(\alpha_j\alpha_i^{-1})}
	)
	\bigg(\prod_{i,j\in I}\td(\pi_{\star}(\mathcal{K}^{\vee}_i \otimes\mathcal{K}_j)_{(\alpha_j\alpha_i^{-1})})
	\bigg)^{-1}. \]

	We know that the normal bundle is given by the moving part of the restriction of the tangent bundle to the fixed loci. We observe that the moving part is
	\[
	\nbun_{\vec{d},I}
	=
	T^{\rm mov}_{\quot_{d}}
	\bigg|_{\fix_{\vec{d},I}}
	= \bigoplus_{i\in I,j\in [N];i\ne j}
	\pi_{\star}(\mathcal{K}_{i}^{\vee})
	_{(\alpha_j\alpha_i^{-1})}
	\ominus
	\bigoplus_{i,j\in I; i\ne j}
	\pi_{\star}
	(\mathcal{K}^{\vee}_i \otimes\mathcal{K}_j)_{(\alpha_j\alpha_i^{-1})}
	.
	\]
	Therefore, we have
	\[
	\frac{1}{e_T(\cN_{\vec{d},I})}
	=
	\prod_{i\in I,j\in [N];i\ne j } \bigg(e_{T}(
	\pi_{\star}(\mathcal{K}_{i}^{\vee})
	_{(\alpha_j\alpha_i^{-1})}
	)
	\bigg)^{-1}
	\prod_{i,j\in I; i\ne j}
	e_{T}(
	\pi_{\star}
	(\mathcal{K}^{\vee}_i \otimes\mathcal{K}_j)_{(\alpha_j\alpha_i^{-1})}
	).
	\]

	\subsection{Explicit contributions}
	
	Let $h_i\in H^{2}(\fix_{\vec{d},I})$ denote the pullback of the hyperplane class in $\mathbb{P}^{d_i}$. In equivariant cohomology, we have
	\begin{align*}
		c^T(
		\pi_{\star}(\mathcal{K}^\vee_i)_{(\alpha_j\alpha_i^{-1})}
		)
		&
		= (1+(h_i-\epsilon_{i}+\epsilon_{j}))^{d_i+1},
		\\
		c^T(
		\pi_{\star}(\mathcal{K}^{\vee}_i \otimes\mathcal{K}_j)_{(\alpha_j\alpha_i^{-1})}
		)
		&
		=(1+(h_i-\epsilon_{i}-h_j+\epsilon_{j}))^{d_i-d_j+1},
	\end{align*}
	where $\epsilon_{i}$'s are the (cohomological) equivariant parameters such that $\alpha_{i}=e^{\epsilon_{i}}$.
	This implies the following expressions for the equivariant Todd classes:
	\begin{align*}
		\td(\pi_{\star}
		(
		\mathcal{K}^\vee_i)_{(\alpha_j\alpha_i^{-1})}
		)
		&= \bigg(\frac{h_i-\epsilon_{i}+\epsilon_{j}}{1-e^{-(h_i-\epsilon_{i}+\epsilon_{j})}}\bigg)^{d_i+1},
		\\
		\td(
		\pi_{\star}(\mathcal{K}^{\vee}_i \otimes\mathcal{K}_j)_{(\alpha_j\alpha_i^{-1})}
		)&= \bigg(\frac{h_i-\epsilon_i-h_j+\epsilon_j}{1-e^{-(h_i-\epsilon_i-h_j+\epsilon_j)}}\bigg)^{d_i-d_j+1}.
	\end{align*}
	The latter equals $1$ when $i=j$.
	Similarly, we obtain the Euler classes :
	\begin{align*}
		e_{T}
		(
		\pi_{\star}(\mathcal{K}^\vee_i)_{(\alpha_j\alpha_i^{-1})}
		)		
		&= (h_i-\epsilon_i+\epsilon_j)^{d_i+1},\\
		e_{T}(
		\pi_{\star}(\mathcal{K}^{\vee}_i \otimes\mathcal{K}_j)_{(\alpha_j\alpha_i^{-1})}
		)&= (h_i-\epsilon_i-h_j+\epsilon_j)^{d_i-d_j+1}.
	\end{align*} 
	\textbf{Simplification.} Observe that over the fixed locus $\fix_{\vec{d},I}$, the factor $\frac{\td (\quot_d)}{e_{T}(		\nbun_{\vec{d},I}
		)}$ restricts to
	\[
	\prod_{i\in I}
	e_{T}
	(
	\pi_{\star}(\mathcal{K}^\vee_i)_{(1)}
	)	
	\prod_{i\in I,j\in [N]}^{}
	\frac{\td}{e_{T}}\bigg(
	\pi_{\star}(\mathcal{K}^\vee_i)_{(\alpha_j\alpha_i^{-1})}
	\bigg)
	\prod_{i,j\in I; i\ne j}\frac{e_{T}}{\td}\bigg(
	\pi_{\star}(\mathcal{K}^{\vee}_i \otimes\mathcal{K}_j)_{(\alpha_j\alpha_i^{-1})}
	\bigg).
	\]
	
	For notational convenience, we set \[z_i=e^{h_i-\epsilon_{i}}\quad \text{and}\quad 
	R(z)=\prod_{i=1}^{N}(z-\alpha_i^{-1})
	,\]
	where $ \alpha_i=e^{\epsilon_{i}}$.
	Thus 
	\begin{align*}
		\frac{\td (\quot_d)}
		{e_{T}
			(
			\nbun_{\vec{d},I}
			)
		}\bigg|_{\fix_{\vec{d},I}}=&
		\prod_{i\in I}
		\bigg(
		h_i^{d_i+1}
		\bigg(
		\frac{z_i^N}{R(z_i)} \bigg)^{d_i+1} 
		\bigg)
		\cdot \prod_{i,j\in I; i\ne j}\bigg(\frac{z_i-z_j}{z_i}\bigg)^{d_i-d_j+1}	\\
		=&(-1)^{(r-1)d}\prod_{i\in I}\bigg(
		\bigg(
		\frac{h_iz_i^{N-r}}{R(z_i)} \bigg)^{d_i+1} z_i^{ d+1}
		\bigg)
		\cdot
		\prod_{i,j\in I; i\ne j}(z_i-z_j).
	\end{align*}

	\begin{proof}
		Applying Hirzebruch--Riemann--Roch followed by Atiyah--Bott localization, we write the Euler characteristic of $\det(\pi_{\star}\cS^\vee)^{-\ell}\otimes \wev_p^*(\mathrm{V})$ as
		\[
		\langle
		\mathrm{V}
		\rangle^{\quot,T,\ell}_{0,d}=
		\sum_{\vec{d}, I}\int_{\fix_{\vec{d},I}}
		\ch\Big(
		\det(\pi_{\star}\cS^\vee)^{-\ell}\big|_{\fix_{\vec{d},I}}
		\Big)
		\ch
		\Big(
		\wev_p^*(\mathrm{V})
		\big|_{\fix_{\vec{d},I}}
		\Big)
		\frac{\td(\quot_d)}{e_{T}
			(
			\nbun_{\vec{d},I}
			)
		}\bigg|_{\fix_{\vec{d},I}},\]
		where $I=\{i_1,\dots,i_r\}$ runs over subsets of $[N]=\{1,2,\dots , N\}$ of $r$ elements and $\vec{d}$ runs over the tuples of non-negative integers $(d_1,\dots ,d_r)$ that sum to $d$.
		
		Recall that $\mathcal{K}_i=
		\cO_{\mathbb{P}^{d_i}}(-1)\boxtimes \cO_{\mathbb{P}^1}(-d_i) $. Hence, over the fixed loci $\fix_{\vec{d},I}$,
		we have
		\[
		\pi_{\star}\cS^{\vee}
		\big
		|_{\fix_{\vec{d},I}}
		=
		\bigoplus_{i=1}^{r}
		\ca{O}_{\bP^{d_{i}}}(1)^{\oplus (d_{i}+1)}
		\quad
		\text{and}
		\quad
		\cS_p^\vee
		\big
		|_{\fix_{\vec{d},I}}
		= 
		\bigoplus_{i=1}^{r}
		\cO_{\mathbb{P}^{d_i}}(1).
		\]
		Since $c_{1}^{T}(	 \cO_{\mathbb{P}^{d_i}}(1))=h_i-\epsilon_{i}$, we have
		\begin{align*}
			\ch\Big(
			\det(\pi_{\star}\cS^\vee)^{-\ell}\big|_{\fix_{\vec{d},I}}
			\Big)
			&=		\prod_{i\in I}z_i^{-\ell(d_i+1)}
			,\quad\text{and}
			\\
			\ch
			\Big(
			\wev_p(\mathrm{V})
			\big |_{\fix_{\vec{d},I}}
			\Big)
			&=
			v(
			z_{I}
			),
		\end{align*}
		where $z_I=(z_{i_1},\dots,z_{i_r})$ with $z_i=e^{h_i-\epsilon_{i}}$, and $v(z_1,\dots,z_r)$ is the character Laurent polynomial for the $\GLr$-representation $\mathrm{V}$.

		We use the explicit calculations from the previous subsection to obtain
		\[
		(-1)^{(r-1)d}
		\sum_{\vec{d}, I}[h^{\vec{d}}]
		v(z_I)
		\prod_{i\in I}z_i^{-\ell(d_i+1)}
		\prod_{i\in I}
		\bigg(
		\bigg(
		\frac{h_iz_i^{N-r}}{R(z_i)}\bigg)^{d_i+1}z_i^{d+1}
		\bigg)
		\prod_{i,j\in I;i\ne j}(z_i-z_j).
		\]
		Here
		$[h^{\vec{d}}]$ denotes taking the coefficient of $\prod_{i\in I}^{}h_i^{d_i}$,
		which corresponds to integrating over the product of projective spaces $\fix_{\vec{d},I}=\mathbb{P}^{d_1}\times \cdots \times \mathbb{P}^{d_r}$. 
		
		We invoke the multivariate Lagrange--B\"urmann formula to sum over $\vec{d}$ (see~\cite{Gessel}). We use the following formulation of the formula: For any formal power series $\Psi(h_1,\dots,h_r)$, and $\Psi_1(h_1),\dots ,\Psi_N(h_r)$ with $\Psi_i(0)\ne 0$, we have
		\begin{align*}
			[h^{\vec{d}}]\Psi(h_1,\dots,h_r)\prod_{i=1}^{r}\Psi_i(h_i)^{d_i+1}=[t^{\vec{d}}]\Psi(h_1,\dots,h_r)\prod_{i=1}^{r}\frac{dh_i}{d t_i},
		\end{align*}
		where we use the change of variable $t_i=\frac{h_i}{\Psi_i(h_i)},$
		and express $h_i$ in terms of $t_i$ on the right hand side.
		
		In our problem, we choose
		\[
		\Psi_{i}(h_{i})=
		\frac{h_iz_i^{N-r-\ell}}{R(z_i)}
		\]
		and use the change of variable
		\begin{equation}
			\label{eq:new_ti}
			t_i=
			\frac{R(z_i)}
			{z_i^{N-r-\ell}},
		\end{equation}
		where $t_i$ is considered as a power series in $h_i$.
		Furthermore,
		\begin{equation}\label{eq:dti/dh}
			\frac{dt_i}{dh_i}= \frac{R'(z_i)-(N-r-\ell)z_i^{-1}R(z_i)}{z_i^{N-r-\ell-1}}.
		\end{equation}
		The Lagrange--B\"urmann formula implies that the previous expression equals
		\begin{multline}
			\label{eq:intermediate-step}
			(-1)^{(r-1)d}
			\sum_{\vec{d},I}[t^{\vec{d}}]
			v(z_I)\prod_{i\in I}
			\bigg(
			\frac{d h_i}{d t_i}z_i^{d+1}
			\bigg)
			\prod_{i,j\in I;i\ne j}(z_i-z_j)
			\\
			=
			(-1)^{(r-1)d}
			\sum_{I}
			\sum_{\vec{d}}[t^{\vec{d}}]
			v(z_I)\prod_{i\in I}
			\bigg(
			\frac{d h_i}{d t_i}z_i^{d+1}
			\bigg)
			\prod_{i,j\in I;i\ne j}(z_i-z_j)
		\end{multline}
		Note that in the above formula, the expression after $[t^{\vd}]$ is independent of $d_{i}$'s
		(and only depends on $d$ and $I$).
		Given a function $f(t_{1},\dots, t_{N})$, we have
		\[
		\sum_{\vd}
		[t^{\vd}]
		f(t_{1},\dots,t_{N})
		=[t^{d}]f(t,\dots,t).
		\] 
		Hence, to evaluate~\eqref{eq:intermediate-step}, we let $t_1=\cdots =t_N=t$ and find the coefficient of $t^d$ in the result sum.
		After the specialization, we note that $z_1,\dots,z_N$ are distinct roots of
		the following polynomial in $z$ (see~\eqref{eq:new_ti}):
		\[
		P(z)=
		R(z)-z^{N-r-\ell}t
		.
		\]
		Note that $P(z)$ is a monic polynomial of degree $N$ because
		$-r< \ell\leq (N-r)$.
		Hence $P(z)=\prod_{i=1}^{N} (z-z_{i})$.

		Using~\eqref{eq:new_ti} and~\eqref{eq:dti/dh}, we observe that
		\[
		\frac{dh_i}{dt}= \frac{z_i^{N-r-\ell-1}}{P'(z_i)}. \]
		We thus rewrite~\eqref{eq:intermediate-step} as 
		
		\begin{equation*}
			(-1)^{(r-1)d}
			[t^d]
			\sum_{I}
			v(z_I)\prod_{i\in I}\frac{z_i^{d+N-r-\ell}}{
				P'(z_i)
			}\prod_{i,j\in I;i\ne j}(z_i-z_j)
			.
		\end{equation*}
		We finish the proof of Theorem~\ref{thm:Vafa_Int_formula_equivariant} by substituting $t\mapsto (-1)^{r-1}t$.
	\end{proof}

	\subsection{Vafa--Intriligator type formula in all genera}
	\label
	{subsec:VI-in-all-genera}
	In this subsection, we focus on the non-equivariant case with level $\ell = 0$.
	
	Fix a smooth projective curve $C$ of genus $g$. Let $\QuotC$ be the Quot scheme
	parametrizing
	short exact sequences $$0\rightarrow S\rightarrow \O_{C}^{\oplus N}
	\rightarrow Q\rightarrow 0$$
	of coherent sheaves on $C$, where $S$ has rank $r$ and degree $-d$.
	The Quot scheme $\quot_d(C,N,r)$ may not be smooth when $C\neq \mathbb{P}^1$.
	A two-term perfect obstruction theory for $\quot_d(C,N,r)$ was constructed
	in~\cite{Marian-Oprea}, yielding a virtual structure sheaf $\ovir_{\QuotC}$ via~\cite{Lee}.

	Fix $p\in C$.
	Let
	\[
	\wev_p:\QuotC \rightarrow 
	\fX
	\]
	be the `stacky' evaluation morphism at $p$.
	\begin{definition}
		For $\mathrm{V}_1,\dots,\mathrm{V}_n\in R(\GLr)$ and $p_1,\dots,p_n\in C$, we define the $K$-theoretic Quot scheme invariant by
		\[
		\langle
		\mathrm{V}_{1},
		\dots,
		\mathrm{V}_{n}
		\rangle^{\quot}_{g,d}
		:=\chi
		\bigg(
		\quot_d(C,N,r),
		\ovir_{\QuotC}\cdot
		\prod_{i=1}^{n}
		\wev_{p_i}^{*}
		\left(
		\mathrm{V}_{i}
		\right)
		\bigg).
		\]
		We define the corresponding generating series by
		\[\langle\!\langle
		\mathrm{V}_{1},
		\dots,
		\mathrm{V}_{n}
		\rangle\!\rangle^{\quot}_{g} =\sum_{d\ge0}^{}q^d\langle
		\mathrm{V}_{1},
		\dots,
		\mathrm{V}_{n}
		\rangle^{\quot}_{g,d}, \]
		where $q$ is the Novikov variable. See \cite[Remark~1.8]{SinhaZhang} for the independence of the above definition from the choice of points $p_i \in C$. For the rest of the article, we assume $p = p_1 = \cdots = p_n$, which implies \[ 	\langle
			\mathrm{V}_{1},\mathrm{V}_2,
			\dots,
			\mathrm{V}_{n}
			\rangle^{\quot}_{g,d} = 	\langle
			\mathrm{V}_{1}\cdot \mathrm{V}_2 
			\cdots
			\mathrm{V}_{n}
			\rangle^{\quot}_{g,d}.
			\]
	\end{definition}
	The above invariants are independent of the choices of $p$ and the complex structure of the genus $g$ curve $C$.
	
	To state the induction formula of the $K$-theoretic Quot scheme invariants
	on the genus, we first introduce the concept of the \emph{quantized duals} to $\rO_{\lambda}$'s
	and the \emph{Euler represenation} $\rH$.
	
	\begin{definition}
		For any $\lambda=(\lambda_1,\dots,\lambda_r)\in\P_{r,N-r}$,
		we define the \emph{quantized $K$-theoretic dual} of $\rO_\lambda$
		by
		\[
		\rO_{\lambda}^{*}
		:=
		\rO_{\lambda^*}
		\cdot
		\det(\rS)
		\in 
		R(\GLr),
		\]
		where $\lambda^*:=(N-r-\lambda_{r},\dots, N-r-\lambda_{1})$ is the complement partition
		and $\mathrm{S}=(\C^{r})^{\vee}\in R(\GLr)$ is the dual of the standard representation.
	\end{definition}
	According to~\cite[Proposition 3.2]{SinhaZhang}, we have the following orthogonality relation:
	\begin{equation}
		\label{eq:orthogonality}
		\langle\!\langle
		\rO_{\nu}
		,
		\rO_{\lambda}^*
		\rangle\!\rangle^{\quot}_0
		=\delta_{\nu,\lambda}
		\quad
		\text{for}
		\ 
		\nu,\lambda\in \P_{r,N-r}.
	\end{equation}

	\begin{definition}
		We define the Euler representation $\rH$ by
		\[
		\rH=\sum_{\alpha \in\P_{r,N-r}}
		\mathrm{O}_{\alpha}\cdot\mathrm{O}^*_{\alpha}
		\in R(\GLr).
		\]
	\end{definition}

	\begin{rem}
		The associated $K$-theory class of $\rH$ restricted to $X=\grass$ is the $K$-theoretic Euler
		class $\wedge_{-1}(TX^\vee)$ of the tangent bundle of $X$.
		When $r=1$, we have 
		\[
		\rH = N(1-\rS)^{N-1}\rS,
		\]
		where $\rS$ is the dual of the standard representation of $\mathrm{GL}_1(\C)\cong\C^*$.
	\end{rem}
	
	\begin{rem}
		The quantum $K$-ring $\qk(X)$ of $X$ equipped with the quantized pairing
		$(F_{\alpha,\beta})$ is a commutative Frobenius algebra (see Section~\ref{subsec:quantum-k-ring}).
		Let $(F^{\alpha,\beta})$ be the matrix inverse to the quantized pairing matrix $(F_{\alpha,\beta})$.
		Let $\kappa: 
		R(\GLr)
		\ra
		\qk(X)$ be the quantum reduction map defined in Definition~\ref{def:kappa-map}.
		Then 
		\[
		\kappa(\rH) = \sum_{\alpha,\beta \in\P_{r,N-r}}
		\O_{\alpha}\cdot F^{\alpha,\beta}\cdot \O_{\beta}
		\]
		is the \emph{handle element} of $\qk(X)$. Here, we use the defintion from Exercise 5 in~\cite[\textsection{2.3}]{Kock}. It is the central element corresponds to the handle operator.
	\end{rem}

	According to~\cite[\textsection{3.3}]{SinhaZhang}, we have an induction formula on the genus:
	\[
	\langle\!\langle
	\mathrm{V} \rangle\!\rangle^{\quot}_{g}
	=
	\sum_{\alpha \in\P_{r,k}} \langle\!\langle
	\mathrm{V},
	\mathrm{O}_{\alpha}, \mathrm{O}^*_{\alpha}
	\rangle\!\rangle^{\quot}_{g-1}
	=  \langle\!\langle
	\mathrm{V}\cdot
	\rH
	\rangle\!\rangle^{\quot}_{g-1},
	\]
	for any $V\in R(\GLr)$.
	By induction, we have 
	\[
	\langle\!\langle
	\mathrm{V} \rangle\!\rangle^{\quot}_{g}
	=
	\langle\!\langle
	\mathrm{V}\cdot
	\rH^g
	\rangle\!\rangle^{\quot}_{0}.
	\]
	In particular, it implies that
	\[
	\langle
	\mathrm{V} \rangle^{\quot}_{g,d}
	=
	\langle
	\mathrm{V}\cdot
	\rH^g
	\rangle^{\quot}_{0,d}
	\]
	for $d\geq 0$.
	Combining the above equation with Theorem~\ref{thm:Vafa_Int_formula_equivariant}, we obtain the following
	Vafa--Intriligator type formula in all genera:
	\begin{theorem}
		\label{thm:Vafa_Int_formula_equivariant_all_genera}
		For any $\GLr$-representation $\mathrm{V}$,
		we have
		\[
		\langle
		\rV
		\rangle^{\quot}_{g,d}
		=
		[t^d]
		\sum_{z_1,\dots,z_r}
		v(z_1,\dots,z_r)h(z_1,\dots,z_r)^g\frac{\prod_{i=1}^{r}z_i^{N-r+d-\ell}\cdot\prod_{i\neq j}(z_i-z_j)}{\prod_{i=1}^{r}P'(z_i)},
		\]
		where we use the same notation as in Theorem~\ref{thm:Vafa_Int_formula_equivariant}
		and $h(z_1, \dots, z_r)$ is a polynomial in $z_i^{-1}$'s, representing the character of $\rH$.
	\end{theorem}
	
	\begin{example}
		When $r=1$, we have
		\[
		h(x^{-1}) = N(1-x)^{N-1}x.
		\]
		When $r=2$, let $s_\lambda=s_\lambda(x_1,x_2)$ denote the Schur function in two variables (see Section~\ref	{sec:symmetric-functions}).
		For $\mathrm{Gr}(2,4)$, we have
		\[
		h(x_1^{-1},x_2^{-1})=6\,s_{1, 1} - 8\,s_{2, 1} + s_{2, 2} +
		3\,s_{3, 1} + 10\,s_{3, 3} - 8\,s_{4, 3} +
		3\,s_{4, 4} + s_{5, 3}.
		\]
		For $\mathrm{Gr}(2,5)$, we have
		\begin{align*}
			h(x_1^{-1},x_2^{-1})=&
			10\,s_{1, 1} - 20\,s_{2, 1} + 5\,s_{2, 2} + 15\,s_{3, 1} - 2\,s_{3, 2} +
			50\,s_{3, 3} - 4\,s_{4, 1} - 80\,s_{4, 3}\\
			& + 55\,s_{4, 4} + 45\,s_{5, 3} -
			30\,s_{5, 4} + 6\,s_{5, 5} - 10\,s_{6, 3} + 3\,s_{6, 4} + s_{7, 3}.
		\end{align*}
	\end{example}
	
	\begin{rem}
		Building on the work of Gorbounov--Korff~\cite{Gorbounov-Korff-1,Gorbounov-Korff-2},
		Ueda--Yoshida~\cite{Ueda-Yoshida} derived formulas for genus $g$ $n$-pointed correlations in the
		2d TQFT associated with the 3d $\ca{N}=2$ Chern–Simons-matter theory on $S^1 \times C$.
		As noted in Remark~\ref{rem:quantum/gauge-corr}, these correlation functions are expected to agree
		with the generating series of $K$-theoretic Quot scheme invariants.
		In particular, equation (2.51) in~\cite{Ueda-Yoshida} suggests that the generating series
		of the invariant in Theorem~\ref{thm:Vafa_Int_formula_equivariant_all_genera} can be expressed as a sum
		indexed by distinct roots of the Bethe ansatz equation (see~\cite[(2.20)]{Ueda-Yoshida}).
		In contrast, our theorems involve a single equation~\eqref{eq:equivariant-bethe-ansatz}, while the Bethe ansatz equation is a system of $r$ algebraic equations. Exploring the connections between these two approaches would be interesting.
	\end{rem}

	\section{Euler characteristics of Schur bundles}
	\label{sec:schur-bundle}
	
	\subsection{Bialternant formulas}
	For a weakly decreasing sequence of integers $\lambda=(\lambda_{1},\lambda_{2},\dots,\lambda_r)$,
	we define the irreducible (rational) representation
	of $\GLr$ of highest weight
	$\lambda$ as:
	\[
	\bS^{\lambda}(\C^{r})
	:=
	\det(\C^{r})^{\lambda_{r}}\otimes
	\bS^{\bar{\lambda}}(\C^{r})
	,
	\]
	where $\bar{\lambda}=(\lambda_{1}-\lambda_{r},\lambda_{2}-\lambda_{r},\dots,\lambda_{r}-\lambda_{r})$
	is a partition.
	Similarly, for any $T$-equivariant vector bundle $\ca{V}$ over a $T$-variety $Y$, we define
	the $T$-equivariant \emph{Schur bundle} over $Y$ as:
	\[
	\bS^{\lambda}(\ca{V})
	:=
	\det(\ca{V})^{\lambda_{r}}\otimes
	\bS^{\bar{\lambda}}(\ca{V}).	
	\]
	
	For an integer sequence $\lambda=(\lambda_{1},\lambda_{2},\dots,\lambda_n)\in \Z^n$,
	we define the (generalized) Schur function using the bialternant formula 
	\begin{equation}
		\label{eq:bialternant-def}
		s_\lambda(z_1,\dots,z_n)=
		\frac{1}{
			\det (z_j^{n-i})
		}
		\begin{vmatrix}
			z_1^{\lambda_{1}+n-1}&\cdots&z_n^{\lambda_1+n-1}  \\
			z_1^{\lambda_{2}+n-2}&\cdots&z_n^{\lambda_{2}+n-2} \\
			\vdots & & \vdots \\
			z_1^{\lambda_{n}}&\cdots&z_n^{\lambda_{n}}
		\end{vmatrix}.
	\end{equation}
	It follows from the definition that $s_\lambda$ is a symmetric Laurent polynomial in $z_{i}$'s and satisfies the following
	\emph{straightening rules}:
	\begin{equation}
		\label
		{eq:straightening-rules}
		s_{(\dots,a,b,\dots)}
		=
		-
		s_{(\dots, b-1,a+1,\dots)}
		\quad\text{and}
		\quad
		s_{(\dots, a, a+1,\dots)}
		=0.
	\end{equation}
	Note that when $\lambda$ is a partition, $s_\lambda$ is the standard Schur polynomial.

	Recall that $\wev_p:\Quot \rightarrow 
	\fX$ is the stacky evaluation map at the point $p\in\bP^1$.
	For any weakly decreasing sequence (of integers) $\lambda$ with $r$ parts, we have
	\begin{equation}
		\label{eq:kcw-example}
		\wev_p^*
		(
		\bS^\lambda(\C^r)
		)
		=
		\bS^\lambda(\ca{S}^\vee_p)
		.
	\end{equation}
	Hence,
	\[
	\langle
	\bS^\lambda(\C^r)
	\rangle^{\quot,T,\ell}_{0,d}
	=
	\chi^{T}
	\left(
	\Quot,
	\det(\pi_{\star}\cS^\vee)^{-\ell}
	\cdot
	\bS^\lambda(\ca{S}^\vee_p)
	\right).
	\]
	\begin{theorem}
		\label
		{thm:Localization_calculation}
		Let $\lambda$ be a weakly decreasing sequence with $r$ parts.
		If the level $\ell$ satisfies
		$-r< \ell \leq  (N-r)$,
		we have
		\[
		\chi^{T}(\quot_d(\bP^1,N,r), 
		\det(\pi_{\star}\cS^\vee)^{-\ell}
		\cdot 
		\bS^{\lambda}(
		\cS_{p}^{\vee}
		)
		)
		=[t^d]s_{
			\lambda+(d-\ell)^r
		}(z_1,\dots,z_N),
		\]
		where $z_1,z_2,\dots,z_N$ are the roots of the equation~\eqref{eq:equivariant-bethe-ansatz}
		and $\lambda+(d-\ell)^r:=(\lambda_{1}+d-\ell,\dots \lambda_r+d-\ell)$.
	\end{theorem}				
	\begin{proof}
		By substituting $\mathrm{V}$ with the Schur functor $\bb{S}^{\lambda}(\bb{C}^r)$ in Theorem~\ref{thm:Vafa_Int_formula_equivariant}, we obtain
		\begin{equation*}
			\langle
			\bS^\lambda(\C^r)
			\rangle^{\quot,T,\ell}_{0,d}
			=
			[t^d]
			\sum_{I}
			s_\lambda(z_I)\prod_{i\in I}\frac{z_i^{d+N-r-\ell}}{
				\prod_{j\in [N];j\neq i}(z_{i}-z_j)
			}\prod_{i,j\in I;i\ne j}(z_i-z_j)
			.
		\end{equation*}
		where the index set $I$ ranges over all $r$-element subsets of $[N]$, and $z_1, \dots, z_N$ are distinct roots of
		\[
		P(z)=
		\prod_{i=1}^{N}(z-\alpha_i^{-1})
		+(-1)^rz^{N-r-\ell}t=0. 
		\]
		Here, we used the identity $P'(z_{i})=P'(z)|_{z=z_{i}}=
		\prod_{j\in [N];j\neq i}(z_{i}-z_j)
		$ to simplify the denominator.
		
		The crucial observation is that after we plug in the bialternant formula~\eqref{eq:bialternant-def} for $s_{\lambda}(z_{I})$,
		we can rewrite the above summation as a ratio of $N\times N$ determinants
		\begin{equation*}
			[t^d]\frac{
				1
			}
			{
				\det (z_j^{N-i})
			}
			\begin{vmatrix}
				z_1^{\lambda_1+N-1+d-\ell}&z_2^{\lambda_1+N-1+d-\ell}&\cdots&z_N^{\lambda_1+N-1+d-\ell}  \\
				z_1^{\lambda_2+N-2+d-\ell}&z_2^{\lambda_2+N-2+d-\ell}&\cdots&z_N^{\lambda_2+N-2+ d-\ell} \\
				\vdots&\vdots & &\vdots \\
				z_1^{\lambda_r+N-r+d-\ell}&z_2^{\lambda_r+N-r+d-\ell}&\cdots&z_N^{\lambda_r+N-r+d-\ell} \\
				z_1^{N-r-1}&z_2^{N-r-1}& \cdots &z_N^{N-r-1}\\
				z_1^{N-r-2}&z_2^{N-r-2}& \cdots &z_N^{N-r-2}\\
				\vdots&\vdots &  &\vdots \\
				z_1&z_2& \cdots &z_N\\
				1		&1&\cdots &1
			\end{vmatrix}
		\end{equation*}
		using the generalized Laplace expansion of the determinant in the numerator along the first $r$ rows.
		We conclude the proof by noting that the above expression is the bialternant formula for
		\[
		[t^d]s_{\lambda+(d-\ell)^r}(z_1,\dots,z_N).
		\]
	\end{proof}

	\begin{rem}
		If we take the non-equivariant limit in Theorem~\ref{thm:Localization_calculation},
		equation~\eqref{eq:equivariant-bethe-ansatz} becomes:
		\begin{equation*}
			(z-1)^{N}
			+(-1)^rz^{N-r-\ell}t=0.
		\end{equation*}
		Let $z_{1},z_{2},\dots, z_{N}$ be the roots of this equation.
		Then the same bialternant-type formula computes the
		\emph{non-equivariant} Euler characteristic $\chi(\quot_d(\bP^1,N,r), 
		\det(\pi_{\star}\cS^\vee)^{-\ell}
		\cdot 
		\bS^{\lambda}(
		\cS_{p}^{\vee}
		))$.

	\end{rem}

	\subsection{Exterior and symmetric powers}
	In this subsection, we focus on the case $\ell=0$
	and use Theorem~\ref{thm:Localization_calculation}
	to derive closed-form expressions for the Euler characteristics of
	the exterior and symmetric powers of $\cS_p^\vee$ over the Quot scheme.

	\begin{cor}
		\label
		{cor:wedge-powers}
		We have
		\begin{align*}
			\sum_{d=0}^{\infty}
			q^d
			\chi^{T}(\quot_d(\bP^1,N,r), \wedge^m(\cS_p^\vee))
			=\begin{cases}
				\frac{1}{1-q}
				e_{m}(\alpha_{1}^{-1},\dots,\alpha_{N}^{-1})
				& \text{if }m< r\\[5pt]
				\frac{1}{(1-q)^2}
				e_{r}(\alpha_{1}^{-1},\dots, \alpha_{N}^{-1})		
				& \text{if } m=r
			\end{cases},
		\end{align*}
		where $e_{m}$ is the $m$-th elementary symmetric polynomial.
		In the non-equivariant limit $\alpha_i \to 1$, we have
		\[
		\sum_{d=0}^{\infty}
		q^d
		\chi(\quot_d(\bP^1,N,r), \wedge^m(\cS_p^\vee))
		=\begin{cases}
			\frac{1}{1-q}
			\binom{N}{m}
			& \text{if}\ m< r\\[5pt]
			\frac{1}{(1-q)^2}
			\binom{N}{r}
			& \text{if}\ m=r
		\end{cases}.
		\]
	\end{cor}
	\begin{proof}
		Set $\lambda=(1)^{m}$ and $\ell=0$ in Theorem~\ref{thm:Localization_calculation}.
		Then $\lambda+(d)^r=(1+d,\dots,1+d,d,\dots,d)$ with $1+d$ and $d$ repeated $m$ and $r-m$ times, respectively.
		Note that by~\eqref{eq:equivariant-bethe-ansatz}, the elementary symmetric polynomials in $z_i$'s are 
		\begin{equation}
			\label
			{eq:elementary-t}
			e_{j}(z_1,\dots,z_N)=\begin{cases}
				e_{j}(\alpha_{1}^{-1},\dots, \alpha_{N}^{-1})
				& \quad \text{if}\ j\ne r\\[5pt]
				e_{r}(\alpha_{1}^{-1},\dots, \alpha_{N}^{-1})		
				+t& \quad \text{if}\ j= r
			\end{cases}.
		\end{equation}
		We use the second Jacobi--Trudi formula to express the Schur polynomial $s_{\lambda+(d)^r}(z_1,\dots,z_N)$ as a $(d+1)\times(d+1)$ determinant
		\[\begin{vmatrix}
			e_{r}& e_{r+1}&e_{r+2}&\cdots& e_{r+d-1}&e_{r+d}\\
			e_{r-1}&e_{r}&e_{r+1}&\cdots &e_{r+d-2}&e_{r+d-1}\\
			\vdots&\vdots&\vdots& &\vdots &\vdots\\
			e_{r-d+1}&e_{r-d+2}&e_{r-d+3}&\cdots& e_r&e_{r+1}\\
			e_{m-d}&e_{m-d+1}&e_{m-d+2}&\cdots& e_{m-1} &e_m\\
		\end{vmatrix}.
		\]
		When $m<r$, the above determinant is a polynomial in $t$ of degree $d$ with leading coefficient $e_m=
		e_{m}(\alpha_{1}^{-1},\dots,\alpha_{N}^{-1})
		$.
		When $m=r$, it is a polynomial in $t$ of degree $d+1$, with the coefficient of $t^{d}$ being $
		(d+1)
		e_{r}(\alpha_{1}^{-1},\dots, \alpha_{N}^{-1})
		$.
		We conclude that for any $d\geq 0$:
		\[
		\chi^{T}(\quot_d(\bP^1,N,r), \wedge^m(\cS_p^\vee))
		=\begin{cases}
			e_{m}(\alpha_{1}^{-1},\dots,\alpha_{N}^{-1})
			&  \text{if}\ m< r\\[5pt]
			(d+1)
			e_{r}(\alpha_{1}^{-1},\dots, \alpha_{N}^{-1})			
			& \text{if}\ m=r
		\end{cases}
		.
		\]
		A simple calculation implies the formulas of the generating series.
		The formulas in the non-equivariant limit follow from the fact that
		$e_i(1,\dots,1)=\binom{N}{i}$.
	\end{proof}
	\begin{cor}
		\label
		{cor:symmetric-power}
		For $m<r$, we have
		\begin{align*}
			\sum_{d=0}^{\infty}q^d
			\chi^{T}
			(\quot_d(\bP^1,N,r), \Sym^m(\cS_p^\vee))
			=
			\frac{1}{1-q}
			h_{m}(\alpha_{1}^{-1},\dots,\alpha_{N}^{-1}),
		\end{align*}
		where $h_{m}$ is the $m$-th complete homogeneous symmetric polynomial.
		For $m=r$, the generating series is equal to
		\[
		\frac{1 }{1-q}h_{r}(\alpha_{1}^{-1},\dots,\alpha_{N}^{-1})
		+\frac{(-1)^{r+1}q}{(1-q)^{2}}e_{r}(\alpha_{1}^{-1},\dots,\alpha_{N}^{-1}).
		\]
		In the non-equivariant limit $\alpha_i\to 1$, we have
		\[
		\sum_{d=0}^{\infty}q^d
		\chi
		(\quot_d(\bP^1,N,r), \Sym^m(\cS_p^\vee))
		=\begin{cases}
			\frac{1}{1-q}
			\binom{N+m-1}{m}
			& \text{if}\ m< r\\[5pt]
			\frac{1}{1-q}
			\binom{N+r-1}{r}
			+
			(-1)^{r+1}
			\frac{q}{(1-q)^2}
			\binom{N}{r}
			& \text{if}\ m=r
		\end{cases}.
		\]
		
	\end{cor}
	\begin{proof}
		We only need to show that if $m<r$, we have
		$$
		\chi^{T}
		(\quot_d(\bP^1,N,r), \Sym^m(\cS_p^\vee))
		=
		h_{m}(\alpha_{1}^{-1},\dots,\alpha_{N}^{-1})
		$$
		for any $d\geq 0$, and
		$$
		\chi^{T}
		(\quot_d(\bP^1,N,r), \Sym^r(\cS_p^\vee))
		=\begin{cases}
			h_{r}(\alpha_{1}^{-1},\dots,\alpha_{N}^{-1})
			& \quad \text{if}\ d=0
			\\[5pt]
			h_{r}(\alpha_{1}^{-1},\dots,\alpha_{N}^{-1})
			+d\cdot e_r(\alpha_{1}^{-1},\dots,\alpha_{N}^{-1})
			& \quad\text{if}\ d>0
		\end{cases}.
		$$
		Set $\lambda=(m)$ and $\ell=0$ in Theorem~\ref{thm:Localization_calculation}.
		Then $\lambda+(d)^r=(m+d,d,\dots,d,0,\dots,0)$ with $d$ repeated $r-1$ times.
		Using the second Jacobi--Trudi formula, we express $s_{\lambda+(d)^r}(z_1,\dots,z_N)$ as
		\begin{equation}
			\label
			{eq:sym-determinant}
			\left|
			\begin{array}{ccccc|ccccc}
				e_{r}& e_{r+1}&\cdots&  \cdots& e_{r+d-1}&  e_{r+d} & \cdots &\cdots &\cdots &\cdots \\
				e_{r-1}&e_{r}&\cdots &  \cdots &e_{r+d-2}&   e_{r+d-1} & \ddots  & & &\vdots \\
				\vdots&\vdots& \ddots& &\vdots & \vdots & & \ddots& &\vdots \\
				\vdots&\vdots& & \ddots &\vdots &\vdots &  & & \ddots&\vdots \\
				e_{r-d+1}&e_{r-d+2}&\cdots& \cdots & e_r& e_{r+1} & \cdots &\cdots &\cdots &\cdots \\
				\hline
				0&0&\cdots& 0 &1 &e_1 & e_{2} & \cdots &\cdots & e_{m} \\
				0 & \ddots &\cdots &\cdots & 0 & 1 &e_{1} &\cdots &\cdots & e_{m-1}
				\\
				\vdots&\vdots&\ddots &&\vdots & \vdots & \vdots &\ddots &&\vdots \\
				\vdots&\vdots& &\ddots &\vdots &0 & \vdots &&\ddots &\vdots \\
				0 &  0 &\cdots& \cdots & 0& 0 & 0 &\cdots &1 &e_{1}
			\end{array}
			\right|
			.
		\end{equation}
		The upper-left square block is of size $d\times d$ and the bottom-right one has size $m\times m$.
		By~\eqref{eq:elementary-t}, we see that when $m<r$, $t$ only occurs on the diagonal of the upper-left block
		and when $m=r$, $t$ also appears in the upper-right corner of
		the bottom-right square block.
		
		Consider the Laplace expansion of~\eqref{eq:sym-determinant} along the first $d$ rows.
		In both cases, $m<r$ or $m=r$,
		there is only one term in the expansion of degree $\geq d$, given by the product of the
		determinants of the upper-left and bottom-right blocks.
		Note that the coefficients of $t^{d}$ and $t^{d-1}$ in the determinant of the upper-left block
		are $1$ and $d\cdot e_{r}(\alpha_{1}^{-1},\dots, \alpha_{N}^{-1})$, respectively,
		and the determinant of the bottom-right block is
		$$
		h_{m}(\alpha_{1}^{-1},\dots,\alpha_{N}^{-1})
		+(-1)^{r+1}\delta_{r,m} t
		.
		$$
		By taking the coefficient of $t^d$ in the product of these two determinants,
		we obtain the formulas of the Euler characteristics at the beginning of the proof.
		
	\end{proof}

	\subsection{Vanishing results}
	\label{subsec:vanishing}
	In this subsection, we establish vanishing results for the $K$-theoretic Quot scheme invariants of
	$\bS^\lambda((\C^r)^\vee)$.
	By~\eqref{eq:kcw-example},
	we have
	\[
	\langle
	\bS^{\lambda}((\C^r)^\vee)
	\rangle^{\quot,T,\ell}_{0,d}
	=
	\chi^{T}(\quot_d(\bP^1,N,r), 
	\det(\pi_{\star}\cS^\vee)^{-\ell}
	\cdot 
	\bS^{\lambda}(
	\cS_{p}
	)
	).
	\]
	For a sequence $\nu=(\nu_1,\dots,\nu_r)$,
	we define
	$$
	\tilde{\nu}:=(-\nu_r,\dots,-\nu_1)
	\quad\text{and}\quad
	\tilde{\nu}+(d-\ell)^r
	:=
	(d-\ell-\nu_{r},\dots,d-\ell -\nu_{1}).
	$$
	Then
	\[
	\bS^{\lambda}((\C^r)^\vee)
	=
	\bS^{\tilde{\lambda}}(\C^r),
	\]
	and Theorem~\ref{thm:Localization_calculation} implies the following:
	\begin{cor}
		\label
		{cor:EulerCharShurBundle}
		Let $\lambda$ be a weakly decreasing sequence with $r$ parts.
		If the level $\ell$ satisfies $-r< \ell \leq  (N-r)$, we have
		\[
		\langle
		\bS^{\lambda}((\C^r)^\vee)
		\rangle^{\quot,T,\ell}_{0,d}
		=
		\chi^{T}(\quot_d(\bP^1,N,r), \det(\pi_{\star}\cS^\vee)^{-\ell}
		\cdot 
		\bS^\lambda(\cS_p)
		)
		=[t^d]
		s_{
			\tilde{\lambda}+(d-\ell)^r
		}(z_1,\dots,z_N)
		\]
		where $z_1,z_2,\dots,z_N$ are the roots of the equation~\eqref{eq:equivariant-bethe-ansatz}.
	\end{cor}
	
	For the rest of this subsection, we assume $\ell=0$ and
	begin with a preliminary vanishing result.

	\begin{prop}
		\label
		{cor:easy-vanishing}
		Let $\lambda=(\lambda_{1},\dots,\lambda_{r})$ be a partition (with nonnegative parts).
		If $\lambda$ is nonempty and
		$d\geq \lambda_{1} - (N-r)$,
		then
		\begin{align*}
			\langle
			\bS^{\lambda}(
			(\C^r)^\vee
			)
			\rangle
			^{\quot,T}_{0,d}
			=
			\chi^{T}(\quot_d(\bP^1,N,r), \bS^\lambda(\cS_p))
			=0.
		\end{align*} 
	\end{prop}
	\begin{proof}
		We apply Corollary~\ref{cor:EulerCharShurBundle} in the case $\ell=0$.
		Recall $
		\tilde{\lambda}+(d)^r
		=(d-\lambda_r,\dots,d-\lambda_{1})$.
		When $-(N-r)\le d-\lambda_{1}<0$, the $r$-th row of the determinant in the numerator of the bialternant formula (see~\eqref{eq:bialternant-def}) for $s_{\tilde{\lambda}+(d)^r}$ has exponents $0\le d-\lambda_{1} +N-r\le N-r-1$.
		Thus, the $r$-th row equals one of the last $(N-r)$ rows, and therefore, the determinant
		equals zero.
		
		When $d-\lambda_{1}\ge 0$, $\tilde{\lambda}+(d)^r$ is an integer partition strictly contained (since $\lambda$ is not trivial) in the rectangular partition $(d)^{r}$.
		Note that the highest power of $e_r$ appearing in the second Jacobi--Trudi expansion of $s_{\tilde{\lambda}+(d)^r}$ in terms of elementary symmetric polynomials
		is strictly less than $d$ (cf. the proof of Corollary~\ref{cor:wedge-powers}).
		Thus, $s_{\tilde{\lambda}+(d)^r}(z_1,\dots,z_N)$ is a polynomial in $t$ of degree at most $d-1$,
		and therefore, its coefficient of $t^{d}$ is 0.
	\end{proof}

	Let $\P_{r,g}$ be the set of partitions contained within the rectangle $(g)^r$.
	The preliminary vanishing result in Proposition~\ref{cor:easy-vanishing}
	leads to the following interesting result in higher genus.
	\begin{cor}\label{cor:vanishin_higher_genus}
		Let $\lambda\in\P_{r,g}$ and  $d\geq g+(2g-1)(N-r)$. Then
		\begin{align*}
			\langle \bS^{\lambda}( \mathrm{S}^\vee)\rangle_{g,d}^{\quot}=\chi^{\vir}(\QuotC, \bS^{\lambda}\left(\cS_{p}^{\vee}\right))= \begin{cases}
				\binom{N}{r}^{g}& \lambda = (g)^r,\\
				0& \text{otherwise}.
			\end{cases}
		\end{align*}
		In particular, when $d\geq g+(2g-1)(N-r)$,
		\begin{align*}
			\langle \Sym^{m}\left( \mathrm{S^{\vee}}\right)\rangle_{g,d}^{\quot} = 0\quad\text{and}\quad
			\langle \left(\det \mathrm{S^{\vee}}\right)^{m}\rangle_{g,d}^{\quot} = 0
		\end{align*}
		for any $m<g$.
	\end{cor}
	\begin{proof}
		Using the genus-induction formula from Section~\ref{subsec:VI-in-all-genera},
		we have
		\begin{align*}
			\langle \bS^{\lambda}( \mathrm{S}^\vee)\rangle_{g,d}^{\quot} =
			\langle \bS^{\lambda}( \mathrm{S}^\vee)\cdot \rH^{g} \rangle_{0,d}^{\quot}.
		\end{align*}
		By definition,
		$\rH = (\det\mathrm{S})\cdot \mathrm{U}$, where $\mathrm{U}$
		is a unique linear combination of $\bS^{\nu}(\mathrm{S})$,
		with $\nu \in \P_{r,2(N-r)}$.
		Lemma~\ref{lem:leading-term-one} shows that the coefficient of
		$1 = \bS^{\emptyset}(\rS)$ in $\mathrm{U}$ is $\binom{N}{r}$.
		Thus,
		$$
		\bS^\lambda(\mathrm{S}^\vee)\cdot \mathrm{H}^{g}  = \bS^{\lambda}(\mathrm{S}^\vee)\cdot (\det\mathrm{S})^g\cdot \mathrm{U}^g =\bS^{(g)^r-(\lambda_r,\dots,\lambda_1)}(\mathrm{S})\cdot \mathrm{U}^g
		$$
		can be written as a linear combination of
		$\bS^{\rho}(\mathrm{S})$ with $\rho \in \P_{r,2g(N-r)+g}$.
		The term $1 = \bS^{\emptyset}(\rS)$ appears in this expansion only
		when $\lambda=(g)^r$, with coefficient $\binom{N}{r}^g$.
		Proposition~\ref{cor:easy-vanishing} implies
		that for $d\ge g+(2g-1)(N-r)$, the 1-pointed genus zero invariants
		for all non-empty partitions in the expansion vanish. 
		When $\lambda = (g)^r$, the 1-pointed invariant $\langle 1 \rangle_{0,d}^{\quot} = 1$
		by Corollary\ref{cor:G-lambda-equals-1}, completing the proof.
	\end{proof}

	Now we state the following stronger vanishing result.
	\begin{theorem}
		\label
		{prop:Vanshing_rparts}
		Suppose $r<N$ and $d>0$. For any partition $\lambda=(\lambda_{1},\dots,\lambda_{r})$,
		we have
		\begin{align*}
			\langle
			\bS^{\lambda}(
			(\C^r)^\vee
			)
			\rangle
			^{\quot,T}_{0,d}
			=
			\chi^{T}(\quot_d(\bP^1,N,r), \bS^\lambda(\cS_p))=0
		\end{align*} 
		if $d\geq \lambda_{1} -2(N-r)$ and one of the following two conditions holds:
		\begin{enumerate}[\normalfont(i)]
			\item $\lambda_r>0$;
			\item $d \geq r$.
		\end{enumerate}
	\end{theorem}

	\begin{proof}
		Let $z_1,\dots ,z_N$ be the roots of the equation $\prod_{i=1}^{N}(z-\alpha_{i}^{-1})+(-1)^rz^{N-r}t=0$. Note that $
		\prod_{i=1}^{N}z_i
		=
		\prod_{i=1}^{N}\alpha_i^{-1}
		$
		since $r<N$.
		Using Corollary~\ref{cor:EulerCharShurBundle} and the defining equation~\eqref{eq:bialternant-def},
		we obtain,
		\begin{align*}
			\bigg(
			\prod_{i=1}^{N}\alpha_i^{-1}
			\bigg)^{N-r}
			\chi^{T}(\quot_d(\bP^1,N,r), \bS^\lambda(\cS_p))
			&=[t^d]\bigg(\prod_{i=1}^{N}z_i\bigg)^{N-r} s_{
				\tilde{\lambda}+(d)^r
			}(z_1,\dots ,z_N)\\
			&
			=[t^d] s_{
				\tilde{\lambda}+(d)^r+(N-r)^{N}
			}(z_1,\dots ,z_N).	
		\end{align*}
		where $
		\tilde{\lambda}+(d)^r+(N-r)^{N}
		=(d-\lambda_{r}+N-r,\dots, d-\lambda_{1}+N-r,N-r\dots,N-r)$ with $N-r$ repeated $N-r$ times at the end.
		
		Set $\widetilde{\Lambda}:=\tilde{\lambda}+(d)^r+(N-r)^{N}$.
		Note that
		$
		\widetilde{\Lambda}
		$ may not be weakly decreasing.
		To use the second Jacobi--Trudi formula, we need to apply the straightening
		rules~\eqref{eq:straightening-rules} repeatedly to express its Schur
		function in terms of that indexed by a partition.
		Equivalently, we reorder the rows of the determinant in the numerator of
		its bialternant formula (in~\eqref{eq:bialternant-def})
		such that the corresponding exponents (of $z_{i}$) are in decreasing order
		(we may assume that the exponents of all rows are distinct, otherwise the determinant vanishes
		as desired).
		More precisely, let $0\le m \le r$ be the largest index such that the exponent in the $(r-m+1)$-th row is less than the exponent appearing in the last row of the determinant in the bialternant formula for $s_{\widetilde{\Lambda}}(z_1,\dots ,z_N)$.
		Here, we set $m=0$ if the exponent in the $r$-th row, given by $2(N-r)+d-\lambda_{1}$,
		is greater than or equal to that in the last row, given by $N-r$.
		In this case, we have $\lambda_{1}\leq d+N-r$, and the vanishing follows from Proposition~\ref{cor:easy-vanishing}.
		Now we assume $m\geq 1$.
		After reordering the rows by placing the last $N-r$ rows above the $(r-m+1)$-th, we obtain that 
		\[ s_{
			\widetilde{\Lambda}
		}(z_1,\dots ,z_N)=(-1)^{(N-r)m} s_{\nu}(z_1,\dots,z_N), \]
		where $\nu=(\nu_1,\dots,\nu_N)$ is the weakly decreasing sequence given by
		\begin{align*}
			\nu_i=\begin{cases}
				(N-r)+d-\lambda_{r+1-i} &\text{when}\  1\le i\le r-m \\
				N-r-m &\text{when}\ r-m< i\le N-m\\
				2(N-r)+d-\lambda_{N+1-i} &\text{when}\ N-m <i\le N\\
			\end{cases}.
		\end{align*}
		
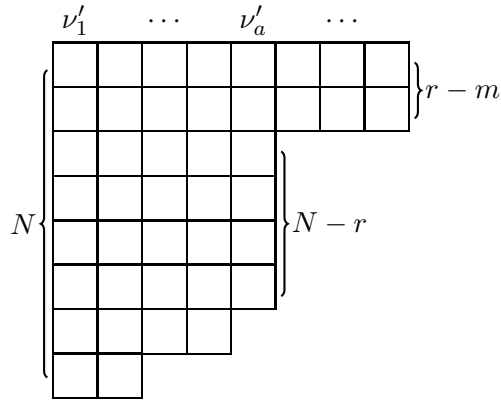
		\[	\begin{ytableau}
			\none[\nu'_1]&  \none & \none[\cdots] & \none& \none[\nu'_a]&\none&\none[\cdots]&\none\\
			\tikznode{a3}{~}&  ~ & ~ & ~& ~&~&~&\tikznode{a5}{~}\\
			~&~ & ~ & ~& ~&~&~&\tikznode{a6}{~}\\\
			~& ~ & ~ & ~& \tikznode{a1}{~}&\none&\none&\none\\
			~&  ~ & ~ & ~& 	~&\none&\none&\none\\
			~&  ~ & ~ & ~& ~&\none&\none&\none\\
			~& ~ & ~ & ~& \tikznode{a2}{~}&\none&\none&\none\\
			~& ~ & ~ & ~& \none&\none&\none&\none\\
			\tikznode{a4}{~}&  ~ & \none& \none& \none&\none&\none&\none\\
		\end{ytableau}\]
		\tikz[overlay,remember picture]{%
			\draw[decorate,decoration={brace},thick] ([yshift=1mm,xshift=3mm]a1.north east) -- 
			([yshift=0mm,xshift=3mm]a2.south east) node[midway,right]{$N-r$};
			\draw[decorate,decoration={brace},thick] ([yshift=1mm,xshift=-3mm]a4.south west) -- 
			([yshift=0mm,xshift=-3mm]a3.north west) node[midway,left]{$N$};
			\draw[decorate,decoration={brace},thick] ([yshift=1mm,xshift=3mm]a5.north east) -- 
			([yshift=0mm,xshift=3mm]a6.south east) node[midway,right]{$r-m$};
		}
		
		\noindent The sequence $\nu$ is a partition due to the assumption $\lambda_{1}\le d+2(N-r)$.
		Here are a few important properties of the partition $\nu$ that we will use.
		Let $a=N-r-m$ and $\nu'$ denote the partition conjugate to $\nu$. Then 
		\begin{align*}
			\nu_a'\ge N-m
			\quad
			\text{and}
			\quad
			\nu_{a+1}'\le r-m
			.
		\end{align*}
		Moreover, the first part of $\nu$ is given by $$\nu_1=\begin{cases}
			(N-r+d)-\lambda_r &\quad\text{if}\ m<r\\
			N-2r&\quad\text{if}\ m =r
		\end{cases}.$$

		For notational convenience, let $M=\nu_1$.
		As before, we now use the second Jacobi--Trudi formula to express $s_\nu(z_1,\dots,z_N)$
		as the following $M\times M$ determinant
		\begin{equation}
			\label
			{main-theorem-det}
			\begin{vmatrix}
				e_{\nu_1'}& e_{\nu_1'+1}&\cdots&\cdots&e_{\nu'_1+a}& \cdots&e_{\nu'_1+M-1}\\
				e_{\nu'_2-1}&e_{\nu'_2}&\cdots&\cdots&\vdots &&\vdots\\
				\vdots&\vdots&\ddots&\vdots&\vdots &&\vdots\\
				e_{\nu'_a-a+1} & 	e_{\nu'_a-a+2} &\cdots&	e_{\nu'_a}&\vdots && \vdots\\
				\vdots&&&& e_{\nu'_{a+1}}&\cdots &\vdots\\
				\vdots&&&&\vdots& \ddots &\vdots\\
				e_{\nu'_M-M+1}&\cdots&\cdots&\cdots&\cdots& \cdots &e_{\nu'_{M}}\\
			\end{vmatrix}.
		\end{equation}
		Here the elementary symmetric polynomials $e_{j}$'s are given by~\eqref{eq:elementary-t}.

		Suppose condition (i) holds, i.e., $\lambda_r>0$. Then the first part $\nu_1\le N-r+d-1 .$
		We claim that $e_r$ is not present as an entry in the first $N-r$ columns of~\eqref{main-theorem-det},
		and therefore, the highest power of $e_r$ in its expansion is at most  $\nu_1-(N-r)\le d-1$. Thus the above determinant is a polynomial in $t$ of degree at most $d-1$; therefore, we have the vanishing result $[t^d]s_\nu(z_1,\dots,z_N)=0$. 
		
		To see the claim, first note that the first $a$ rows do not contain $e_{r}$ because
		$$\nu'_{a}-a+1\ge r+1.$$
		Furthermore, since $\nu'_{a+1}\le r-m$ or, equivalently, $\nu'_{a+1}+m-1\le r-1$,
		the submatrix determined by the last $M-a$ rows and the first $a+m=N-r$ columns does not contain $e_{r}$ either.
		Hence, the first $N-r$ columns do not contain $e_r$ as an entry.

		Now suppose $\lambda_{r}=0$ but condition (ii) holds, i.e., $d\geq r$.
		In this case, $t^{d}$ can only show up in the determinant~\eqref{main-theorem-det}
		if all the diagonal entries of the $d\times d$ submatrix, lying in the intersection of the
		$(a+1)$-th, $(a+2)$-th, ..., $(a+d)$-th rows and the last $d$ columns, are $e_{r}$.
		Note that there is no $e_{r}$ outside of this submatrix. The assumption $d\geq r$ implies that
		the entries in the intersection of the $(a+d+1)$-th row and the first $N-r$ columns are all zero.
		Hence the complement minor of this $d\times d$ submatrix is zero and, therefore, the vanishing of
		determinant follows from the generalized Laplace expansion along the last $d$ columns.
	\end{proof}

	\subsection{Tautological class}
	In this subsection, we present a formula related to the results in~\cite{OpreaShubham}.
	Let $M\to \mathbb{P}^1$ be a line bundle.
	Let $\pi_{\bP^1}:\Quot\times\bP^{1}\ra\bP^{1}$ be the projection onto the second factor and
	let $\ca{Q}$ be the tautological quotient sheaf on the universal curve $\pi:\Quot\times\bP^{1}\ra\Quot$.
	The Euler characteristics of the exterior powers of the tautological class,
	defined by $$
	M^{[d]}:=
	\pi_{\star} \left(\pi_{\bP^1}^{*} M\cdot \mathcal Q\right),
	$$ were computed in~\cite{OpreaShubham}.
	\begin{lem}
		Let $M=\cO_{\mathbb{P}^1}(m)$. Then over $\quot_{d}(\bb{P}^1,N,r)$, we have
		\[
		\det M^{[d]}=\det(\pi_{\star} \cS^\vee)^{-1}\cdot \det(\cS_p^\vee)^{m+2}.  
		\]
	\end{lem}
	\begin{proof}
		Tensoring the universal short exact sequence~\eqref{eq:universal-seq-quot-scheme}
		with $\O_{\bP^{1}}(m)$, we find
		\[
		\det(\pi_{\star}\ca{Q}(m))
		=
		\det(\pi_{\star}\ca{S}(m))^{-1}.
		\]
		Using the isomorphism $\O_{\bP^{1}}(1)\cong\ca{O}(x)$, we obtain
		\[
		0\ra
		\cS(m-1)
		\ra
		\cS(m)
		\ra
		\cS_p
		\ra
		0.
		\]
		This leads to $\det(\pi_{\star}\ca{S}(m-1))\cdot \det(\cS_{p})=
		\det(\pi_{\star}\ca{S}(m))
		$.
		By induction:
		$$
		\det(\pi_{\star}\ca{S}(m))=\det(\cS_{p})^{m+2}
		\cdot
		\det(\pi_{\star}\ca{S}(-2))
		.
		$$
		Finally, noting that $
		\det(\pi_{\star}\ca{S}(-2))
		=
		\det(\pi_{\star}\ca{S}^{\vee})$ due to the Serre (or Grothendieck) duality,
		we conclude the proof.
	\end{proof}
	In fact, it follows from the results in~\cite{Stromme} that the Picard group of $\Quot$ is generated by $\det(\pi_{\star}\ca{S}^{\vee})$ and $\det(\cS_{p})$. Theorem~\ref{thm:Localization_calculation} enables us to calculate the Euler characteristics of the determinant of the tautological class and its powers, $(\det M^{[d]} )^\ell $, when $-r<\ell\le (N-r)$.
	
	\begin{cor}
		Let $M$ be a line bundle over $\mathbb{P}^1$ of degree $m$.
		If the level $\ell$ satisfies $-r< \ell\leq N-r$,
		then 
		\begin{align*}
			\chi^{T}
			(\quot_{d}(\bb{P}^1,N,r),(\det M^{[d]})^\ell )= [t^d]
			s_{((m+1)\ell+d)^r}(z_1,\dots,z_N),
		\end{align*}
		where
		$z_1,\dots, z_N$ are the roots of~\eqref{eq:equivariant-bethe-ansatz}.
	\end{cor}

	\section{Applications to quantum $K$-theory}
	\label
	{sec:applications}

	In this section, we assume $\ell=0$.
	We first develop a bialternant-type formula for the 1-pointed invariants of $\rO_\lambda$.
	Then we recall the quantum reduction morphism defined in~\cite{SinhaZhang},
	linking 1-pointed invariants to the quantum $K$-ring of Grassmannian.
	The final subsection focuses on the rank-2 case ($X=\mathrm{Gr}(2,N)$),
	detailing explicit formulas for quantum $K$-products,
	1-pointed invariants, the quantum reduction map, and both the quantized pairing and its inverse.
	
	\subsection{Invariants of $\rO_\lambda$}
	Recall from Section~\ref{sec:K-theory-Grassmannian}
	that the representation $\rO_\lambda$ is defined for any partition $\lambda$, with its associated $K$-theory
	class being the Schubert structure sheaf $\O_\lambda$ for $\lambda\in\prk$. Using
	Theorem~\ref{thm:Vafa_Int_formula_equivariant},
	we derive a bialternant-type formula for $\langle
	\rO_\lambda
	\rangle^{\quot,T}_{0, d}$.
	
	Before presenting the formula, we introduce a convenient generalization of the
	Schur functions.
	For any tuple of Laurent polynomials $(f_1(z),\dots, f_r(z))$, we define
	\begin{equation}
		\label
		{eq:generalized-schur-function}
		s_{(f_1,\dots, f_r)} (z_1,\dots,z_N)= \frac{1}{\det (z_j^{N-i})}\begin{vmatrix}
			f_1(z_1)z_1^{N-1}&\cdots&f_1(z_N)z_N^{N-1}  \\
			f_2(z_1)z_1^{N-2}&\cdots&f_2(z_N)z_N^{N-2} \\
			\vdots &  & \vdots \\
			f_r(z_1)z_1^{N-r}&\cdots&f_r(z_N)z_N^{N-r} \\
			z_1^{N-r-1}& \cdots &  z_N^{N-r-1}\\
			z_1^{N-r-2}& \cdots &  z_N^{N-r-2}\\
			\vdots & &\vdots \\
			1		&\cdots &1
		\end{vmatrix}.
	\end{equation}
	Note that when $f_i(z)=z^{\lambda_i}$, it recovers the definition of the Schur
	functions in~\eqref{eq:bialternant-def}, i.e., $s_{(z^{\lambda_1},\dots,z^{\lambda_r})}
	=s_{(\lambda_1,\dots,\lambda_r)}$.
	
	\begin{theorem} 
		\label
		{cor:G-lambda}
		For a partition $\lambda=(\lambda_{1},\dots,\lambda_{r})$,
		the 1-pointed invariant is computed by
		\begin{equation*}
			\langle
			\rO_\lambda
			\rangle^{\quot,T}_{0,d}
			=
			\chi^{T}
			\left(
			\Quot,
			\bG^{\lambda}
			(
			\cS_{p}^{\vee}
			)
			\right)
			=[t^d]s_{(f_1,\dots, f_r)} (z_1,\dots,z_N)
		\end{equation*}
		where 
		\[f_i(z)=
		z^d\bigg(1-\frac{1}{z}\bigg)^{\lambda_i+r-i}
		,\]
		and $z_1,z_2,\dots,z_N$ are the roots of:
		\begin{equation}
			\label
			{eq:level-0-bethe-equation_Grothendieck}
			\prod_{i=1}^{N}(z-\alpha_i^{-1})+(-1)^rz^{N-r}t=0
			.
		\end{equation}
		Similarly, the non-equivariant invariant $	\langle
		\rO_\lambda
		\rangle^{\quot}_{0,d}$ is given by the same formula
		with $z_1,z_2,\dots,z_N$ as roots of
		$(z-1)^N+(-1)^rz^{N-r}t=0$.
	\end{theorem}

	\begin{proof}
		Recall that the character of $\rm{O}_\lambda$ is given by the Grothendieck polynomial
		\begin{align*}
			G_{\lambda}
			(1-x^{-1}_{1},\dots,1- x^{-1}_{r}) &=\frac{ \det((1-x_j^{-1})^{\lambda_i+r-i}x_{j}^{-i+1})_{1\le i,j\le r}}{\det\big((1-x_j^{-1})^{r-i}\big)_{1\le i,j\le r}}\\
			&= \frac{ \det((1-x_j^{-1})^{\lambda_i+r-i}x_{j}^{-i+1})_{1\le i,j\le r}}{\prod_{1\le i<j\le r}(x_j^{-1}-x_{i}^{-1})}\\
			&= \bigg(\prod_{i=1}^{r}x_i\bigg)^{r-1}\frac{ \det((1-x_j^{-1})^{\lambda_i+r-i}x_{j}^{-i+1})_{1\le i,j\le r}}{\prod_{1\le i<j\le r}(x_i-x_{j})}.
		\end{align*}
		Using Theorem~\ref{thm:Vafa_Int_formula_equivariant}, we obtain
		\begin{align*}
			\langle
			\rO_\lambda
			\rangle^{\quot,T}_{0,d}&= [t^d]
			\sum_{z_1,\dots,z_r}
			G_\lambda(1-z_1^{-1},\dots,1-z_r^{-1})\frac{\prod_{i=1}^{r}z_i^{N-r+d}\cdot\prod_{i\neq j}(z_i-z_j)}{\prod_{i=1}^{r}P'(z_i)}\\
			&=[t^d]
			\sum_{z_1,\dots,z_r}
			\frac{\prod_{i=1}^{r}z_i^{N-1+d}\cdot \det((1-z_j^{-1})^{\lambda_i+r-i}z_{j}^{-i+1})_{1\le i,j\le r}\cdot\prod_{i\neq j}(z_i-z_j)}{\prod_{1\le i<j\le r}(z_i-z_{j})\cdot \prod_{i=1}^{r}P'(z_i)}\\
			&=[t^d]
			\sum_{z_1,\dots,z_r}
			\frac{\det((1-z_j^{-1})^{\lambda_i+r-i}z_{j}^{N+d-i})_{1\le i,j\le r}\cdot\prod_{i\neq j}(z_i-z_j)}{\prod_{1\le i<j\le r}(z_i-z_{j})\prod_{i=1}^{r}P'(z_i)},
		\end{align*}
		where $z_1,\dots,z_r$ run over $\binom{N}{r}$ subsets of $N$ distinct roots of \eqref{eq:level-0-bethe-equation_Grothendieck}.
		The theorem now follows by applying the generalized Laplace expansion along the first $r$ rows to the bialternant formula in~\eqref{eq:generalized-schur-function}, as done in the proof of Theorem~\ref{thm:Localization_calculation}.
	\end{proof}
	For a partition $\nu=(\nu_1,\dots,\nu_r)$,
	it follows from Lemma~\ref{lem:leading-term-one} that
	\begin{equation*}
		G_{\lambda}
		(1-x^{-1}_{1},\dots,1- x^{-1}_{r})
		=
		\sum_{\nu\subseteq (\lambda_1)^r} D^{\nu}_{\lambda}\,
		s_{\tilde{\nu}}(x_{1},\dots,x_{r})
	\end{equation*}
	and
	\begin{equation*}
		\rO_\lambda
		=
		\sum_{\nu\subseteq (\lambda_1)^r} D^{\nu}_{\lambda}\,
		\bS^{\tilde{\nu}}(\C^r),	
	\end{equation*}
	where $D^{\nu}_{\lambda}\in\Z$ and $\tilde{\nu}=(-\nu_r,\dots,-\nu_1)$.
	The vanishing results for Schur bundles show that the $K$-theoretic Quot scheme invariant of $\rO_\lambda$
	stabilizes to 1 when the degree is sufficiently large.
	\begin{corollary}
		\label
		{cor:G-lambda-equals-1}
		We have
		\[
		\langle
		\rO_\lambda
		\rangle^{\quot,T}_{0,d}
		=
		\chi^{T}
		\left(
		\Quot,
		\bG^{\lambda}
		(
		\cS_{p}^{\vee}
		)
		\right)
		=1,
		\]
		if one of the following two conditions hold:
		\begin{enumerate}[\normalfont(i)]
			\item $d\geq \lambda_1-(N-r)$;
			\item $d \geq r$ and $\lambda_1\leq 2(N-r)+r$.
		\end{enumerate}
	\end{corollary}
	\begin{proof}
		By Lemma~\ref{lem:leading-term-one}, we have 
		\[
		\bG^{\lambda}(\cS_{x}^{\vee})
		=
		\ca{O}_{\Quot}
		+
		\sum_{\nu\in \mathrm{P}_{r,d+N-r};\nu\neq\emptyset} D^{\nu}_\lambda
		\,\bS^{\nu}(\cS_{p})
		.
		\]
		It follows that
		\begin{align*}
			\chi
			\left(
			\Quot,
			\bG^{\lambda}
			(
			\cS_{p}^{\vee}
			)
			\right)
			=&
			\chi
			\left(
			\Quot,
			\O_{\Quot}
			\right)\\
			&+
			\sum_{\substack{\nu\in \mathrm{P}_{r,d+N-r}\\ \nu\neq\emptyset}} D^{\nu}_\lambda
			\chi
			\left(
			\Quot,
			\bS^{\nu}
			(
			\cS_{p}
			)
			\right).
		\end{align*}
		Note that the first term on the right-hand side equals 1;
		we can prove it by either using the fact that the Quot scheme is rational (c.f.~\cite[Remark (3.4)]{Stromme})
		or by setting $m=0$ in Corollary~\ref{cor:wedge-powers}.
		We conclude the proof by noting that the remaining terms in the summation vanish
		by Proposition~\ref{cor:easy-vanishing} for case (i)
		and Theorem~\ref{prop:Vanshing_rparts} for case (ii).
	\end{proof}

	For any partition $\lambda$ with $r$ parts, let
	\[
	\langle\!\langle
	\rO_\lambda
	\rangle\!\rangle
	^{\quot,T}_{0}
	=
	\sum_{d=0}^{\infty}	q^d\chi^{T}
	\left(\quot_d(\bP^1,N,r), 
	\bG^{\lambda}(\cS_p^{\vee})
	\right)
	\]
	be the correlation function of $\rO_\lambda$.
	Here $q$ is the Novikov variable. We prove some explicit formulas for the above generating function for $\mathrm{Gr}(2,N)$ in Proposition~\ref{prop:rank2_1-pointed_invariants}. The stabilization results from Corollary~\ref{cor:G-lambda-equals-1}
	yield the following descriptions of this correlation function.
	\begin{cor}
		\label
		{rem:G-correlation-function}
		\begin{enumerate}[\normalfont(i)]
			\item
			For any partition $\lambda$, we have
			\[
			\langle\!\langle
			\rO_\lambda
			\rangle\!\rangle
			^{\quot,T}_{0}
			=
			P_\lambda(q)
			+
			\frac{q^{d_{\lambda}}}{1-q},
			\]
			where $d_{\lambda}=\max\{0, \lambda_{1}-k \}$ with $k=N-r$ and
			$P_{\lambda}(q)$ is a polynomial in $q$ of degree less than $d_{\lambda}$.
			In particular, we have
			\[
			\langle\!\langle
			\rO_\lambda
			\rangle\!\rangle
			^{\quot,T}_{0}
			=\frac{1}{1-q}
			\]
			for any $\lambda\in\P_{r,k}$.
			
			\item For any partition $\lambda\in \P_{r,2(N-r)+r}$, we have 
			\begin{equation*}
				\langle\!\langle
				\rO_{\lambda}
				\rangle\!\rangle^{\quot, T}_{0}
				=
				\tilde{P}_{\lambda}(q)+
				\frac{q^{r}}{1-q},
			\end{equation*} 
			where $\tilde{P}_\lambda(q)$ is a polynomial in $q$ of degree $\leq r-1$.
		\end{enumerate}  
	\end{cor}

	Let 
	$\langle\!\langle
	\rO_{\lambda}
	\rangle\!\rangle^{\quot}_{0}$ be the non-equivariant correlation function, i.e., the
	non-equivariant limit of $\langle\!\langle
	\rO_{\lambda}
	\rangle\!\rangle^{\quot, T}_{0}$.
	Theorem~\ref{cor:G-lambda} implies a periodic behavior of the 1-pointed invariants,
	which is very useful in computations.	
	\begin{cor}
		\label
		{cor:periodic_Grotendieck_1-point_invariants}
		For any partition $\lambda=(\lambda_{1},\dots,\lambda_{r})$ with $\lambda_{r}\ge N$, 
		\[	\langle\!\langle
		\rO_{\lambda+(N)^{r}}	\rangle\!\rangle^{\quot}_{0}= q^r 	\langle\!\langle
		\rO_{\lambda}
		\rangle\!\rangle^{\quot}_{0}.\]
		Here $\lambda+(N)^r:=(\lambda_1+N,\dots,\lambda_r+N)$.
	\end{cor}
	\begin{proof}
		For any root $z_j$ of the polynomial $(z-1)^N+(-1)^rtz^{N-r}=0$, we have the identity $(1-z_j^{-1})^N=(-1)^{r-1}tz_j^{-r}$. This implies
		\begin{align*}
			[t^d]s_{(z^d(1-\frac{1}{z})^{\lambda_i+N+r-i})_{i=1}^{r}}(\vec{z})
			&=
			[t^{d}](-1)^{r(r-1)} t^r s_{(z^{d-r}(1-\frac{1}{z})^{\lambda_i+r-i})_{i=1}^{r}}({\vec{z}})
			\\
			&=
			[t^{d-r}]s_{(z^{d-r}(1-\frac{1}{z})^{\lambda_i+r-i})_{i=1}^{r}}({\vec{z}})
			.
		\end{align*}
		We conclude the proof by using Theorem~\ref{cor:G-lambda}
		and summing the two sides over all $d$, after multiplying with $q^d$ and $q^{d-r}$, respectively.
	\end{proof}
	
	\begin{rem}
		There is a geometric explanation of the above periodicity. Let $p\in \bP^1$, then
		\[
		\iota_p:  
		\quot_d(\bP^1,N,r)
		\rightarrow
		\quot_{d+r}(\bP^1,N,r)
		\]
		be the embedding of Quot schemes which on closed points is given by
		\[
		[S\hookrightarrow \O^{\oplus N}_{\bP^1}]
		\mapsto
		[S(-p)\hookrightarrow \O^{\oplus N}_{\bP^1}(-p)\hookrightarrow \O_{\bP^1}^{\oplus N}].
		\]
		This exhibits $\quot_d(\bP^1,N,r)$ as the zero locus of a section of $(\ca{S}_p^{\oplus N})^\vee$ on the Quot scheme $\quot_{d+r}(\bP^1,N,r)$ (c.f. the proof of~\cite[Theorem 2]{Marian-Oprea}).
		Hence
		\[
		(\iota_p)_*\O_{\quot_d(\bP^1,N,r)}=
		\lambda_{-1}(\ca{S}_p^{\oplus N})
		\cdot
		\O_{\quot_{d+r}(\bP^1,N,r)}
		=
		(
		\lambda_{-1}(\ca{S}_p))^N
		\cdot
		\O_{\quot_{d+r}(\bP^1,N,r)}.
		\]
		Using the definition of the Grothendieck polynomial~\eqref{eq:G-lambda-bialternant-formula},
		we have
		$$
		G_{\lambda+(N)^r}(x_1,\dots,x_r)
		=(x_1\cdots x_r)^N
		G_{\lambda}(x_1,\dots,x_r),$$
		and therefore,
		\[
		\bG^{\lambda+(N)^r}(\cS_{p}^\vee)
		=
		(\lambda_{-1}(\cS_{p}))^N
		\cdot
		\bG^\lambda(\cS_{p}^\vee)
		.
		\]
		The periodicity in Corollary~\ref{cor:periodic_Grotendieck_1-point_invariants} follows from
		the above two identities and the projection formula along $\iota_p$.

	\end{rem}

	\subsection{Quantum $K$-ring}
	\label
	{subsec:quantum-k-ring}
	For the rest of the paper, we focus on the non-equivariant setting and assume $k=N-r>0$.
	
	Let $X=\grass$. We first recall the definitions of Givental--Lee's quantum $K$-invariants
	and quantum $K$-ring~\cite{WDVV,Lee,Buch-Mihalcea}.
	For $d\in H_{2}(X)$, let $\ms$ be the moduli stack of degree-$d$, genus-$0$, $n$-pointed stable maps into $X$.
	Let $\ev_{i}:\ms\rightarrow X$ denote the
	evaluation map at the $i$-th marking for $1\leq i \leq n$.
	Then the genus-0, $n$-pointed
	quantum $K$-invariants is defined by
	\[
	\langle
	E_{1},
	\dots,
	E_{n}
	\rangle_{0,n,d}
	:=\chi
	\bigg(
	\ms,
	\prod_{i=1}^{n}\ev_{i}^{*}(E_{i})
	\bigg),
	\]
	where $E_{1},\dots,E_{n}\in K(X)$.
	
	Consider the $\mathbb{Z}$-basis of $K(X)$, consisting of Schubert structure sheaves $\{\mathcal{O}_\lambda\}_{\lambda \in \prkk}$.
	The \emph{quantized pairing} $(F_{\alpha,\beta})_{\alpha,\beta\in\prkk}$ with respect to
	the Schubert basis is defined by
	\begin{equation*}
		F_{\alpha,\beta}
		:=
		\langle\!\langle
		\O_{\alpha},
		\O_{\beta}
		\rangle\!\rangle_{0,2}
		=
		\chi
		(X, 
		\O_\alpha
		\cdot 
		\O_{\beta}
		)
		+
		\sum_{d>0}
		q^{d}
		\langle
		\O_\alpha,
		\O_{\beta}
		\rangle_{0,2,d}
		.
	\end{equation*}
	Let $(F^{\alpha,\beta})$ be the matrix inverse to the quantized pairing matrix
	$(F_{\alpha,\beta})$.
	We define the following generating series of the 3-pointed Givental--Lee's invariants
	\begin{equation*}
		F_{\alpha,\beta,\gamma}:
		=
		\langle\!\langle
		\O_{\alpha},
		\O_{\beta},
		\O_{\gamma}
		\rangle\!\rangle_{0,3}
		=
		\sum_{d\geq0}
		q^{d}
		\langle
		\O_{\alpha},
		\O_{\beta},
		\O_{\gamma}
		\rangle_{0,3,d}.
	\end{equation*}
	
	Given three partitions $\lambda,\mu,\nu \in \P_{r,k}$, we define the~\emph{structure constant} by
	\begin{equation*}
		F_{\lambda,\mu}^{\nu}
		:=
		\sum_{\alpha\in\P_{r,k}}
		F^{\nu,\alpha}
		F_{\alpha,\lambda,\mu}.
	\end{equation*}
	It was first shown by Givental~\cite{WDVV} that
	the \emph{quantum $K$-product}
	defined by
	\begin{equation*}
		\O_{\lambda}
		\bullet
		\O_{\mu}
		=\sum_{\nu}
		F_{\lambda,\mu}^{\nu}
		\O_{\nu}
	\end{equation*}
	is \emph{associative};
	this result is a formal consequence of the WDVV equation in quantum $K$-theory.
	The \emph{quantum $K$-ring} of $X$, denoted by $\qk(X)$, is the $\Z[[q]]$-algebra
	$K(X)\otimes_{\Z}\Z[[q]]$ equipped with the quantum $K$-product.
	Givental--Lee showed that the quantum $K$-ring $\qk(X)$ equipped with the
	quantized pairing $F_{\alpha,\beta}$
	is a (formal) commutative \emph{Frobenius algebra} with unit 1;
	see~\cite[Corollary 1]{WDVV}
	and~\cite[Proposition 12 (1)]{Lee}.

	\subsection{Quantum reduction map}
	\label
	{subsec:quantum-reduction-map}
	In Part I~\cite[Section 4]{SinhaZhang}, we show how to recover ring presentations for the quantum $K$-ring of the Grassmannian using the vanishing results in Section~\ref{subsec:vanishing}. To achieve this, we defined the quantum reduction map, which we recall below. In the remainder of the paper, we introduce a strategy for computing the quantum $K$-product of $X$ using the Littlewood--Richardson rule for multiplying Grothendieck polynomials and the 1-pointed invariants in Theorem~\ref{cor:G-lambda}.
	
	\begin{definition}
		\label{def:kappa-map}
		We define the quantum reduction map
		\[
		\kappa: 
		R(\GLr)\otimes_{\Z}\Z[[q]]
		\ra
		\qk(X)
		\]
		as the $\Z[[q]]$-linear extension of the map
		\[
		\rV
		\mapsto
		\sum_{\alpha\in\prkk}
		\langle\!\langle
		\rV
		\cdot
		\rO_{\alpha}^*
		\rangle\!\rangle^{\quot}_{0}
		\O_\alpha,
		\quad
		\rV\in 
		R(\GLr).
		\]
	\end{definition}
	In Part I~\cite[Theorem 1.9]{SinhaZhang}, we prove the following
	\begin{theorem}[\cite{SinhaZhang}]
		The quantum reduction map $\kappa$ is a surjective ring homomorphism such that for
		any $\lambda\in\prkk$, we have 
		\[
		\kappa(\rO_\lambda)=
		\O_{\lambda}	.
		\]
	\end{theorem}
	The above theorem allows us to study the quantum $K$-ring through the representation ring of $\GLr$
	and the quantum reduction map $\kappa$.
	According to~\cite{Buch}, there is a $K$-theoretic Littlewood--Richardson rule:
	\begin{equation}
		\label
		{eq:G-LR-rule}
		\rO_\lambda
		\cdot
		\rO_\mu
		=
		\sum_{\nu\in \P_{r,2k}}
		c_{\lambda,\mu}^{\nu}
		\rm{O}_\nu,
	\end{equation}
	where the coefficient $c_{\lambda,\mu}^{\nu}
	$ is equal to $(-1)^{|\nu|-|\lambda|-|\mu|}$ times the number of set-valued tableaux $T$
	of shape $\lambda*\mu$ with content $\nu$ such that $w(T)$ is a reverse lattice word.
	The definitions and notation are detailed in~\cite{Buch}, and the Littlewood--Richardson
	rule for the case $r=2$ is explicitly presented in Section~\ref{subsec:rank-2-case}
	(see~\eqref{eq:rank_2_product_grothendieck_functor}).

	Applying the quantum reduction map $\kappa$
	to the Littlewood--Richardson rule~\eqref{eq:G-LR-rule},
	we obtain
	\[
	\O_{\lambda}
	\bullet
	\O_{\mu}
	=\kappa(\rm{O}_\lambda\cdot \rm{O}_{\mu})=
	\sum_{\nu\in \P_{r,k}}
	c_{\lambda,\mu}^{\nu}
	\O_{\nu}
	+
	\sum_{\substack{\nu\in\P_{r,2k}
			\\
			\nu\notin \P_{r,k}}}
	c_{\lambda,\mu}^{\nu}
	\kappa(
	\rO_\nu
	)
	\]
	for $\lambda,\mu\in \P_{r,k}$.
	This shows that the quantum multiplications of
	Schubert structure sheaves can be fully determined using
	Littlewood--Richardson coefficients and
	the images of
	$\rO_\nu,\nu\in \P_{r,2k}$
	under $\kappa$.
	Employing this approach,
	we deduce a close formula for the structure constants of
	$\qk(\mathrm{Gr}(2,N))$ in Section~\ref{subsec:rank-2-case},
	aligning with results for $\qk(\mathrm{Gr}(2,4))$ and $\qk(\mathrm{Gr}(2,5))$ in~\cite[Example 5.9]{Buch-Mihalcea}
	and~\cite[Appendix C]{Ueda-Yoshida} respectively.
	Complete multiplication tables for
	$\qk(\mathrm{Gr}(2,6))$ and $\qk(\mathrm{Gr}(3,6))$
	are provided in Appendix~\ref{subsec:examples}. 
	\begin{rem}
		The quantum $K$-product for $\qk(3, N)$ was calculated in \cite{CZJM} using the Seidel representation. A natural question is whether the approach described above can be used to recover the result for $\qk(3, N)$ and to find a general formula for the quantum product in $\qk(r, N)$ for all $r$ and $N$.
	\end{rem}

	\subsection{Quantum $K$-ring of $\mathrm{Gr}(2,N)$}
	\label
	{subsec:rank-2-case}
	In this subsection, we focus on the rank $r=2$ case
	and explicitly calculate the 1-pointed Quot scheme invariants
	of $\rO_\lambda$ using the bialternant formula from Theorem~\ref{cor:G-lambda}.
	Employing the $K$-theoretic Littlewood--Richardson rule and
	the derived formulas for the 1-pointed invariants,
	we fully determine the structure constants for
	the quantum $K$-product of $X=\mathrm{Gr}(2,N)$.
	As a corollary, we re-derive an explicit formula for the product in the quantum cohomology ring of $\mathrm{Gr}(2,N)$.
	Additionally, we detail the quantized
	pairing matrix and its inverse.
	Throughout the discussion, we assume $r=2$ and $k=N-2$.

	As described in~\cite{Buch}, the $K$-theoretic Littlewood--Richardson rule for
	two-variable Grothendieck polynomials states that for any partitions $\lambda$ and $\mu$:
	\[\rO_{\lambda} \cdot \rO_{\mu} =\sum_{\nu}^{} c_{\lambda,\mu}^{\nu}\rO_{\nu} \]
	where $\nu$ runs over partitions with at most two parts, and $c_{\lambda,\mu}^{\nu}$ is $(-1)^{|\nu|-|\lambda|-|\mu|}$ times the number of set-valued tableaux $T$
	of shape $\lambda*\mu$ with content $\nu$ such that the word $w(T)$ of $T$ is a reverse lattice word. 
		
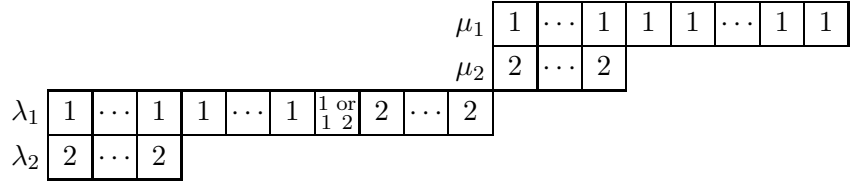
\[	\begin{ytableau}
			\none&  \none&\none&  \none&  \none &\none&  \none &\none&  \none &\none& \none[\mu_{1}] &1&\cdots& 1&1&1&\cdots&1&1\\
			\none&  \none&\none&  \none &  \none&\none&  \none &\none&  \none &\none&  \none[\mu_{2}] & 2&\cdots &2\\
			\none[\lambda_{1}]&1&\cdots& 1&1&\cdots& 1 &   
			\substack{1 \text{\ or} \\ 1\ 2}
			&2&\cdots&2&\none& \none\\
			\none[\lambda_{2}]&2&\cdots &2
		\end{ytableau}\]
	
	Consider the shape $\lambda*\mu$ drawn above.
	The semi-standard condition on the set-valued tableaux
	means that the rows are weakly increasing from left to right and the columns are strictly
	increasing from top to bottom.
	This implies that the columns with 2 boxes must contain $\{1\}$ followed by $\{2\}$
	and the set $\{1,2\}$ can occur at most once in the third row.
	The condition on the word $w(T)$ being a reverse lattice word implies that
	each box in the first row of $\lambda*\mu$ contains only the set $\{1\}$
	and the number of the set $\{2\}$ in the third row is at most $$m_{\lambda,\mu}=\min\{\mu_1-\mu_2,\lambda_{1}-\lambda_{2} \}.$$  
	Therefore, we can explicitly write the $K$-theoretic Littlewood--Richardson rule as
	\begin{equation}
		\label
		{eq:rank_2_product_grothendieck_functor}
		\rO_{\lambda} \cdot \rO_{\mu} =\rO_{\lambda+\mu}+ \sum_{i=1}^{m_{\lambda,\mu}} \Big[\rO_{\lambda+\mu+(-i,i)}-\rO_{\lambda+\mu+(-i+1,i)}\Big], 
	\end{equation}
	where $i$ denotes the number of $\{2\}$'s in the third row of $\lambda*\mu$.  
	
	We now present the formula for the quantum product between Schubert structure sheaves.
	For tuples of weakly decreasing integers (not necessarily partitions)
	$\lambda=(\lambda_{1},\lambda_{2})$ and
	$\mu= (\mu_{1},\mu_2)$, we define an element in $K(\grtwo)$ 
	using the Littlewood--Richardson rule~\eqref{eq:rank_2_product_grothendieck_functor}:
	\[  
	\ca{G}_{\lambda,\mu}:= \cO_{\lambda+\mu}+ \sum_{i=1}^{m_{\lambda,\mu}} \Big[\cO_{\lambda+\mu+(-i,i)}-\cO_{\lambda+\mu+(-i+1,i)}\Big],
	\]
	where $m_{\lambda,\mu}:=\min\{\mu_1-\mu_2,\lambda_{1}-\lambda_{2}\}$,
	and $\cO_\nu =0$ for $\nu\notin \P_{2,k}$.
	Note that if $\lambda,\mu\in\P_{r,k}$,
	then $ \cO_{\lambda} \cdot \cO_{\mu}  = \ca{G}_{\lambda,\mu}$ in $K(\grtwo)$.

	\begin{theorem}
		\label
		{thm:rank2_quantum_K_ring}
		Let $\lambda, \mu\in \P_{2,N-2}$ with $\mu_{1}-\mu_{2}\le \lambda_{1}-\lambda_2$.
		Then the product in the quantum $K$-ring $\qk(\mathrm{Gr}(2,N))$ is:
		\begin{equation}
			\label{eq:quantum-LR-rule}
			\cO_{\lambda}\bullet \cO_{\mu}=\ca{G}_{\lambda,\mu}+ \ca{G}_{\tilde{\lambda},\mu}q+\ca{G}_{\lambda-(N,N),\mu}q^2.
		\end{equation}
		where $\tilde{\lambda} = (\lambda_2-1,\lambda_1-N+1)$.
	\end{theorem}

	\begin{exm}
		For $N=6$, Theorem~\ref{thm:rank2_quantum_K_ring} implies  $ \O_{2}  \bullet  \O_{4,1}    = \O_{4,3}   + (\O_1 - \O_{2})q $ and $\O_{4,2}    \bullet  \O_{4,2}   = (\O_{3,3} + \O_{4,2} - \O_{4,3})q+ (1- \O_1)q^2 . $
	\end{exm}
	
	The quantum cohomology ring of the Grassmannian is fully determined by its quantum $K$-ring.
	As a corollary, we explicitly compute the product in the quantum cohomology of $X=\mathrm{Gr}(2,N)$.
	Let $X_\lambda$ be the Schubert variety for $\lambda$.
	The quantum cohomology ring $\mathrm{QH}(X)$, as a
	$\Z[q]$-algebra $H^*(X)\otimes_\Z\Z[q]$,
	is defined with the quantum product:
	\[
	[X_\lambda]\star [X_\mu]
	=\sum_{\nu,d\geq 0}
	q^d
	\left\langle
	[X_{\lambda}],
	[X_{\mu}],
	[X_{\nu^\vee}]
	\right\rangle_{0,3,d}^{\rm GW}
	[X_\nu],
	\]
	where $\langle-\rangle_{0,3,d}^{\rm GW}$ denotes the 3-pointed (cohomological) Gromov--Witten invariant.
	Assigning a cohomological degree $N$ to the Novikov variable $q$,
	$\mathrm{QH}(X)$ becomes a graded ring.
	
	By Remark~\cite[Remark 5.2]{Buch-Mihalcea}, $\qk(X)$ has a topological filtration by ideals
	defined by $F_j\qk(X)=\bigoplus_{|\lambda|+Ni\geq j}\Z\cdot q^i\O_\lambda$,
	and the associated graded ring is $\mathrm{QH}(X)$
	after identifying $\O_\lambda$ with $[X_\lambda]$.
	Thus, replacing $\O_\lambda$ with $[X_\lambda]$ in the quantum $K$-Littlewood--Richardson
	rule~\eqref{eq:quantum-LR-rule} and omitting the terms on the right-hand side whose cohomological
	degrees are greater than $|\lambda|+|\mu|$,
	we obtain a formula for the quantum cohomology product that matches the known result of Bertram--Fulton--Ciocan-Fontanine~\cite{Bertram-Ciocan-Fontanine-Fulton} in the rank 2 case.
	\begin{cor}
		\label{cor:cohomological-quantum-LR-rule}
		For $\lambda, \mu\in \P_{2,k}$ with $\mu_{1}-\mu_{2}\le \lambda_{1}-\lambda_2$,
		the product in the quantum cohomology ring $\mathrm{QH}(X)$ is given by
		\begin{equation*}
			[X_{\lambda}]\star [X_{\mu}]
			=
			\overline{\ca{G}}_{\lambda,\mu}
			+ \overline{\ca{G}}_{\tilde{\lambda},\mu}q
			+\overline{\ca{G}}_{\lambda-(N,N),\mu}q^2.
		\end{equation*}
		where $\tilde{\lambda} = (\lambda_2-1,\lambda_1-k-1)$ and
		\[
		\overline{\ca{G}}_{\lambda,\mu}
		:=
		[X_{\lambda+\mu}]
		+\sum_{i=1}^{m_{\lambda,\mu}} 
		[X_{\lambda+\mu+(-i,i)}].
		\]
		
	\end{cor}
	
	\subsection{Proof of Theorem~\ref{thm:rank2_quantum_K_ring}}
	To prove the quantum $K$-Littlewood--Richardson rule, we first compute some 1-pointed correlation
	functions of $\rO_\lambda$'s using Theorem~\ref{cor:G-lambda}.
	Let $z_1,z_2,\dots,z_N$ be the roots of $(z-1)^N+z^{N-2}t=0$.
	Then the elementary symmetric polynomials in $z_i$'s are
	\begin{equation}
		\label
		{eq:elementary-t-rank-2}
		e_{j}(\vec{z})=\begin{cases}
			\binom{N}{j}
			& \quad \text{if}\ j\ne 2
			\\[5pt]
			\binom{N}{2}
			+t& \quad \text{if}\ j= 2
		\end{cases}.
	\end{equation}
	Note that only $e_2$ depends on $t$.
	
	\begin{prop}\label
		{prop:rank2_1-pointed_invariants}
		For any partition $\lambda=(\lambda_{1},\lambda_{2}) \in \P_{2,2k+2}$, 
		\begin{equation*}
			\langle\!\langle
			\rO_{\lambda}
			\rangle\!\rangle^{\quot}_{0}=\begin{cases}
				\frac{1}{1-q} & \lambda_{1}\le k\\[5pt]
				\frac{q^2}{1-q}  & k <\lambda_{1}\le  2k+2, N\le \lambda_{2}\\[5pt]
				\frac{q}{1-q}  &k <\lambda_{1}\le 2k+1, \lambda_{2}<N\\[5pt]
				(\lambda_2-N)q+\frac{q}{1-q}  &\lambda_{1}=2k+2,\lambda_{2}<N
			\end{cases}.
		\end{equation*}
	\end{prop}
	\begin{proof}
		The first case follows from Corollary~\ref{rem:G-correlation-function} (i) and the second
		case can be calculated from the first case using the periodicity in Corollary~\ref{cor:periodic_Grotendieck_1-point_invariants}.
		
		Now we focus on the last two cases.
		Corollary~\ref{rem:G-correlation-function} (ii)
		implies $\langle\rO_{\lambda}\rangle^{\quot}_{0,d}=1$
		for any $d\geq 2$ and $\lambda \in \P_{2,2k+2}$.
		Also note that the degree-0 term $\langle\rO_{\lambda}\rangle^{\quot}_{0,0}$
		vanishes when $\lambda_{1}>k$.
		Hence to finish proving the last two cases, we only to need calculate the degree-1 invariant
		$\langle\rO_{\lambda}\rangle^{\quot}_{0,1}$ for $k<\lambda_{1}\le 2k+2$ and $\lambda_{2}<N$.
		Note that $\lambda_{1}+1\ge N$.
		Thus Theorem~\ref{cor:G-lambda} and the identity $(1-z_j^{-1})^N=-tz_j^{-2}$ implies
		\begin{align*}
			\langle
			\rO_{\lambda}
			\rangle^{\quot}_{0,1}&= [t^1]s_{(z(1-\frac{1}{z})^{\lambda_{1}+1},z(1-\frac{1}{z})^{\lambda_{2}}) }(\vec{z})\\
			&=-[t^0]s_{(z^{-1}(1-\frac{1}{z})^{\lambda_{1}+1-N},z(1-\frac{1}{z})^{\lambda_{2}}) }(\vec{z})\\
			&=-s_{(z^{-1}(1-\frac{1}{z})^{\lambda_{1}+1-N},z(1-\frac{1}{z})^{\lambda_{2}}) }(\vec{z})\big|_{z_1=\cdots = z_N=1}
			.
		\end{align*}
		Using the binomial expansions of $(1-\frac{1}{z})^{\lambda_{1}+1-N}$ and $(1-\frac{1}{z})^{\lambda_{2}}$,
		the above expression equals
		\begin{align*}
			\sum_{a=0}^{\lambda_{1}+1-N}\sum_{b=0}^{\lambda_{2}}\binom{\lambda_{1}+1-N}{a}\binom{\lambda_{2}}{b}(-1)^{a+b+1}s_{(-1-a,1-b)}(\vec{z})\big|_{z_1=\cdots = z_N=1}.
		\end{align*}
		Note that $s_{(-1-a,1-b)}(\vec{z})$ equals zero when $1 \leq a \leq N-2$ or $2 \leq b \leq N-1$
		because the first row, respectively the second row, in the bialternant formula~\eqref{eq:bialternant-def}
		matches with one of the last $N-2$ rows.
		Furthermore, when $(a,b)=(0,1)$, we have $s_{(-1-a,1-b)}(\vec{z})=s_{(-1,0)}=0$ by the straightening rule~\eqref{eq:straightening-rules}.
		Hence there are only three potentially non-zero terms in the above expansion which
		correspond to $(a,b)=(0,0),(N-1,0),(N-1,1)$; the last two cases only occur when
		$\lambda_1=2k+2$.
		
		For $(a,b)=(0,0)$, we have $s_{(-1,1)} = -s_{(0,0)}$ (by straightening).
		This finishes the proof of the third case of the proposition.
		Now suppose $\lambda_1=2k+2$.
		Using the straightening rule for Schur functions, we obtain
		\begin{align*}
			s_{(-N,1-b)}(\vec{z}) &= (-1)^{N-1} s_{(-b,-1,\dots,-1,-1)}(\vec{z})\\
			&=(-1)^{N-1}s_{(1-b)}(\vec{z})\prod_{i=1}^{N}z_i^{-1}.
		\end{align*}
		Therefore, when $\lambda_{1}=2k+2$, we have
		\begin{align*}
			\langle
			\rO_{\lambda}
			\rangle^{\quot}_{0,1}
			&= -s_{(-1,1)}(1,\dots,1)+(-1)^{N}s_{(-N,-1)}(1,\dots,1)+(-1)^{N+1}s_{(-N,0)}(1,\dots,1)
			\\
			&=
			1-s_{(1)}(1,\dots, 1)+\lambda_{2}s_{(0)}(1,\dots ,1)\\
			&= 1-N+\lambda_{2}. 
		\end{align*}
		This concludes the proof of the fourth case.
	\end{proof}
	
	For any partition $\nu$, let $$\Delta_\nu := \rO_{\nu}-\rO_{\nu+(1,0)}.$$
	
	\begin{cor}
		Let $\lambda,\mu\in \P_{2,k}$ such that $\mu_{1}-\mu_{2}\le \lambda_{1}-\lambda_{2}$.
		Then the quantized pairing matrix equals
		\begin{equation*}
			F_{\lambda,\mu} =	\langle\!\langle
			\rO_{\lambda+\mu}
			\rangle\!\rangle^{\quot}_{0}+ \begin{cases}
				1& \lambda_1+\mu_2\le k<\lambda_{1}+\mu_{1}\\
				0& \text{Otherwise}
			\end{cases},
		\end{equation*}
		where the first term is given in Proposition~\ref{prop:rank2_1-pointed_invariants}.
	\end{cor}
	
	\begin{proof}
		The formulas of the 1-pointed correlation functions in Proposition~\ref{prop:rank2_1-pointed_invariants}
		imply that for any partition $\nu\in \P_{2,2k}$,
		\begin{equation*}
			\langle\!\langle
			\Delta_\nu
			\rangle\!\rangle^{\quot}_{0}= 
			\delta_{k,\nu_1}.
		\end{equation*}
		Using~\eqref{eq:rank_2_product_grothendieck_functor} in conjunction with the above identity, we obtain
		\[ 	F_{\lambda,\mu} =\langle\!\langle \rO_{\lambda} \cdot \rO_{\mu}\rangle\!\rangle^{\quot}_{0}= 	\langle\!\langle
		\rO_{\lambda+\mu}
		\rangle\!\rangle^{\quot}_{0}+ \sum_{i=1}^{\mu_{1}-\mu_{2}} \delta_{k,\lambda_1+\mu_1-i}.  \]
		Note that the summation equals $1$ if $\lambda_1+\mu_2\le k<\lambda_{1}+\mu_{1}$ and 0 otherwise. 
	\end{proof}

	We follow the strategy outlined in Section~\ref{subsec:quantum-reduction-map}
	and prove Theorem~\ref{thm:rank2_quantum_K_ring} using the quantum reduction
	map $\kappa$.
	To explicitly compute $\kappa(\rO_\lambda)$ for any $\lambda\in\P_{2,2k}$,
	we need the 1-pointed invariants for all partitions in $\P_{2,3k+1}$.
	\begin{prop}\label
		{prop:rank2_1-point_invariants_2k+2to3k}
		For any partition $\lambda\in \P_{2,3k+1}$, and $\lambda\notin \P_{2,2k+2}$,
		\begin{equation}
			\label
			{eq:one-pointed-invariant-3k-1}
			\langle\!\langle
			\rO_{\lambda}
			\rangle\!\rangle^{\quot}_{0}
			=\begin{cases}
				\frac{q^3}{1-q}  & 2k+2<\lambda_{1}\le 3k+1,\lambda_{2}<2N\\[5pt]
				\frac{q^4}{1-q}  & 2k+2<\lambda_{1}\le 3k+1,\lambda_{2}\ge 2N\\
			\end{cases}.
		\end{equation}
	\end{prop}
	\begin{proof}
		When $\lambda_{2}\ge N$, the proposition follows using periodic property in Corollary~\ref{cor:periodic_Grotendieck_1-point_invariants} and Proposition~\ref{prop:rank2_1-pointed_invariants}.
		Now we focus on the first case of~\eqref{eq:one-pointed-invariant-3k-1} assuming that $\lambda_2\leq N-1$.
		Since $(1-z_j^{-1})^{2N}=t^2z_j^{-4}$, we have
		\begin{align*}
			\langle
			\rO_{\lambda}
			\rangle^{\quot}_{0,d}&= [t^d]s_{(z^d(1-\frac{1}{z})^{\lambda_{1}+1},z^d(1-\frac{1}{z})^{\lambda_{2}}) }(\vec{z})\\
			&=[t^{d-2}]s_{(z^{d-4}(1-\frac{1}{z})^{a},z^d(1-\frac{1}{z})^{b}) }(\vec{z})\\
			&=-[t^{d-2}]s_{(z^{d-1}(1-\frac{1}{z})^{b},z^{d-3}(1-\frac{1}{z})^{a}) }(\vec{z}),
		\end{align*}
		where $(a,b)=(\lambda_{1}+1-2N,\lambda_2)$.
		This implies that the 1-pointed invariant vanishes if $d\leq 1$.
		Suppose $d\geq2 $.
		Note that $0\le a\le N-4$ and $0\le b\le N-1$.
		Using the binomial expansions of $(1-\frac{1}{z})^b$ and $(1-\frac{1}{z})^a$,
		we have
		\begin{align*}
			\langle
			\rO_{\lambda}
			\rangle^{\quot}_{0,d}
			&=[t^{d-2}] 
			\sum_{i=0}^{b}\sum_{j=0}^{a}(-1)^{i+j+1}\binom{b}{i}\binom{a}{j}
			s_{(d-1-i,d-3-j) }(\vec{z})\\
			&=[t^{d-2}] 
			\sum_{i=0}^{d}\sum_{j=0}^{d-3}(-1)^{i+j+1}\binom{b}{i}\binom{a}{j}
			s_{(d-1-i,d-3-j) }(\vec{z})
			.
		\end{align*}
		Here the second equality holds because $s_{(d-1-i,d-3-j) }(\vec{z})=0$
		when $i\geq d+1$ or $j\geq d-2$;
		under either condition, the first row, respectively the second row, in the bialternant formula~\eqref{eq:bialternant-def}
		of $s_{(d-1-i,d-3-j) }(\vec{z})$ matches with one of the last $N-2$ rows.
		
		Since the summation for $j$ is empty when $d=2$, the 1-pointed
		invariant of degree 2 vanishes.
		When $d\ge 3$,
		the term $i=j=0$ equals $-s_{d-1,d-3}=e_{2}^{d-3}(e_2-e_1^2)$ which is a monic
		polynomial in $t$ of degree $d-2$ by~\eqref{eq:elementary-t-rank-2}.
		If either $i$ or $j$ is nonzero,
		the partitions obtained by straightening $(d-1-i,d-3-j)$ is strictly contained
		in the partition $(d-1,d-3)$.
		In particular,
		the degree of the corresponding Schur polynomial is strictly less that
		$2(d-2)$ and, therefore,
		the highest power of $e_2$ (which has degree 2 in $z_i$'s)
		appearing in the Jacobi--Trudi expansion of $s_{(d-1-i,d-3-j) }$ is strictly less than $d-2$. Hence, $	\langle
		\rO_{\lambda}
		\rangle^{\quot}_{0,d}=1$ for all $d\ge 3$.
	\end{proof}
	
	Recall that $\mathrm{S}=(\C^{2})^{\vee}$ denotes the dual of the standard representation of $\mathrm{GL}_2(\C)$.
	\begin{cor}\label{cor:rank_2_1-pointed_dual}
		For any partition $\lambda\in \P_{2,3k}$, 
		\begin{equation*}
			\langle\!\langle
			\rO_{\lambda}
			\cdot
			\det(\rS)
			\rangle\!\rangle^{\quot}_{0}=\begin{cases}
				1 & \lambda=(k,k) \\
				q  & \lambda_1=2k+1,\ \lambda_2\le k+1 \\
				-q  & \lambda_1=2k+2,\ \lambda_2\le k \\
				q^2  & \lambda = (2k+2,2k+2) \\
				0 &\text{otherwise}
			\end{cases}.
		\end{equation*}
	\end{cor}
	\begin{proof}
		Using the Littlewood--Richardson rule~\eqref{eq:rank_2_product_grothendieck_functor},
		we obtain
		\begin{align*}
			\rO_{\lambda}
			\cdot
			\det(\rS) = \rO_{\lambda}\cdot(1-\rO_{(1)}) 
			= \rO_{\lambda}-\rO_{\lambda+(1,0)} - \bar{\delta}_{\lambda_{1},\lambda_{2}}\big(\rO_{\lambda+(0,1)}-\rO_{\lambda+(1,1)}\big)\\
		\end{align*}
		Here $\bar{\delta}_{\lambda_{1},\lambda_{2}} := 1-\delta_{\lambda_{1},\lambda_{2}}$.
		The corollary follows by carefully plugging in the formulas for the 1-pointed invariants calculated in Propositions~\ref{prop:rank2_1-pointed_invariants} and~\ref{prop:rank2_1-point_invariants_2k+2to3k}.
		We leave the details to the reader.
	\end{proof}
	 
	Recall from Definition~\ref{def:kappa-map} that for any partition $\nu\in \P_{2,2k}$, \[\kappa(\rO_\nu) = \sum_{\alpha\in \P_{2,k}}^{}\langle\!\langle
	\rO_{\nu}\cdot 	\rO_{\alpha}^*
	\rangle\!\rangle^{\quot}_{0} \cO_\alpha, \]
	where $\langle\!\langle
	\rO_{\nu}\cdot 	\rO_{\alpha}^*
	\rangle\!\rangle^{\quot}_{0}$ is explicitly described in the following corollary.

	\begin{prop}
		\label
		{cor:rank2_kappa_map}
		For any partition $\nu\in \P_{2,2k}$ and $\alpha\in \P_{2,k}$,
		\begin{equation}
			\label
			{eq:rank2_kappa_map}
			\langle\!\langle
			\rO_{\nu}\cdot 	\rO_{\alpha}^*
			\rangle\!\rangle^{\quot}_{0}=\begin{cases}
				1 & \nu=\alpha\\
				q^2& \nu=\alpha+(N,N)\\
				q& \nu +\alpha^*=(2k+1,k+1)\\
				q& \nu_1 = \alpha_1+k+1,\ \nu_2\le \alpha_2\\
				-q&\nu_1 = \alpha_1+k+2,\ \nu_2\le \alpha_2\\
				0& \text{otherwise}
			\end{cases}.
		\end{equation}
	\end{prop}
	\begin{proof}
		Recall that $\langle\!\langle
		\rO_{\nu}\cdot 	\rO_{\alpha}^*
		\rangle\!\rangle^{\quot}_{0} = \langle\!\langle
		\rO_{\nu}\cdot 	\rO_{\alpha^*}\cdot \det(\rS)
		\rangle\!\rangle^{\quot}_{0}$. We will first give a suitable description of
		$\rO_{\nu}\cdot\rO_{\alpha^*}$.	For any partitions $\lambda=(\lambda_1,\lambda_2)$ with $\lambda_1\ne \lambda_2$, we define $$\widetilde{\Delta}_{\lambda} := \rO_{\lambda}-\rO_{\lambda+(0,1)}. $$
		We rewrite the Littlewood--Richardson rule as
		\begin{equation}
			\label
			{eq:rewrite-LR-rule}
			\rO_{\nu}\cdot 	\rO_{\alpha^*}
			=  \rO_{\nu+\alpha^*}
			+ \sum_{i=1}^{m_{\nu,\alpha^*}} 
			\Delta_{\nu+\alpha^*+(-i,i)}
			=\rO_{\nu\oplus\alpha^*}
			+ \sum_{i=1}^{m_{\nu,\alpha^*}} \widetilde{\Delta}_{\nu\oplus\alpha^* + (i,-i)}.
		\end{equation}
		where $\nu\oplus\alpha^*$ is the partition with parts $\nu_{1}+k-\alpha_1$ and $\nu_{2}+k-\alpha_2$.
		Note that for any $\nu$ and $\alpha$, the partitions $\nu\oplus\alpha^*+(i,-i)$
		in the summation have parts differed by at least $2$. 
		
		Corollary~\ref{cor:rank_2_1-pointed_dual} implies that for any partition
		$\lambda\in \P_{2,3k}$ with $\lambda_1-\lambda_2\ge 2$, we have 
		\begin{equation}\label{eq:rank_2_Delta_tilde}
			\langle\!\langle
			\widetilde{\Delta}_{\lambda}\cdot  \det(\rS)
			\rangle\!\rangle^{\quot}_{0}= \begin{cases}
				q  & \lambda = (2k+1, k+1) \\
				-q  & \lambda= (2k+2, k) \\
				0 &\text{otherwise}
			\end{cases}.
		\end{equation}
		In particular, the summation involving $\widetilde{\Delta}$'s in~\eqref{eq:rewrite-LR-rule} does not contribute to the $q^0$ and $q^2$ terms in $\langle\!\langle
		\rO_{\nu}\cdot 	\rO_{\alpha}^*
		\rangle\!\rangle^{\quot}_{0}$.
		Thus, Corollary~\ref{cor:rank_2_1-pointed_dual} implies that the coefficient of
		$q^0$ (resp. $q^2$) in $\langle\!\langle
		\rO_{\nu}\cdot 	\rO_{\alpha}^*
		\rangle\!\rangle^{\quot}_{0}$ is nonzero (and equals 1)
		if and only if $\nu\oplus \alpha^* = (k,k)$
		(resp. $\nu\oplus \alpha^* = (2k+2,2k+2)$).
		Note that these two conditions are equivalent to $\nu=\alpha$ and $\nu
		=\alpha +(k+2,k+2)
		=\alpha +(N,N)
		$,
		respectively.
		Hence the first two cases in~\eqref{eq:rank2_kappa_map} hold.
		
		We prove the remaining cases by evaluating the coefficient of $q^1$ in $\langle\!\langle
		\rO_{\nu}\cdot 	\rO_{\alpha}^*
		\rangle\!\rangle^{\quot}_{0}$.
		We may assume that $\nu_{1}>k$ because the case $\nu_1\leq k$
		is covered by the orthogonality in~\eqref{eq:orthogonality}.
		Note that the summation
		$\sum_{i=1}^{m_{\nu,\alpha^*}} 
		\widetilde{\Delta}_{\nu\oplus\alpha^* + (i,-i)}$
		in~\eqref{eq:rewrite-LR-rule} is non-empty if and only if $\nu_1 \neq \nu_2$
		and $\alpha_1 \neq \alpha_2$.
		In this case, it follows from~\eqref{eq:rank_2_Delta_tilde} that
		\begin{equation}
			\label{eq:delta_sum}
			\sum_{i=1}^{m_{\nu,\alpha^*}} 
			\langle
			\widetilde{\Delta}_{\nu\oplus\alpha^* + (i,-i)}
			\cdot
			\det(\rS)
			\rangle^{\quot}_{0,1}= \begin{cases}
				q& \nu+\alpha^* = (2k+1,k+1)\\
				-q& \nu\oplus\alpha^* = (2k+1,k+1)\\
				0& \text{otherwise}
			\end{cases}.
		\end{equation}
		Furthermore, by Corollary~\ref{cor:rank_2_1-pointed_dual}, we have
		\begin{equation}
			\label{eq:main_term}
			\langle \rO_{\nu\oplus\alpha^*}
			\cdot
			\det(\rS)
			\rangle^{\quot}_{0,1}=\begin{cases}
				q& \nu\oplus\alpha^* = (2k+1,j),\ j\le k+1\\
				-q& \nu\oplus\alpha^* = (2k+2,j),\ j\le k\\
				0& \text{otherwise}
			\end{cases}.
		\end{equation}
		There are three cases to consider: (i) $\nu_1 =\nu_2$; (ii) $\alpha_1=\alpha_2$; (iii)
		$\nu_1 \neq \nu_2,\alpha_1 \neq \alpha_2$.
		In the first two cases,~\eqref{eq:delta_sum} does not contribute. The proposition follows by carefully adding~\eqref{eq:delta_sum} and~\eqref{eq:main_term}.
		We leave the details to the reader.
	\end{proof}

	\begin{cor}
		\label
		{cor:rank-2-kappa-map-formula}
		For any partition
		$\nu\in \P_{2,2k}$,
		we have
		\[
		\kappa(\rO_\nu)
		=
		\O_{\nu}
		+q^2\O_{\nu-(N,N)}
		+
		q\bigg(\cO_{(\nu_2-1,\nu_1-k-1)} +\sum_{i=\nu_2}^{\nu_1-k-1} \cO_{(\nu_1-k-1,i)}-\sum_{i=\nu_2}^{\nu_1-k-2} \cO_{(\nu_1-k-2,i)}
		\bigg). 
		\]
		Here we set $\O_\lambda=0$ when $\lambda\notin \P_{2,k}$.
	\end{cor}
	\begin{proof}
		It follows directly from Proposition~\ref{cor:rank2_kappa_map} and the definition of $\kappa$.
	\end{proof}

	\begin{proof}[Proof of Theorem~\ref{thm:rank2_quantum_K_ring}]
		Let $\lambda,\mu\in \P_{2,k}$ and $\mu_1-\mu_2\le \lambda_1-\lambda_{2}$,
		which implies $\lambda\oplus \mu = (\lambda_1+\mu_{2},\lambda_2+\mu_1)$.
		By applying the quantum reduction map $\kappa$ to the Littlewood--Richardson rule~\eqref{eq:rewrite-LR-rule},
		we obtain
		\[
		\cO_\lambda\bullet\cO_{\mu} = \kappa(\rO_{\lambda\oplus\mu})+\sum_{i=1}^{\mu_1-\mu_2}\kappa \big(\widetilde{\Delta}_{\lambda\oplus \mu+(i,-i)}\big).
		\]
		Finding the coefficient of $q^0$ and $q^2$ in $\cO_\lambda\bullet\cO_{\mu}$ is
		straightforward using Corollary~\ref{cor:rank-2-kappa-map-formula}.
		
		We will consider the $q^1$ coefficient in the rest of the proof. According to Corollary~\ref{cor:rank-2-kappa-map-formula},
		the coefficient of $q^1$ in $\kappa(\rO_\nu)$ is
		\[
		[q^1]\kappa(\rO_\nu)
		=
		\cO_{(\nu_2-1,\nu_1-k-1)} +\sum_{i=\nu_2}^{\nu_1-k-1} \cO_{(\nu_1-k-1,i)}-\sum_{i=\nu_2}^{\nu_1-k-2} \cO_{(\nu_1-k-2,i)}. 
		\]
		Note that for any partition involved in the last two summations,
		the difference between its parts is no less than $k+1$.
		Since $(\lambda_{1}+\mu_{2})-(\lambda_{2}+\mu_1)\le k$,	we have
		\[
		[q^1]\kappa(\rO_{\lambda\oplus\mu}) =q  \cO_{\tilde{\lambda}+\mu},
		\]
		where $\tilde{\lambda} := (\lambda_2-1,\lambda_1-k-1)$. Corollary~\ref{cor:rank-2-kappa-map-formula} also implies that
		for any $\nu\in \P_{2,2k}$,
		\[
		[q^1]
		\kappa(\widetilde{\Delta}_{\nu} ) = \begin{cases}
			\Delta_{(\nu_2-1,\nu_1-k-1)}& \nu_1-\nu_2\le k\\
			0& \nu_1-\nu_2= k+1\\
			-\Delta_{(\nu_1-k-2, \nu_2) }& \nu_1-\nu_2\ge k+2\\
		\end{cases}. \]
		Using the above observation, a careful calculation gives us
		\begin{align*}
			[q^1]\sum_{i=1}^{\mu_1-\mu_2}\kappa \big(\widetilde{\Delta}_{\lambda\oplus \mu+(i,-i)}\big)
			&=\sum_{i=1}^{m_{\tilde{\lambda},\mu}} \Delta_{\tilde{\lambda}+\mu+(-i,i)}.
		\end{align*} 
		This finishes the proof.
	\end{proof}

	\begin{rem}
		In Section 3.2 of Part I~\cite{SinhaZhang}, we showed that
		$F^{\lambda, \mu}=	\langle\!\langle
		\rO_{\lambda}^{*}
		,
		\rO_{\mu}^{*}
		\rangle\!\rangle^{\quot}_{0}$.
		Using the identity
		\begin{align*}
			\rO_{\lambda}^{*}
			= \rO_{\lambda^*}-\rO_{\lambda^*+(1,0)} - \bar{\delta}_{\lambda_{1},\lambda_{2}}\big(\rO_{\lambda^*+(0,1)}-\rO_{\lambda^*+(1,1)}\big)
		\end{align*}
		with $\bar{\delta}_{\lambda_{1},\lambda_{2}} = 1-\delta_{\lambda_{1},\lambda_{2}}$,
		and Proposition~\ref{cor:rank2_kappa_map},
		one can show that the inverse of the quantized pairing matrix is given by
		\begin{equation*}
			F^{\lambda,\mu} = \begin{cases}
				1& \lambda=\lambda^*\\
				-1 & \lambda = \lambda^* +(1,0)\\
				-1 & \mu = \lambda^* + (1,0)\\
				1 & \lambda=\lambda^* + (1,1), \lambda_{1}\ne \lambda_{2}\\
				q& \lambda+\mu= (k, 0), 1\le \lambda_{1}\le k-1\\
				-q& \lambda+\mu= (k-1,0)\\
				0& \text{otherwise}
			\end{cases}.
		\end{equation*}
	\end{rem}

	
	\appendix

	\section{Quantum $K$-ring of $\mathrm{Gr}(3,6)$}
	\label
	{subsec:examples}
	
	Consider the Schubert basis of $K(\mathrm{Gr}(3,6))$:
	\begin{align*}
		\{&\O,\O_{1},\O_{1,1},\O_{1,1,1},\O_{2},\O_{2,1},\O_{2,1,1},\O_{2,2},\O_{2,2,1},\O_{2,2,2},\O_{3}, \\
		& \O_{3,1},\O_{3,1,1},\O_{3,2},\O_{3,2,1},\O_{3,2,2},\O_{3,3},\O_{3,3,1},\O_{3,3,2},\O_{3,3,3}\}
	\end{align*}
	We note that using Seidel representation (see \cite{CZJM}) we only require to explicitly state the product table for multiplication for a subset of the basis elements. Below, we give a full multiplication table for $\qk(\mathrm{Gr}(3,6))$ calculated using the quantum reduction map in Definition~\ref{def:kappa-map}:
	{\allowdisplaybreaks
		\footnotesize
		\begin{align*}
			\O_{1} \bullet \O_{1} & = \O_{1,1}+\O_{2}-\O_{2,1}
			&
			\O_{1} \bullet \O_{1,1} & = \O_{1,1,1}+\O_{2,1}-\O_{2,1,1}
			\\
			\O_{1} \bullet \O_{1,1,1} & = \O_{2,1,1}
			&
			\O_{1} \bullet \O_{2} & = \O_{2,1}+\O_{3}-\O_{3,1}
			\\
			\O_{1} \bullet \O_{2,1} & = \O_{2,1,1}+\O_{2,2}-\O_{2,2,1}
			&
			\O_{1} \bullet \O_{2,1,1} & = \O_{2,2,1}+\O_{3,1,1}-\O_{3,2,1}
			\\
			+\O_{3,1}&-\O_{3,1,1}-\O_{3,2}+\O_{3,2,1}
			&
			&
			\\
			\O_{1} \bullet \O_{2,2} & = \O_{2,2,1}+\O_{3,2}-\O_{3,2,1}
			&
			\O_{1} \bullet \O_{2,2,1} & = \O_{2,2,2}+\O_{3,2,1}-\O_{3,2,2}
			\\
			\O_{1} \bullet \O_{2,2,2} & = \O_{3,2,2}
			&
			\O_{1} \bullet \O_{3} & = \O_{3,1}
			\\
			\O_{1} \bullet \O_{3,1} & = \O_{3,1,1}+\O_{3,2}-\O_{3,2,1}
			&
			\O_{1} \bullet \O_{3,1,1} & = q -q\O_{1}+\O_{3,2,1}
			\\
			\O_{1} \bullet \O_{3,2} & = \O_{3,2,1}+\O_{3,3}-\O_{3,3,1}
			&
			\O_{1} \bullet \O_{3,2,1} & = q\O_{1}-q\O_{1,1}-q\O_{2}
			\\
			&
			&
			+q\O_{2,1}&+\O_{3,2,2}+\O_{3,3,1}-\O_{3,3,2}
			\\
			\O_{1} \bullet \O_{3,2,2} & = q\O_{1,1}-q\O_{2,1}+\O_{3,3,2}
			&
			\O_{1} \bullet \O_{3,3} & = \O_{3,3,1}
			\\
			\O_{1} \bullet \O_{3,3,1} & = q\O_{2}-q\O_{2,1}+\O_{3,3,2}
			&
			\O_{1} \bullet \O_{3,3,2} & = q\O_{2,1}-q\O_{2,2}+\O_{3,3,3}
			\\
			\O_{1} \bullet \O_{3,3,3} & = q\O_{2,2}
			&
			\O_{1,1} \bullet \O_{1,1} & = \O_{2,1,1}+\O_{2,2}-\O_{2,2,1}
			\\
			\O_{1,1} \bullet \O_{1,1,1} & = \O_{2,2,1}
			&
			\O_{1,1} \bullet \O_{2} & = \O_{2,1,1}+\O_{3,1}-\O_{3,1,1}
			\\
			\O_{1,1} \bullet \O_{2,1} & = \O_{2,2,1}+\O_{3,1,1}+\O_{3,2}
			&
			\O_{1,1} \bullet \O_{2,1,1} & = \O_{2,2,2}+\O_{3,2,1}-\O_{3,2,2}
			\\
			&-2\O_{3,2,1}
			&
			&
			\\
			\O_{1,1} \bullet \O_{2,2} & = \O_{3,2,1}+\O_{3,3}-\O_{3,3,1}
			&
			\O_{1,1} \bullet \O_{2,2,1} & = \O_{3,2,2}+\O_{3,3,1}-\O_{3,3,2}
			\\
			\O_{1,1} \bullet \O_{2,2,2} & = \O_{3,3,2}
			&
			\O_{1,1} \bullet \O_{3} & = \O_{3,1,1}
			\\
			\O_{1,1} \bullet \O_{3,1} & = q -q\O_{1}+\O_{3,2,1}
			&
			\O_{1,1} \bullet \O_{3,1,1} & = q\O_{1}-q\O_{1,1}+\O_{3,2,2}
			\\
			\O_{1,1} \bullet \O_{3,2} & = q\O_{1}-q\O_{2}+\O_{3,3,1}
			&
			\O_{1,1} \bullet \O_{3,2,1} & = q\O_{1,1}+q\O_{2}-2q\O_{2,1}
			\\
			&
			&
			&+\O_{3,3,2}
			\\
			\O_{1,1} \bullet \O_{3,2,2} & = q\O_{2,1}-q\O_{2,2}+\O_{3,3,3}
			&
			\O_{1,1} \bullet \O_{3,3} & = q\O_{2}
			\\
			\O_{1,1} \bullet \O_{3,3,1} & = q\O_{2,1}+q\O_{3}-q\O_{3,1}
			&
			\O_{1,1} \bullet \O_{3,3,2} & = q\O_{2,2}+q\O_{3,1}-q\O_{3,2}
			\\
			\O_{1,1} \bullet \O_{3,3,3} & = q\O_{3,2}
			&
			\O_{1,1,1} \bullet \O_{1,1,1} & = \O_{2,2,2}
			\\
			\O_{1,1,1} \bullet \O_{2} & = \O_{3,1,1}
			&
			\O_{1,1,1} \bullet \O_{2,1} & = \O_{3,2,1}
			\\
			\O_{1,1,1} \bullet \O_{2,1,1} & = \O_{3,2,2}
			&
			\O_{1,1,1} \bullet \O_{2,2} & = \O_{3,3,1}
			\\
			\O_{1,1,1} \bullet \O_{2,2,1} & = \O_{3,3,2}
			&
			\O_{1,1,1} \bullet \O_{2,2,2} & = \O_{3,3,3}
			\\
			\O_{1,1,1} \bullet \O_{3} & = q 
			&
			\O_{1,1,1} \bullet \O_{3,1} & = q\O_{1}
			\\
			\O_{1,1,1} \bullet \O_{3,1,1} & = q\O_{1,1}
			&
			\O_{1,1,1} \bullet \O_{3,2} & = q\O_{2}
			\\
			\O_{1,1,1} \bullet \O_{3,2,1} & = q\O_{2,1}
			&
			\O_{1,1,1} \bullet \O_{3,2,2} & = q\O_{2,2}
			\\
			\O_{1,1,1} \bullet \O_{3,3} & = q\O_{3}
			&
			\O_{1,1,1} \bullet \O_{3,3,1} & = q\O_{3,1}
			\\
			\O_{1,1,1} \bullet \O_{3,3,2} & = q\O_{3,2}
			&
			\O_{1,1,1} \bullet \O_{3,3,3} & = q\O_{3,3}
			\\
			\O_{2} \bullet \O_{2} & = \O_{2,2}+\O_{3,1}-\O_{3,2}
			&
			\O_{2} \bullet \O_{2,1} & = \O_{2,2,1}+\O_{3,1,1}+\O_{3,2}
			\\
			&
			&
			&-2\O_{3,2,1}
			\\
			\O_{2} \bullet \O_{2,1,1} & = q -q\O_{1}+\O_{3,2,1}
			&
			\O_{2} \bullet \O_{2,2} & = \O_{2,2,2}+\O_{3,2,1}-\O_{3,2,2}
			\\
			\O_{2} \bullet \O_{2,2,1} & = q\O_{1}-q\O_{1,1}+\O_{3,2,2}
			&
			\O_{2} \bullet \O_{2,2,2} & = q\O_{1,1}
			\\
			\O_{2} \bullet \O_{3} & = \O_{3,2}
			&
			\O_{2} \bullet \O_{3,1} & = \O_{3,2,1}+\O_{3,3}-\O_{3,3,1}
			\\
			\O_{2} \bullet \O_{3,1,1} & = q\O_{1}-q\O_{2}+\O_{3,3,1}
			&
			\O_{2} \bullet \O_{3,2} & = \O_{3,2,2}+\O_{3,3,1}-\O_{3,3,2}
			\\
			\O_{2} \bullet \O_{3,2,1} & = q\O_{1,1}+q\O_{2}-2q\O_{2,1}
			&
			\O_{2} \bullet \O_{3,2,2} & = q\O_{1,1,1}+q\O_{2,1}-q\O_{2,1,1}
			\\
			&+\O_{3,3,2}
			&
			&
			\\
			\O_{2} \bullet \O_{3,3} & = \O_{3,3,2}
			&
			\O_{2} \bullet \O_{3,3,1} & = q\O_{2,1}-q\O_{2,2}+\O_{3,3,3}
			\\
			\O_{2} \bullet \O_{3,3,2} & = q\O_{2,1,1}+q\O_{2,2}-q\O_{2,2,1}
			&
			\O_{2} \bullet \O_{3,3,3} & = q\O_{2,2,1}
			\\
			\O_{2,1} \bullet \O_{2,1} & = q -2q\O_{1}+q\O_{1,1}
			&
			\O_{2,1} \bullet \O_{2,1,1} & = q\O_{1}-q\O_{1,1}-q\O_{2}
			\\
			+q\O_{2}&-q\O_{2,1}+\O_{2,2,2}+2\O_{3,2,1}
			&
			+q\O_{2,1}&+\O_{3,2,2}+\O_{3,3,1}-\O_{3,3,2}
			\\
			-2\O_{3,2,2}&+\O_{3,3}-2\O_{3,3,1}+\O_{3,3,2}
			&
			&
			\\
			\O_{2,1} \bullet \O_{2,2} & = q\O_{1}-q\O_{1,1}-q\O_{2}
			&
			\O_{2,1} \bullet \O_{2,2,1} & = q\O_{1,1}+q\O_{2}-2q\O_{2,1}
			\\
			+q\O_{2,1}&+\O_{3,2,2}+\O_{3,3,1}-\O_{3,3,2}
			&
			&+\O_{3,3,2}
			\\
			\O_{2,1} \bullet \O_{2,2,2} & = q\O_{2,1}
			&
			\O_{2,1} \bullet \O_{3} & = \O_{3,2,1}
			\\
			\O_{2,1} \bullet \O_{3,1} & = q\O_{1}-q\O_{1,1}-q\O_{2}
			&
			\O_{2,1} \bullet \O_{3,1,1} & = q\O_{1,1}+q\O_{2}-2q\O_{2,1}
			\\
			+q\O_{2,1}&+\O_{3,2,2}+\O_{3,3,1}-\O_{3,3,2}
			&
			&+\O_{3,3,2}
			\\
			\O_{2,1} \bullet \O_{3,2} & = q\O_{1,1}+q\O_{2}-2q\O_{2,1}
			&
			\O_{2,1} \bullet \O_{3,2,1} & = q\O_{1,1,1}+2q\O_{2,1}-2q\O_{2,1,1}
			\\
			&+\O_{3,3,2}
			&
			-2q\O_{2,2}&+q\O_{2,2,1}+q\O_{3}-2q\O_{3,1}
			\\
			&
			&
			+q\O_{3,1,1}&+q\O_{3,2}-q\O_{3,2,1}+\O_{3,3,3}
			\\
			\O_{2,1} \bullet \O_{3,2,2} & = q\O_{2,1,1}+q\O_{2,2}-q\O_{2,2,1}
			&
			\O_{2,1} \bullet \O_{3,3} & = q\O_{2,1}
			\\
			+q\O_{3,1}&-q\O_{3,1,1}-q\O_{3,2}+q\O_{3,2,1}
			&
			&
			\\
			\O_{2,1} \bullet \O_{3,3,1} & = q\O_{2,1,1}+q\O_{2,2}-q\O_{2,2,1}
			&
			\O_{2,1} \bullet \O_{3,3,2} & = q\O_{2,2,1}+q\O_{3,1,1}+q\O_{3,2}
			\\
			+q\O_{3,1}&-q\O_{3,1,1}-q\O_{3,2}+q\O_{3,2,1}
			&
			&-2q\O_{3,2,1}
			\\
			\O_{2,1} \bullet \O_{3,3,3} & = q\O_{3,2,1}
			&
			\O_{2,1,1} \bullet \O_{2,1,1} & = q\O_{1,1}-q\O_{2,1}+\O_{3,3,2}
			\\
			\O_{2,1,1} \bullet \O_{2,2} & = q\O_{2}-q\O_{2,1}+\O_{3,3,2}
			&
			\O_{2,1,1} \bullet \O_{2,2,1} & = q\O_{2,1}-q\O_{2,2}+\O_{3,3,3}
			\\
			\O_{2,1,1} \bullet \O_{2,2,2} & = q\O_{2,2}
			&
			\O_{2,1,1} \bullet \O_{3} & = q\O_{1}
			\\
			\O_{2,1,1} \bullet \O_{3,1} & = q\O_{1,1}+q\O_{2}-q\O_{2,1}
			&
			\O_{2,1,1} \bullet \O_{3,1,1} & = q\O_{1,1,1}+q\O_{2,1}-q\O_{2,1,1}
			\\
			\O_{2,1,1} \bullet \O_{3,2} & = q\O_{2,1}+q\O_{3}-q\O_{3,1}
			&
			\O_{2,1,1} \bullet \O_{3,2,1} & = q\O_{2,1,1}+q\O_{2,2}-q\O_{2,2,1}
			\\
			&
			&
			+q\O_{3,1}&-q\O_{3,1,1}-q\O_{3,2}+q\O_{3,2,1}
			\\
			\O_{2,1,1} \bullet \O_{3,2,2} & = q\O_{2,2,1}+q\O_{3,2}-q\O_{3,2,1}
			&
			\O_{2,1,1} \bullet \O_{3,3} & = q\O_{3,1}
			\\
			\O_{2,1,1} \bullet \O_{3,3,1} & = q\O_{3,1,1}+q\O_{3,2}-q\O_{3,2,1}
			&
			\O_{2,1,1} \bullet \O_{3,3,2} & = q\O_{3,2,1}+q\O_{3,3}-q\O_{3,3,1}
			\\
			\O_{2,1,1} \bullet \O_{3,3,3} & = q\O_{3,3,1}
			&
			\O_{2,2} \bullet \O_{2,2} & = q\O_{1,1}+q\O_{2}-q\O_{2,1}
			\\
			\O_{2,2} \bullet \O_{2,2,1} & = q\O_{2,1}+q\O_{3}-q\O_{3,1}
			&
			\O_{2,2} \bullet \O_{2,2,2} & = q\O_{3,1}
			\\
			\O_{2,2} \bullet \O_{3} & = \O_{3,2,2}
			&
			\O_{2,2} \bullet \O_{3,1} & = q\O_{1,1}-q\O_{2,1}+\O_{3,3,2}
			\\
			\O_{2,2} \bullet \O_{3,1,1} & = q\O_{2,1}-q\O_{2,2}+\O_{3,3,3}
			&
			\O_{2,2} \bullet \O_{3,2} & = q\O_{1,1,1}+q\O_{2,1}-q\O_{2,1,1}
			\\
			\O_{2,2} \bullet \O_{3,2,1} & = q\O_{2,1,1}+q\O_{2,2}-q\O_{2,2,1}
			&
			\O_{2,2} \bullet \O_{3,2,2} & = q\O_{3,1,1}+q\O_{3,2}-q\O_{3,2,1}
			\\
			+q\O_{3,1}&-q\O_{3,1,1}-q\O_{3,2}+q\O_{3,2,1}
			&
			&
			\\
			\O_{2,2} \bullet \O_{3,3} & = q\O_{2,1,1}
			&
			\O_{2,2} \bullet \O_{3,3,1} & = q\O_{2,2,1}+q\O_{3,1,1}-q\O_{3,2,1}
			\\
			\O_{2,2} \bullet \O_{3,3,2} & = q^2 -q^2\O_{1}+q\O_{3,2,1}
			&
			\O_{2,2} \bullet \O_{3,3,3} & = q^2\O_{1}
			\\
			\O_{2,2,1} \bullet \O_{2,2,1} & = q\O_{2,2}+q\O_{3,1}-q\O_{3,2}
			&
			\O_{2,2,1} \bullet \O_{2,2,2} & = q\O_{3,2}
			\\
			\O_{2,2,1} \bullet \O_{3} & = q\O_{1,1}
			&
			\O_{2,2,1} \bullet \O_{3,1} & = q\O_{1,1,1}+q\O_{2,1}-q\O_{2,1,1}
			\\
			\O_{2,2,1} \bullet \O_{3,1,1} & = q\O_{2,1,1}+q\O_{2,2}-q\O_{2,2,1}
			&
			\O_{2,2,1} \bullet \O_{3,2} & = q\O_{2,1,1}+q\O_{3,1}-q\O_{3,1,1}
			\\
			\O_{2,2,1} \bullet \O_{3,2,1} & = q\O_{2,2,1}+q\O_{3,1,1}+q\O_{3,2}
			&
			\O_{2,2,1} \bullet \O_{3,2,2} & = q\O_{3,2,1}+q\O_{3,3}-q\O_{3,3,1}
			\\
			&-2q\O_{3,2,1}
			&
			&
			\\
			\O_{2,2,1} \bullet \O_{3,3} & = q\O_{3,1,1}
			&
			\O_{2,2,1} \bullet \O_{3,3,1} & = q^2 -q^2\O_{1}+q\O_{3,2,1}
			\\
			\O_{2,2,1} \bullet \O_{3,3,2} & = q^2\O_{1}-q^2\O_{2}+q\O_{3,3,1}
			&
			\O_{2,2,1} \bullet \O_{3,3,3} & = q^2\O_{2}
			\\
			\O_{2,2,2} \bullet \O_{2,2,2} & = q\O_{3,3}
			&
			\O_{2,2,2} \bullet \O_{3} & = q\O_{1,1,1}
			\\
			\O_{2,2,2} \bullet \O_{3,1} & = q\O_{2,1,1}
			&
			\O_{2,2,2} \bullet \O_{3,1,1} & = q\O_{2,2,1}
			\\
			\O_{2,2,2} \bullet \O_{3,2} & = q\O_{3,1,1}
			&
			\O_{2,2,2} \bullet \O_{3,2,1} & = q\O_{3,2,1}
			\\
			\O_{2,2,2} \bullet \O_{3,2,2} & = q\O_{3,3,1}
			&
			\O_{2,2,2} \bullet \O_{3,3} & = q^2 
			\\
			\O_{2,2,2} \bullet \O_{3,3,1} & = q^2\O_{1}
			&
			\O_{2,2,2} \bullet \O_{3,3,2} & = q^2\O_{2}
			\\
			\O_{2,2,2} \bullet \O_{3,3,3} & = q^2\O_{3}
			&
			\O_{3} \bullet \O_{3} & = \O_{3,3}
			\\
			\O_{3} \bullet \O_{3,1} & = \O_{3,3,1}
			&
			\O_{3} \bullet \O_{3,1,1} & = q\O_{2}
			\\
			\O_{3} \bullet \O_{3,2} & = \O_{3,3,2}
			&
			\O_{3} \bullet \O_{3,2,1} & = q\O_{2,1}
			\\
			\O_{3} \bullet \O_{3,2,2} & = q\O_{2,1,1}
			&
			\O_{3} \bullet \O_{3,3} & = \O_{3,3,3}
			\\
			\O_{3} \bullet \O_{3,3,1} & = q\O_{2,2}
			&
			\O_{3} \bullet \O_{3,3,2} & = q\O_{2,2,1}
			\\
			\O_{3} \bullet \O_{3,3,3} & = q\O_{2,2,2}
			&
			\O_{3,1} \bullet \O_{3,1} & = q\O_{2}-q\O_{2,1}+\O_{3,3,2}
			\\
			\O_{3,1} \bullet \O_{3,1,1} & = q\O_{2,1}+q\O_{3}-q\O_{3,1}
			&
			\O_{3,1} \bullet \O_{3,2} & = q\O_{2,1}-q\O_{2,2}+\O_{3,3,3}
			\\
			\O_{3,1} \bullet \O_{3,2,1} & = q\O_{2,1,1}+q\O_{2,2}-q\O_{2,2,1}
			&
			\O_{3,1} \bullet \O_{3,2,2} & = q\O_{2,2,1}+q\O_{3,1,1}-q\O_{3,2,1}
			\\
			+q\O_{3,1}&-q\O_{3,1,1}-q\O_{3,2}+q\O_{3,2,1}
			&
			&
			\\
			\O_{3,1} \bullet \O_{3,3} & = q\O_{2,2}
			&
			\O_{3,1} \bullet \O_{3,3,1} & = q\O_{2,2,1}+q\O_{3,2}-q\O_{3,2,1}
			\\
			\O_{3,1} \bullet \O_{3,3,2} & = q\O_{2,2,2}+q\O_{3,2,1}-q\O_{3,2,2}
			&
			\O_{3,1} \bullet \O_{3,3,3} & = q\O_{3,2,2}
			\\
			\O_{3,1,1} \bullet \O_{3,1,1} & = q\O_{2,1,1}+q\O_{3,1}-q\O_{3,1,1}
			&
			\O_{3,1,1} \bullet \O_{3,2} & = q\O_{2,2}+q\O_{3,1}-q\O_{3,2}
			\\
			\O_{3,1,1} \bullet \O_{3,2,1} & = q\O_{2,2,1}+q\O_{3,1,1}+q\O_{3,2}
			&
			\O_{3,1,1} \bullet \O_{3,2,2} & = q\O_{2,2,2}+q\O_{3,2,1}-q\O_{3,2,2}
			\\
			&-2q\O_{3,2,1}
			&
			&
			\\
			\O_{3,1,1} \bullet \O_{3,3} & = q\O_{3,2}
			&
			\O_{3,1,1} \bullet \O_{3,3,1} & = q\O_{3,2,1}+q\O_{3,3}-q\O_{3,3,1}
			\\
			\O_{3,1,1} \bullet \O_{3,3,2} & = q\O_{3,2,2}+q\O_{3,3,1}-q\O_{3,3,2}
			&
			\O_{3,1,1} \bullet \O_{3,3,3} & = q\O_{3,3,2}
			\\
			\O_{3,2} \bullet \O_{3,2} & = q\O_{2,1,1}+q\O_{2,2}-q\O_{2,2,1}
			&
			\O_{3,2} \bullet \O_{3,2,1} & = q\O_{2,2,1}+q\O_{3,1,1}+q\O_{3,2}
			\\
			&
			&
			&-2q\O_{3,2,1}
			\\
			\O_{3,2} \bullet \O_{3,2,2} & = q^2 -q^2\O_{1}+q\O_{3,2,1}
			&
			\O_{3,2} \bullet \O_{3,3} & = q\O_{2,2,1}
			\\
			\O_{3,2} \bullet \O_{3,3,1} & = q\O_{2,2,2}+q\O_{3,2,1}-q\O_{3,2,2}
			&
			\O_{3,2} \bullet \O_{3,3,2} & = q^2\O_{1}-q^2\O_{1,1}+q\O_{3,2,2}
			\\
			\O_{3,2} \bullet \O_{3,3,3} & = q^2\O_{1,1}
			&
			\O_{3,2,1} \bullet \O_{3,2,1} & = q^2 -2q^2\O_{1}+q^2\O_{1,1}
			\\
			&
			&
			+q^2\O_{2}&-q^2\O_{2,1}+q\O_{2,2,2}+2q\O_{3,2,1}
			\\
			&
			&
			-2q\O_{3,2,2}&+q\O_{3,3}-2q\O_{3,3,1}+q\O_{3,3,2}
			\\
			\O_{3,2,1} \bullet \O_{3,2,2} & = q^2\O_{1}-q^2\O_{1,1}-q^2\O_{2}
			&
			\O_{3,2,1} \bullet \O_{3,3} & = q\O_{3,2,1}
			\\
			+q^2\O_{2,1}&+q\O_{3,2,2}+q\O_{3,3,1}-q\O_{3,3,2}
			&
			&
			\\
			\O_{3,2,1} \bullet \O_{3,3,1} & = q^2\O_{1}-q^2\O_{1,1}-q^2\O_{2}
			&
			\O_{3,2,1} \bullet \O_{3,3,2} & = q^2\O_{1,1}+q^2\O_{2}-2q^2\O_{2,1}
			\\
			+q^2\O_{2,1}&+q\O_{3,2,2}+q\O_{3,3,1}-q\O_{3,3,2}
			&
			&+q\O_{3,3,2}
			\\
			\O_{3,2,1} \bullet \O_{3,3,3} & = q^2\O_{2,1}
			&
			\O_{3,2,2} \bullet \O_{3,2,2} & = q^2\O_{2}-q^2\O_{2,1}+q\O_{3,3,2}
			\\
			\O_{3,2,2} \bullet \O_{3,3} & = q^2\O_{1}
			&
			\O_{3,2,2} \bullet \O_{3,3,1} & = q^2\O_{1,1}+q^2\O_{2}-q^2\O_{2,1}
			\\
			\O_{3,2,2} \bullet \O_{3,3,2} & = q^2\O_{2,1}+q^2\O_{3}-q^2\O_{3,1}
			&
			\O_{3,2,2} \bullet \O_{3,3,3} & = q^2\O_{3,1}
			\\
			\O_{3,3} \bullet \O_{3,3} & = q\O_{2,2,2}
			&
			\O_{3,3} \bullet \O_{3,3,1} & = q\O_{3,2,2}
			\\
			\O_{3,3} \bullet \O_{3,3,2} & = q^2\O_{1,1}
			&
			\O_{3,3} \bullet \O_{3,3,3} & = q^2\O_{1,1,1}
			\\
			\O_{3,3,1} \bullet \O_{3,3,1} & = q^2\O_{1,1}-q^2\O_{2,1}+q\O_{3,3,2}
			&
			\O_{3,3,1} \bullet \O_{3,3,2} & = q^2\O_{1,1,1}+q^2\O_{2,1}-q^2\O_{2,1,1}
			\\
			\O_{3,3,1} \bullet \O_{3,3,3} & = q^2\O_{2,1,1}
			&
			\O_{3,3,2} \bullet \O_{3,3,2} & = q^2\O_{2,1,1}+q^2\O_{3,1}-q^2\O_{3,1,1}
			\\
			\O_{3,3,2} \bullet \O_{3,3,3} & = q^2\O_{3,1,1}
			&
			\O_{3,3,3} \bullet \O_{3,3,3} & = q^3 
		\end{align*}
	}

		\bibliographystyle{alphnum}
		\bibliography{ref2}

	\end{document}